\documentclass[11pt,reqno]{article}

\usepackage{anysize}
\marginsize{3.5cm}{3.5cm}{3cm}{3cm}

\usepackage{authblk}

\usepackage[utf8]{inputenc}

\usepackage{amsmath}
\usepackage{amsfonts,amssymb}
\usepackage{enumerate}
\usepackage{mathrsfs}
\usepackage{amsthm}
\usepackage{xcolor}
\usepackage[normalem]{ulem}
\usepackage{enumitem}
\usepackage{bm}

\usepackage{tikz}
\usepackage{tikz-cd}
\usepackage{pgfplots}

\DeclareMathOperator{\cnx}{div}
\DeclareMathOperator{\cn}{div}

\DeclareMathOperator{\di}{d}
\DeclareMathOperator{\diff}{d}
\DeclareSymbolFont{pletters}{OT1}{cmr}{m}{sl}
\DeclareMathSymbol{s}{\mathalpha}{pletters}{`s}
\def\ba{\begin{align}}
\def\ea{\end{align}}

\def\bad{\begin{aligned}}
\def\ead{\end{aligned}}

\def\be{\begin{equation}}
\def\ee{\end{equation}}

\def\dsigma{\diff \! \sigma}
\def\dtheta{\diff \! \theta}
\def\dt{\diff \! t}
\def\dx{\diff \! x}
\def\du{\diff \! u}

\def\dy{\diff \! y}

\def\dtheta{\diff \! \theta}

\def\pat{\partial_{\theta}}
\def\patt{\partial_{\theta \theta}}
\def\pas{\partial_{s}}
\def\pass{\partial_{ss}}
\def\passs{\partial_{sss}}
\def\past{\partial_{s \theta}}
\def\passt{\partial_{s s \theta}}
\def\passtt{\partial_{s s \theta \theta}}
\def\passst{\partial_{s s s \theta}}
\def\passstt{\partial_{s s s \theta \theta}}
\def\pastt{\partial_{s \theta \theta}}

\def\eps{\varepsilon}

\def\le{\leq}

\def\xC{\mathbb{C}}
\def\xN{\mathbb{N}}
\def\xR{\mathbb{R}}

\def\xZ{\mathbb{Z}}

\def\dd{\mathrm{d}}

\numberwithin{equation}{section}

\makeatletter
\renewenvironment{proof}[1][\proofname]{\par
\pushQED{\qed}%
\normalfont \topsep1\p@\@plus3\p@\relax
\trivlist
\item[\hskip\labelsep
    \bfseries
#1\@addpunct{\slshape{.}}]\ignorespaces
}{%
\popQED\endtrivlist\@endpefalse
\vspace{10pt plus 3pt}
}
\makeatother

\newtheoremstyle{mystyle}
  {.4cm} % Space above
  {.4cm} % Space below
  {\sl} % Body font
  {} % Indent amount
  {\bfseries} % Theorem head font
  {.} % Punctuation after theorem head
  {.5em} % Space after theorem head
  {} % Theorem head spec (can be left empty, meaning `normal')

\theoremstyle{mystyle}

\newtheorem{theo}{Theorem}[section]
\newtheorem*{theo*}{Theorem}
\newtheorem{prop}[theo]{Proposition}
\newtheorem{lemm}[theo]{Lemma}
\newtheorem{coro}[theo]{Corollary}
\newtheorem{defi}[theo]{Definition}
\newtheorem{rema}[theo]{Remark}

\usepackage[]{hyperref}
\hypersetup{
    colorlinks=true,       % false: boxed links; true: colored links
    linkcolor=red,          % color of internal links
    citecolor=blue,        % color of links to bibliography
    filecolor=magenta,      % color of file links
    urlcolor=cyan           % color of external links
}
\hypersetup{linkbordercolor=black}
\usepackage{titlesec}
\titleformat{\section}
  {\centering\Large\bf}{\thesection}{1em}{}
\titleformat{\subsection}
  {\centering\large\bf}{\thesubsection}{1em}{}

\def\lm{\lambda}
\def\vp{\varphi}

\newcommand{\pare}[1]{\left( #1 \right)}
\newcommand{\norm}[1]{\left\| #1 \right\|}
\newcommand{\av}[1]{\left| #1 \right|}
\newcommand{\bra}[1]{\left[ #1 \right]}
\newcommand{\set}[1]{\left\{ #1 \right\}}

\renewcommand{\t}[1]{\text{#1}}

\def\pa{\partial} 
 
\def\ga{\gamma} 
 
\def\al{\alpha} 
 
\def\q{\quad} 
 
\newcommand{\w}[1]{\overline{#1}}

\def\FF{\mathcal F} 
\def\veps{\varepsilon}

\def\tGa{\widetilde{\Gamma}}
\def\tO{\widetilde{\Omega}}

\newcommand{\wt}[1]{\widetilde{#1}}

\def\qq{\qquad}
\def\q{\quad}

\allowdisplaybreaks

\usepackage{marginnote} 

\title{
Traveling wave solution for a coupled incompressible Darcy's free boundary problem with surface tension}

\author[1]{Thomas Alazard}
\author[1]{Martina Magliocca}
\author[3]{Nicolas Meunier}

\affil[1]{Universit\'e Paris-Saclay, ENS Paris-Saclay, CNRS,
	Centre Borelli UMR9010, avenue des Sciences, F-91190 Gif-sur-Yvette, France, (thomas.alazard@ens-paris-saclay.fr) (martina.magliocca@gmail.com)}
\affil[2]{LaMME, UMR 8071 CNRS, Universit\'e \'Evry Val d'Essonne, France. (nicolas.meunier@univ-evry.fr)}

\date{\empty}

\begin{document}

\maketitle

\begin{center}
\today 

\vspace*{1cm}

{\large \textbf{Abstract}}
\end{center}

We study  an incompressible Darcy's free boundary problem, recently introduced in \cite{LMVC}. Our goal is to prove the existence of non-trivial traveling wave solutions and thus validate the interest of this model  to describe cell motility. The model equations include a convection diffusion equation for the polarity marker concentration and the incompressible Darcy's equation. The mathematical novelty of this problem is the nonlinear destabilizing term in the boundary condition that describes the active character of the cell cytoskeleton. We first study the linear stability of this problem and we show that, above a well precise threshold,  the disk becomes linearly unstable.
By using two different approaches we prove existence of traveling wave solutions, which describes persistent motion of a biological cell. One is explicit, by construction. The other is established implicitly, as the one bifurcating from stationary solution.

{\hypersetup{linkcolor=black}
	\tableofcontents
}

\medskip

\textbf{Acknowledgement:}
TA and NM want to thank D. Smets for helpful discussion.\\



\section{Introduction} 

In this paper we study the existence of traveling waves for a two-dimensional free boundary problem modeling the dynamics of a living cell. We consider a coupled incompressible Darcy's problem whose novelty lies in the coupling in the boundary term of the Hele-Shaw model, with surface tension, with a PDE stated on the free evolving domain, see Section \ref{sec:prop} for details. 

More precisely, we consider a smooth open bounded set $\Omega_0$ in $\xR^2$, representing the portion of the space occupied by an incompressible fluid at time $t=0$, and a given smooth non-negative function $c_0=c_0(x,y)$ defined on $\Omega_0$, which represents the concentration of a solute at time $t=0$. Then, we seek a family of open sets $\Omega(t)$ of $\xR ^2$ with boundary $\Gamma(t)=\partial \Omega(t)$ and concentration functions $c(t,x,y)$ defined on $\Omega(t)$ solutions of
\begin{subequations}\label{eq:pb}
\begin{align}
-\Delta P  &= 0  &&  \mbox{ in } \Omega(t), \label{eq:pression}\\
P &= \gamma \kappa +\chi f(c) && \mbox{ on } \Gamma(t),\label{eq:pression_bord} \\
 V_n=\bm{V}_{\Gamma(t)}\cdot \bm{n} &= - \nabla P  \cdot \bm{n}    && \mbox{ on } \Gamma(t),\label{eq:cin}\\
 \partial _t c-D \Delta c  &= (1-a) \nabla P \cdot\nabla c  &&  \mbox{ in } \Omega(t), \label{eq:marqueur}\\
 D\nabla c \cdot \bm{n}&=ac\nabla P \cdot \bm{n}&& \mbox{ on } \Gamma(t),\label{eq:marqueur_bord} \\
c(0,x,y)&=c_0(x,y) &&\mbox{ in }\Omega_0,\\
 \Omega(0)&=\Omega_0,&&
\end{align}
\end{subequations}
where $\kappa$ is the curvature (positive for a circle) of the evolving free-boundary $\Gamma(t)$, $\bm{n}$ is the outward pointing unit normal on $\Gamma(t)$,  the surface tension $\gamma \ge 0$ is a given constant, $\chi \ge 0$, $D>0$ is the diffusion coefficient, $a \in [0,1]$, and $f:\xR \to \xR$ is a given function.  We denote by $\bm{V}_{\Gamma(t)}$ (resp. $V_n$) the velocity (resp. normal velocity) of the moving free-boundary $\Gamma(t)$.

The first two equations \eqref{eq:pression} -- \eqref{eq:pression_bord} determine the pressure according to Darcy's law, which states that the velocity of the fluid is $\bm u =-\nabla P$, incompressibility of the fluid, that is $\cn \bm u =0$, and Laplace condition; the third one \eqref{eq:cin} is a kinematic condition which states that the interface $\Gamma(t)$ is transported by the velocity of the fluid $-\nabla P$. In \eqref{eq:marqueur_bord}, a zero solute flux (of $c$) on the moving boundary $\Gamma(t)$ is imposed. The coupling with the unknown $c$ occurs in the boundary term \eqref{eq:pression_bord}. The time evolution of $ c$ follows the advection-diffusion equation \eqref{eq:marqueur}. We refer to Section \ref{sec:prop} for a biological justification of the model.

Note that in \eqref{eq:pb}, we have formally conservation of molecular content, that is for all time $t$ it holds that
\begin{equation}\label{eq:M}
	M:=\int_{\Omega(t)} c(t,x,y)\dx \dy = \int_{\Omega_0} c_0(x,y)\dx \dy.
\end{equation}

For simplicity, here we will assume that
\begin{equation*}%\label{eq:diff}
	D=1.
\end{equation*}

We make the following assumptions on $f$:
\begin{align}
 f\in C ^1(\xR^+),\q f \textrm{ is increasing, } \q 
 f(0)=0,\q \lim\limits_{x\to +\infty} f(x) =L\label{A}.%\\
\end{align}
\begin{rema}
A prototype example, \cite{LMVC}, of a function satisfying the previous assumptions is 
\begin{equation}\label{eq:prototype_example}
	f(x)= \frac{ Lx}{\alpha + x}, 
	%\q\t{with}\q L<\frac{4}{a\chi}.
\end{equation}
where $L>0$ is the maximal pulling force and $\alpha >0$ is a satuartion parameter. 
\end{rema}

Let us briefly comment the existing  literature concerning \eqref{eq:pb}. Moving interface problems have raised many interesting and challenging mathematical issues. A well known example is the Stefan problem which describes the dynamics of the boundary between ice and water. In the biophysical community, we find a large number of free boundary models to describe tumor and tissue growth, cell motility and other phenomena. Most of them are formulated through a fluid approach with surface tension. Some tumor growth models (e.g. \cite{Friedman_2004, Friedman_Hu_2006, Friedman_Reitich_2001}) resemble our model \eqref{eq:pb}. However, there is an important difference: tumor growth naturally involves expanding domain while we consider here incompressible solutions. 
In the context of the motility of eukaryotic cells on substrates, various free boundary problems have been derived and studied, see \cite{Berlyand4, Berlyand3, Berlyand2, Berlyand1,Berlyand_2020,Berlyand_2021,Mizuhara}. In the 1D setting, Keller-Segel system with free boundaries as a model for contraction driven motility were introduced and studied in \cite{Recho_2013, Recho_2013-2,Recho_2015,Recho_2018}. The models presented in \cite{Berlyand_2020,Berlyand_2021} show some similarities with our model but the coupling between the bulk equation and the boundary is different.  
The existence of traveling wave solutions for these models is proved in \cite{Berlyand4,Berlyand2,Berlyand_2020, Berlyand_2021}. In the context of sharp interface limit some models for cell motility were studied in \cite{CMM_2020,CMM_2022}.  We refer to \cite{Aranson, Levine} for a review.

\subsection{On traveling waves}
A remarkable feature of cell motility is the appearance of sustained movement in a given direction without external cue, see \cite{Barnhart, Aranson, Keren}.
This phenomenon, known as spontaneous polarization, see \cite{CHMV,EMV,PLOS} e.g., is mathematically described by the existence of traveling wave solutions and is the main subject of this article.

Traveling wave solutions of \eqref{eq:pb} correspond to a fixed shape domain  moving by translation with constant velocity $V\in\xR$ in a given direction $\bold{u}$, that is 
\begin{equation}\label{eq:tw00}
	\Omega_{\rm tw}(t)=\tO + tV  \bold{u},
\end{equation}
for some speed $V$ and  direction of motion $\bold{u}$.

Note that the problem is isotropic, so we can assume without loss of generality that $\bold{u}=\bold{e}_x= (1,0)$ and $V>0$. In that case, the normal boundary velocity of $\Omega(t)$ given by \eqref{eq:tw00}
satisfies $V_n=V \, \bold{e}_x  \cdot \bold{n}$.

Using the traveling wave ansatz 
\begin{equation*} 
	c=c(x- Vt,y), \qquad P=P(x-Vt,y), \qquad \Omega(t)=\tO +(Vt,0),
\end{equation*} a traveling-wave solution of \eqref{eq:pb} is defined as following, see Section  \ref{sec:graph} for details.
\begin{defi}\label{def:TW_1}
A traveling wave solution of \eqref{eq:pb} is given by a domain $\tO \subset \xR^2$ with $\mathcal{C}^{2,1}$  boundary $\tGa$, a positive real number $V$ and two $C^\infty
$ functions $P,\,c$   defined on $\tO$ satisfying  
\begin{subequations}\label{eq:pb-tw}
\begin{align}
 -\Delta P  &= 0  &&  \mbox{ in } \tO, \label{eq:pression_TW_1}\\
P &= \gamma \kappa   +\chi f(c)&& \mbox{ on } \tGa,\label{eq:pression_bord_TW_1} \\
    - \nabla P \cdot \bm{n} &=(V,0) \cdot\bm{n}    && \mbox{ on } \tGa,\label{eq:cin_TW_1}\\
 \cn \left((V,0)c+\left( 1-a\right)c \nabla P +\nabla c\right)&= 0  &&  \mbox{ in } \tO, \label{eq:marqueur_TW_1}\\
ac(V,0)\cdot \bm{n} +\nabla c \cdot \bm{n}&= 0&& \mbox{ on }\tGa .\label{eq:marqueur_bord_TW_1} 
\end{align}
\end{subequations}
\end{defi}

We begin by observing that the pressure and the concentration have simple forms.
\begin{prop} \label{prop:cP}
For traveling wave solutions of \eqref{eq:pb}, the functions $P$ and $c$ have the form:
\[
P(x,y) = p_1- Vx, \quad p_1\in \xR, \qquad \textrm{ and } \qquad c(x,y)= \frac{M}{\int_{\tO}e^{-aVx'}\dx' \dy'  }e^{-aVx},
\] 
with $M\ge 0$.
\end{prop}

\begin{rema}[On traveling waves and Jordan curves]\label{rmk:equiv-def}
The equations in \eqref{eq:pb-tw} mandate that the entire fluid bulk flows at a uniform speed, that is $\nabla P=-(V,0)$ in $\tO$. 
Traveling wave solutions of \eqref{eq:pb} can be regarded as a family of closed Jordan curves $\Gamma_{\rm tw}(t)$, which represents the boundary of a domain $\Omega_{\rm tw} \subset \xR^2$ at time $t$, traveling with constant shape $\tGa=\partial \tO$ and velocity $V\bold{e}_x \in \xR^2$, i.e.,
\begin{equation}\label{def:TW_Jordan}
\Gamma_{\rm tw}(t) =\tGa + tV\bold{e}_x ,
\end{equation}
for $t\ge  0$, where $\tGa$ is a Jordan curve that is positively oriented (counterclockwise direction) and is governed by the curvature equation \eqref{eq:pression_bord_TW_1}. 
In particular, recalling the expressions of $P$ and $c$ in Proposition \ref{prop:cP}, the boundary condition \eqref{eq:pression_bord_TW_1} takes the form 
	\begin{equation}
		\label{eq:TW_shape}
		\gamma \kappa(x)= p_1 - V x - \chi f\left(  \frac{M}{\int_{\Omega_{0}}e^{-aVx'}\dx' \dy'  }e^{-aVx}\right) \qquad \textrm{  on }   \tGa.
	\end{equation}
\end{rema}

\begin{rema}[On $V$ and $\tO$]
	In this problem, the set $\tO$ and  the speed $V$ must be found together and they depend on $\chi,\,p_1,\,a,\,\gamma$ and $f$. The parameter $p_1$ can be seen as a Lagrange multiplier for the volume of $\tO$, 
	but  the problem is invariant by translation. Any translation of $\tO$ leads to a solution of \eqref{eq:TW_shape}, with the same $V$ but a different value of $p_1$. 
\end{rema}

\subsection{Main results}

We establish the following properties of the model \eqref{eq:pb} in Section \ref{sec:prop}:
\begin{align}
\av{\Omega(t)}&=\av{\Omega_0}\q \forall t\ge0,\label{eq:area}\\
M(t)&=M(0)=M \label{eq:mass_marker}\\
&\bm{u}_{\bm{\mathcal C}}(t) =-\frac{\chi}{|\Omega_0| } \int_{\Gamma(t)}  f( c)  n \dsigma,\label{eq:centre_mass}
\end{align}
where $\bm{u}_{\bm{\mathcal C}}(t)$ is the velocity of the center of mass.

Let
\[
R_0=\sqrt{\frac{|\Omega_0|}{\pi}},
\]
that is 
$|B_{R_0}|=|\Omega_0|$, where $B_{R_0}$ is the disk with radius $R_0$. Then,  problem \eqref{eq:pb} with \eqref{eq:area} and \eqref{eq:mass_marker}  possesses a unique radially symmetric solution with both $P$ and $c$ being constant. 
\begin{lemm}\label{lem:stat}
Assume that $|B_{R_0}|=|\Omega_0|$. The problem \eqref{eq:pb} with \eqref{eq:area} and \eqref{eq:mass_marker} admits a unique radially symmetric solution $(c,P)=(\tilde{c},\wt{P})$ which has the form 
\begin{equation}\label{eq:stat}
\pare{\tilde{c}, \wt{P}} := \pare{\frac{M}{\pi R_0^2}, \frac{\gamma }{R_0}+\chi f\pare{\frac{M}{\pi R_0^2}}}=\pare{\tilde{c}, \frac{\gamma }{R_0}+\chi f(\tilde{c})}.
\end{equation}
\end{lemm}

We begin performing a linear stability analysis around the  radially symmetric solution $(\tilde{c},\wt{P})$. 
Define 
\begin{equation}\label{def:chi_star}
	\chi^* := \frac{\pi R_0^2}{a  M f'\left( \frac{ M}{\pi R_0^2} \right)}= \frac{1}{a  \tilde{c} f'(\tilde{c})}.
\end{equation}
\begin{prop}\label{prop:lin_stab}
	Assume that $\chi \in ( \chi^*, \infty)$, then the linearized problem around $(\tilde{c},\wt{P})$ associated with \eqref{eq:pb} has at least one eigenvalue with positive real part.
	On the contrary, if $\chi \in [ 0,\chi^*]$, then all the eigenvalues of the linearized problem around $(\tilde{c},\wt{P})$ associated with \eqref{eq:pb} have non-positive real parts.
\end{prop}

Then, we wonder when equations \eqref{def:TW_Jordan} -- \eqref{eq:TW_shape} admit traveling waves. Our first partially result states that no traveling wave exists under a well precise condition on the parameters.

\begin{prop}\label{prop:non-ex-TW}
	Assume that the function $f$ in \eqref{eq:pb} verifies \eqref{A} and
	\begin{equation}\label{def:no_TW}
	a \chi s f'(s)< 1\q\forall s\in\xR^+.
	\end{equation} 
There does not exist any traveling wave solution to \eqref{eq:pb} in the sense of Definition \ref{def:TW_1}.
\end{prop}

It is the goal of this paper to prove the existence of non-trivial traveling wave solutions of \eqref{eq:pb} and thus validate the interest of this model  to describe cell motility.  For this purpose we use two different approaches: constructive by fixing the value of $p_1$ or based on a bifurcation argument by fixing the area, that is $R_0$.

\begin{theo}[Explicit by construction]\label{thm:TW_intro}
	Assume that $f$ satisfies assumptions \eqref{A}. For all $a\in (0,1]$, $\gamma>0$ and all $p_1>\chi L$, there exists a one parameter family of traveling wave  solutions $(\tO_{\chi }^{p_1},V_{\chi }^{p_1})$  of \eqref{eq:pb}, parametrized by $\chi \in ( \chi^*, \infty)$, such that $P$ is of the form $P=p_1-Vx$.
\end{theo}	
\noindent With this constructive method we obtain conditions under which the set $\tO_{\chi }^{p_1}$ is convex, see Theorem \ref{thm:TW_intro_2} in Section \ref{sec:construction} for more details.
 
\medskip
 
While the proof of this Theorem is constructive, it does not clearly identify what happens for a cell of fixed volume when the value of $ \chi$ increases. The next result solves this problem.

\begin{theo}[Implicit by bifurcation]\label{thm:TW:bif}
	Assume that $f$ satisfies assumptions \eqref{A}. For all $a\in (0,1]$, $\gamma>0$ and $R_0>0$, there exists a one parameter family of traveling wave  solutions $(\tO_{\chi }^{R_0},V_{\chi }^{R_0})$  of \eqref{eq:pb}, parametrized by $\chi \in ( \chi^*, \infty)$ such that $|\tO_{\chi }^{R_0} |=\pi R_0^2$.
\end{theo}
\noindent With this implicit method we can characterize the nature of the bifurcation, either pitchfork or saddle node, see Theorem \ref{thm:TW:bif-CR} in Section \ref{sec:thm:TW:bif-CR} for more details.

\begin{rema}
	Note that condition \eqref{def:no_TW}
	\[
	a\chi s f'(s)<1\q\forall s\in \xR^+,
	\]
	in Proposition \ref{prop:non-ex-TW}, and condition $\chi \in ( \chi^*, \infty)$, that is
	\[
	a\chi \tilde{c} f'(\tilde{c})\ge1
	\]
	in Proposition \ref{prop:lin_stab}, Theorems \ref{thm:TW_intro} and \ref{thm:TW:bif} cannot hold at the same time.
\end{rema}

\medskip 
 
This work is organized as follows. We give some biological justification and we present some properties of the problem \eqref{eq:pb} in Section \ref{sec:prop}. In Section \ref{sec:lin_stab}, we study the linear stability of the system and we prove Proposition \ref{prop:lin_stab}. Sections \ref{sec:low_bound} and \ref{sec:construction} contain, respectively, the proofs of  Proposition \ref{prop:non-ex-TW} and  Theorem \ref{thm:TW_intro}. Theorem \ref{thm:TW:bif} is proved in Section \ref{sec:thm:TW:bif} and in Section \ref{sec:thm:TW:bif-CR} we discuss the nature of the bifurcation. 
Finally, we give some conclusions. We also provide the reader with three appendices, collecting some useful facts and bifurcation results.

\section{Biological justification and first properties of the problem \eqref{eq:pb}}\label{sec:prop}

In this section we justify, from a biological point of view, the interest of  \eqref{eq:pb} and we derive some properties of the coupled free boundary problem \eqref{eq:pb}. We proceed formally (considering smooth enough solutions) and deduce the conservation of the area, the marker content as well as a law for the velocity of the center of mass. We also discuss the existence and uniqueness of a radially symmetric solution to \eqref{eq:pb}.  

\subsection{Biological justification}
Cell motility at the single cell level is a prime example of self-propulsion and one of the simplest example of active system. Recently, many free boundary models have been proposed to describe cell motility (see \cite{Aranson} for a review). The model \eqref{eq:pb}, first introduced in \cite{LMVC}, is a minimal hydrodynamic model of polarization, migration and deformation of a living cell confined between two parallel surfaces. In this model, the cell cytoplasm is an out of equilibrium system thanks to the active forces generated in the cytoskeleton. The cytoplasm is described as a passive viscous droplet in the Hele-Shaw flow regime. It contains a dilute solution which controls the active force induced by the cytoskeleton. Although relatively simple, this two-dimensional model predicts a very rich range of dynamic behaviors, see \cite{LMVC}. 

More precisely, in \eqref{eq:pression} -- \eqref{eq:cin} a Hele-Shaw cell is considered, that is a fluid droplet of constant viscosity is confined between two parallel plates separated by a gap. In such a case if $u$ denotes the gap-averaged planar flow and $P = P(t, x, y)$ is the fluid pressure, we let $\bm{u}=- \nabla P$ and rewrite equation \eqref{eq:pression} -- \eqref{eq:cin} as
\begin{align*}
	\bm{u}+ \nabla P&=0 &&\textrm{ in } \Omega(t),\\
	\nabla \cdot \bm{u} &=0 &&\textrm{ in } \Omega(t),\\
	P - \chi f(c) &= \gamma \kappa &&\textrm{ on } \Gamma(t),\\
	V_n &=\bm{V}_{\Gamma(t)}\cdot \bm{n}= \bm{u} \cdot \bm{n} &&\textrm{ on } \Gamma(t)\, \\
	\Omega(t=0)&=\Omega_0.&&
\end{align*}
Note that $\bm{u}$ averages the parabolic Hele-Shaw flow profile, which approximates the solution to the Stokes momentum-balance equation in thin films.

As anticipated, the novelty of this model lies in the normal force balance on the  boundary $\Gamma(t)$.  The classical Young-Laplace condition is perturbed by an active traction force, $-\chi f(c)\bm{n}$.  
This force is defined per unit length and is controlled locally by the gap-integrated concentration of an internal solute, $c = c(t, x, y)$. We stress that $f(c)$ can be either negative (pushing outwards) or positive (pulling inwards). In this regard, any uniform term $f_0 \in \xR$ added to $f(c)$ would merely offset the pressure $P$ by a constant and thus be irrelevant to the dynamics. 

The last boundary condition is the kinematic condition, stating that the normal velocity of the sharp interface, $V_n$, is given by the normal velocity of the fluid on $\Gamma(t)$.

To close the system, the internal solute transport problem is formulated in \eqref{eq:marqueur} -- \eqref{eq:marqueur_bord}. In the bulk $\Omega(t)$,  fast adsorption on the top and bottom plates (or onto an adhered cortex) is assumed. With rapid on and off rates, the quasi-2D transport dynamics are given by \eqref{eq:marqueur} -- \eqref{eq:marqueur_bord} where $a$ is the steady fraction of adsorbed molecules not convected by the average flow and the effective diffusion coefficient is assumed to be 1, see \cite{LMVC} for more details.

In \eqref{eq:marqueur_bord}, a zero solute flux on the moving boundary $\Gamma(t)$ is imposed. Simply put, the solute is effectively convected at a slower velocity than that of the fluid. Hence, its concentration decreases (increases) towards an advancing (retracting) front.

Finally the solute can be any cytoplasmic protein controlling the active force-generation / adhesion machinery. In this model, it is assumed that the concentration $c$ either induces an inwards pulling force or inhibits an outwards pushing force. This means that $f$ is assumed to satisfy $f'(c)>0$ for any $c>0$.\\

\subsection{Area preservation}

Using the fluid volume conservation, imposed by incompressibility \eqref{eq:pression} and kinematic condition \eqref{eq:cin}, we deduce that 
\[
\frac{\di }{\dt} |\Omega(t)|= \int_{\Gamma(t)} V_n \dsigma =- \int_{\Gamma(t)} \nabla P\cdot \bm{n} \dsigma =  -\int_{ \Omega(t)} \Delta P\dx \dy=0,
\] 
where $\dsigma$ denotes the infinitesimal length element of $\Gamma(t)$.
Hence, any smooth solution of \eqref{eq:pb} 
is area preserving:
\begin{equation}\label{eq:area2}
|\Omega(t)|=|\Omega_0| \qquad \forall t \ge 0.
\end{equation}

\subsection{Conservation of the marker content}\label{sec:cons_mass}
Let $M(t)$ denote the mass of  molecular content:
\[
M(t) =\int_{\Omega(t)} c(t,x,y)\dx \dy.
\]
Thanks to the boundary condition \eqref{eq:marqueur_bord} and to the kinematic condition \eqref{eq:cin}, we have
\begin{align*}
\frac{\di }{\dt} M(t) &= \int_{\Gamma(t)} c V_n \dsigma + \int_{\Omega(t)} \partial _t c \dx \dy \nonumber \\
&= -\int_{\Gamma(t)} D\nabla c \cdot \bm{n}  \dsigma + \int_{\Omega(t)} D\Delta c \dx \dy \nonumber \\
&= 0.\nonumber
\end{align*}

Thus, in \eqref{eq:pb}, we have formally conservation of molecular content:
\begin{equation}\label{eq:mass_marker2}
M(t)=M(0)=M.
\end{equation}

\subsection{Velocity of the center of mass}
For each $t > 0$, we define the momentum $\bm{\mathcal{M}}_{\Omega(t)}$ of $\Omega(t)$ by
\[
\bm{\mathcal{M}}_{\Omega(t)}= \int_{\Omega(t)} (x,y) \dx \dy=\int_{\Omega(t)}\bm{z}\t{d}\bm{z},
\]
where $\bm{z}=(x, y)$ is the vector coordinate of a point in $\Omega(t)$. In particular, $\bm{\mathcal{M}}_{\Omega(t)}$ is a vector containing the $x$ and $y$-momentum.\\
The center of mass $\bm{\mathcal{C}}_{\Omega(t)}$ of $\Omega(t)$ is defined by
\begin{equation*} 
\bm{\mathcal{C}}_{\Omega(t)}=\frac{\bm{\mathcal{M}}_{\Omega(t)}}{|\Omega(t)|}
=\frac{1}{|\Omega_0|}\int_{\Omega(t)} (x,y) \dx \dy,
\end{equation*}
by using the area preservation \eqref{eq:area2}.\\
The velocity  $\bm{u}_{\bm{\mathcal C}}(t)$ of the center of mass $\bm{\mathcal{C}}_{\Omega(t)}$ is
\begin{equation}\label{def:vit_center_mass}
\bm{u}_{\bm{\mathcal C}}(t)=\frac{\di }{\dt}\bm{\mathcal{C}}_{\Omega(t)}.
\end{equation}
From the incompressibility \eqref{eq:pression} and the boundary condition \eqref{eq:pression_bord}, we deduce that
\begin{eqnarray*}
	\frac{\di }{\dt}\int_{\Omega(t)} x  \dx \dy&=& \int_{\partial \Omega(t)} x V \dsigma = -\int_{\partial \Omega(t)} x \nabla P\cdot n \dsigma \nonumber \\
	&=& -\int_{\Omega(t)} \cnx \left(x \nabla P\right)\dx \dy=-\int_{\Omega(t)} \nabla P \cdot \nabla x \dx \dy\nonumber \\
	&=&-\int_{\Omega(t)} \cnx \left(P \nabla x \right)\dx \dy=- \int_{\partial \Omega(t)} P \nabla x \cdot n \dsigma \nonumber \\
	&=&- \int_{\partial \Omega(t)} \left( \gamma \kappa  +  \chi f(c)\right) n_x \dsigma\,  
\end{eqnarray*}
and similarly
$$	\frac{\di }{\dt}\int_{\Omega(t)} y  \dx \dy=- \int_{\partial \Omega(t)} \left( \gamma \kappa  + + \chi f(c)\right) n_y \dsigma
$$
Using that $\int_{\partial \Omega(t)}  \kappa    n \dsigma=0,$
it follows that
\begin{equation}
\bm{u}_{\bm{\mathcal C}}(t) =-\frac{\chi}{|\Omega_0| } \int_{\Gamma(t)}  f( c)  n \dsigma.\label{eq:centre_mass_2}
\end{equation}

\begin{rema}
We recognize that \eqref{eq:centre_mass_2} represents the external force balance on the droplet $\Omega(t)$.
\end{rema}

\subsection{Stationary solution}\label{sec:stat-sol}

\begin{proof}[Proof of Lemma \ref{lem:stat}]
Equations \eqref{eq:pression} and \eqref{eq:cin} imply that any stationary solution to \eqref{eq:pb} with  \eqref{eq:mass_marker2} satisfies 
\begin{equation*}
- \Delta P =  0 \quad \mbox{ in } \Omega_0, \qquad \nabla P \cdot \bm{n}  = 0\mbox{ on } \partial \Omega_0,
\end{equation*}
hence $\nabla P =0$ in $\Omega_0$. \\
Consequently, \eqref{eq:marqueur} and \eqref{eq:marqueur_bord} imply that $c$ satisfies
\begin{equation*}
- \Delta c =  0 \quad \mbox{ in } \Omega_0, \qquad \nabla c \cdot \bm{n}  = 0\mbox{ on } \partial \Omega_0,
\end{equation*}
and then $\nabla c =0$ in $\Omega_0$ too. From \eqref{eq:mass_marker2}, we deduce that $\tilde{c}=\frac{M}{\pi {R_0}^2}$, and the expression of $P$ follows from \eqref{eq:pression_bord}. \\
In particular, we deduce that the mean curvature in \eqref{eq:pression_bord} is constant, hence $\Omega_0=B_{R_0}$ since $\av{\Omega_0}=\av{B_{R_0}}$ by assumption.
\end{proof}

\subsection{Competition between the effects of surface tension and the marker}

Unlike  the classical Hele-Shaw equation with surface tension, the perimeter $\mathcal{P}(\Omega(t))$ defined by 
\[
\mathcal{P}(\Omega(t))=\int_{\partial \Omega(t)} \dsigma
\]
 is no more a Lyapunov functional for \eqref{eq:pb}. Indeed using a classical computation (see \cite{Otto}), we obtain 
\begin{align*}
\frac{\di }{\dt} \mathcal{P}(\Omega(t)) &= \int_{ \Gamma(t)} \kappa V_n \dsigma\\
& = \frac{1}{\gamma}  \int_{ \Gamma(t)}P \nabla P \cdot \bm{n} \dsigma - \frac{\chi}{\gamma} \int_{ \Gamma(t)}f(c) \nabla P \cdot \bm{n} \dsigma \nonumber \\
&= - \frac{1}{\gamma}\int_{\Omega(t)} |\nabla P|^2\dx \dy + \frac{\chi}{a\gamma} \int_{ \Gamma(t)} f(c) \nabla \log c \cdot \bm{n}  \dsigma . 
\end{align*} 

Note that if we consider the case where $f(c)=\pm c$, we obtain
\begin{equation*}
\frac{\di }{\dt} \mathcal{P}(\Omega(t)) =  - \frac{1}{\gamma}\int_{\Omega(t)} |\nabla P|^2\dx \dy \pm \frac{\chi}{a\gamma}\int_{\partial \Omega(t)}   \nabla c \cdot \bm{n}  \dsigma . 
\end{equation*} 
Then, the effects of the two terms located in the right side might be opposite and thus the term $\frac{\chi}{a\gamma}\int_{\partial \Omega(t)}   \nabla c \cdot \bm{n}  \dsigma$ might be  destabilizing. This leads to some interesting behaviors, see \cite{LMVC}. In this work we are interested in making part of this informal statement rigorous.

\section{Linear stability analysis. Proof of Proposition \ref{prop:lin_stab}}\label{sec:lin_stab}

In this section we perform a linear stability analysis characterizing the steady-state solution $(\tilde{c},\wt{P})$ given by \eqref{eq:stat}. This analysis shows that a global polarization-translation (motility) mode becomes unstable beyond a critical threshold of solute activity $\chi =\chi^*$, with $\chi^*$ defined by \eqref{def:chi_star}, that we recall now for the convenience of the reader
\begin{equation*}
	\chi^* := \frac{\pi R_0^2}{a  M f'\left( \frac{ M}{\pi R_0^2} \right)}= \frac{1}{a  \tilde{c} f'(\tilde{c})}.
\end{equation*}
Note that the stability analysis only depends on three factors, i.e. $a$, $\chi$ and $f'(\tilde{c})$.

We first construct the linearized operator $\mathcal A$ associated to \eqref{eq:pb}  around the circular homogeneous stationary solution \eqref{eq:stat}. Then, we study its spectrum and its eigenvectors. In particular, we prove that 
\[
\chi\le \chi^*
\]
 is a sufficient condition for all eigenvalues of $\mathcal A$ to be nonpositive. Finally we discuss the well-posedness character of $\mathcal A$.

\subsection{The linearized problem}

We first recall the definition of the Dirichlet-to-Neumann operator $\mathcal I$ on the open disk $B_{R_0}\subset \xR^2$, and then we give the linearized problem associated to \eqref{eq:pb}  around the circular homogeneous stationary solution $(\tilde{c},\wt{P})$ given by \eqref{eq:stat}.

\begin{defi}\label{def:DTN_0} 
 For $\psi \in H^1(\partial B_{R_0})$, the Dirichlet-to-Neumann operator $\mathcal I$ is defined by:
\begin{equation}\label{def:DTN} 
\mathcal I[\psi  ] = \nabla q \cdot \bm{n},
\end{equation}
where $ q$ denotes the harmonic extension of $\psi$ to the disk $B_{R_0}$, that is
\begin{equation*} 
- \Delta q =  0 \quad \mbox{ in } B_{R_0}, \qquad q = \psi \quad \mbox{ on } \mathbb \partial B_{R_0}.
\end{equation*}
\end{defi}

By  using Fourier series, the Dirichlet-to-Neumann $\mathcal I $ operator can be defined as the following linear operator :
\begin{defi}\label{def:DTN_3}
Given $\psi\, :\, \xR / 2\pi \xZ\to \xR$ with Fourier series
\[ 
\psi(\theta) = a_0+ \sum_{m=1}^\infty  a_m \cos (m\theta) + b_m \sin (m\theta) , 
\]
we set 
\begin{equation*} 
\mathcal I  (\psi)(\theta) = \sum_{m=1}^\infty   m \left(a_m \cos (m\theta) + b_m \sin (m\theta)\right).
\end{equation*}
\end{defi}
Indeed, by the Definition \ref{def:DTN_0},  we have
\[
q(r,\theta)= \sum_{m\ge 0}  a_{m} \cos(m\theta) r^m +b_{m}  \sin(m\theta)r^{m}\q\t{for }(r,\theta ) \in [0,R]\times \xR/2\pi \xZ
\] 
where we discarded solutions that diverge at $r = 0$.\\
Hence, the Definition \ref{def:DTN_3} follows from \eqref{def:DTN} and
\[
\partial_r q(r,\theta) = \sum_{m\ge 0}  mr^{m-1}\left( a_{m} \cos(m\theta) +b_{m} \sin(m\theta)\right) .
\]

\medskip

We take a perturbation of the free boundary of the form 
\[
r=  R_0 + \varepsilon \varphi(t,\theta),
\]
i.e.
\begin{equation*} 
\Omega(t)   = \left\{ (x,y)=\left( r \cos \theta ,  r \sin \theta \right); 0 \le r < R_0 + \varepsilon \rho(t,\theta)\right\} .
\end{equation*}

\begin{lemm}\label{prop-ex-pb-lin}
The linearized problem associated with \eqref{eq:pb}  around the radially symmetric solution \eqref{eq:stat} is 
\begin{equation}\label{eq:lin_general}
\frac{\di }{\dt} \left( \begin{array}{c}
\rho \\
c
\end{array} \right)=\mathcal A \left( \begin{array}{c} 
\rho \\
c
\end{array} \right),
\end{equation}
where $\mathcal A$ is the operator defined on $H^3\left(\partial B_{R_0}\right) \times H^2\left(B_{R_0}\right)$ by
\begin{equation*} 
\mathcal A :         
\left( \begin{array}{c} 
\rho \\
c
\end{array} \right)
\mapsto
\left( \begin{array}{c} 
\mathcal I \left[\frac{\gamma}{R_0} \left(\partial_{\theta \theta}^2\rho + \rho\right) - \chi f'(\tilde{c})c\right]\\
\Delta c  
\end{array} \right),
\end{equation*}
with the boundary condition  
\begin{equation}\label{eq:bord_lin_fourier}
 \partial _r c =  -a \tilde{c} \, \mathcal I \left[\frac{\gamma}{R_0} \left(\partial_{\theta \theta}^2\rho + \rho\right)- \chi f'(\tilde{c})c\right]\qq\t{on }\partial B_{R_0} .
\end{equation} 
\end{lemm}

\begin{proof}
We perform a formal expansion of the solution $(P,c)$ to \eqref{eq:pb} near the radially symmetric solution \eqref{eq:stat}, $\tilde{c}=\frac{M}{\pi R_0^2}$, $\wt{P}=\frac{\gamma}{R_0} +\chi f\left(\frac{M}{\pi R_0^2}\right)$: 
\begin{align*}
P(t,r,\theta)&=\wt{P} +\varepsilon Q(t,r,\theta) + \mathcal O (\varepsilon ^2), \\
c(t,r,\theta)&=\tilde{c} +\varepsilon S(t,r,\theta)+ \mathcal O (\varepsilon ^2). 
\end{align*}
For $\theta \in (-\pi,\pi]$ and $t\ge 0$, we easily find that
\begin{align*}
\Delta Q(t,r,\theta)&=0 && \textrm{ if } r<R_0, 
\\
\partial _t \rho (t,\theta)  &= -\partial _r Q(t,1,\theta),&&  \\ 
Q(t,1,\theta) &=-\frac{\gamma}{R_0} \left( \partial ^2_{\theta \theta} \rho(t,\theta)  +\rho(t,\theta) \right) +\chi f'(\tilde{c})S(t,R_0,\theta). && 
\end{align*}

Firstly, the last formula is obtained by using the general formula for the curvature of a curve $r=g(\theta)$,
\begin{equation*} 
	\kappa(g)= \frac{2\left(g '(\theta)\right)^2 - g (\theta)g ''(\theta)+ g (\theta)^2}{\left( g (\theta)^2 + \left(g '(\theta)\right)^2 \right)^{3/2} }, \qquad r=g(\theta)
\end{equation*}
which gives, to the first order in $\varepsilon$, 
\[	
\kappa\left(R_0+\varepsilon \rho (\theta)\right) =\frac1R_0-\frac{\varepsilon}{R_0^2} \left(\partial ^2_{\theta \theta}\rho(t,\theta) +\rho(t,\theta)\right).
\]
The additional term $\chi f'(\tilde{c})S(t,R_0,\theta)$ comes from the linearization of $f(c)$.\\

The second formula follows from the definition of the normal derivative
\[
\frac{\partial P}{\partial n }=\nabla P \cdot \bm{n} = \frac{1}{\left(g(\theta)^2 +\left(g'(\theta)\right)^2\right)^{1/2}} \left( g(\theta)\frac{\partial P}{\partial r} - \frac{g'(\theta)}{g(\theta)} \frac{\partial P}{\partial \theta}\right) 
\]
along the curve $r = g(\theta)$ which gives, to the first order in $\varepsilon$, 
\[
\nabla P \cdot \bm{n}_{\left(R_0+\varepsilon \rho (\theta)\right)}= \frac{\partial P}{\partial r}  \bm{e}_r - \frac{\varepsilon}{R_0} \partial_\theta \rho(t,\theta) \frac{\partial P}{\partial \theta}  \bm{e}_\theta,
\]
where $\bm{n}_{\left(R_0+\varepsilon \rho (\theta)\right)}$ is the  the unit normal vector to the boundary of
\begin{equation}\label{eq:set}
\{r = R_0 + \varepsilon \rho (t,\theta)\},
\end{equation}
and
\[
\bm{e}_r=\begin{pmatrix}
\cos\theta\\\sin \theta
\end{pmatrix}
\q\t{and}\q 
\bm{e}_{\theta}=
\begin{pmatrix}
\sin\theta\\-\cos\theta
\end{pmatrix}.
\]
Furthermore, up to the first order in $\varepsilon$, the normal velocity $V_{n_{\left(R_0+\varepsilon \rho (\theta)\right)}}$ to the boundary  of \eqref{eq:set} is
\[
V_{n_{\left(R_0+\varepsilon \rho (\theta)\right)}}  =\varepsilon \partial_ t \rho(t,\theta).
\]

The linearization of the convection diffusion equation \eqref{eq:marqueur} -- \eqref{eq:marqueur_bord} around $(\tilde{c},\wt{P})$ is  
\[
\varepsilon \partial _t S +\varepsilon (a-1) \left(\nabla \wt{P} \cdot \nabla S + \nabla Q\cdot \nabla \tilde{c}\right) +\varepsilon ^2 (a-1) \nabla Q \cdot \nabla S-\varepsilon \Delta S+ \mathcal O (\varepsilon ^3)=0.
\]
Moreover, since $\tilde{c}$ and $\wt{P}$ are real numbers and neglecting the $\varepsilon ^2$ convection term, we obtain the heat equation. A similar computation yields the boundary term, hence for $\theta \in (-\pi,\pi]$ and $t\ge 0$, 
\begin{align*}
\partial_t S(t,r,\theta)&= \Delta S (t,r,\theta) && \textrm{ if } r<R_0, \\ 
\partial _r S (t,R_0,\theta) &= a \tilde{c} \, \partial _r Q(t,R_0,\theta).&&  
\end{align*}
Hence, the result.
\end{proof}

\subsection{Eigenvalue problem for  $\mathcal A$}

The eigenvalue problem for $\mathcal A$ is:
\begin{equation*}  
\mathcal A \left( \begin{array}{c}
\rho \\
c
\end{array} \right)=\lambda \left( \begin{array}{c}
\rho \\
c
\end{array} \right).
\end{equation*}
Thanks to the radial symmetry of the problem the spectral analysis of $\mathcal A$ amounts to perform a Fourier analysis.

\begin{lemm}\label{lem:ind:mode}
Given $\left( \rho ,c \right)$ with Fourier series, 
then \eqref{eq:lin_general} -- \eqref{eq:bord_lin_fourier} describes a closed dynamical system for the cosine (resp. sine)  perturbations.
\end{lemm}

\begin{proof}
Let 
\[
\left( \begin{array}{c}
	\rho \\
	c
\end{array} \right)= \sum_{m\in\xN} \left( \begin{array}{c}
	\hat \rho_{cm}  \\
	\hat c_{cm}  (r) 
\end{array} \right)\cos(m\theta ) + \sum_{m\in\xN} \left( \begin{array}{c}
	\hat \rho_{sm}   \\
	\hat c_{sm}  (r) 
\end{array} \right)\sin(m\theta ),
\] 
by linearity of the operator $\mathcal A$, we see that
\begin{equation*} 
\mathcal A \left( \begin{array}{c}
\rho  \\
c 
\end{array} \right) =
\sum\limits_{m\ge 0} \mathcal A_m \left( \begin{array}{c}
\hat \rho_{cm}  \\
\hat c_{cm}  (r) 
\end{array} \right) \cos(m\theta )
+ \sum\limits_{m\ge 0} \mathcal A_m \left( \begin{array}{c}
\hat \rho_{sm}  \\
\hat c_{sm}  (r)
\end{array} \right) \sin(m\theta ),
\end{equation*}
with $\mathcal A_m$ defined by
\begin{equation}\label{def:A_mfourier}
\mathcal A_m \left( \begin{array}{c}
\hat  \rho \\
\hat c  (r) 
\end{array} \right)
=
\left( \begin{array}{c} 
 \mathcal I\left[-\frac{\gamma}{R_0} (m^2-1)\hat \rho- \chi f'(\tilde{c})\hat c  (R_0) \right]\\
\left( \partial_{r}^2 + r^{-1}\partial_r  - r^{-2}m^2  \right)\hat c(r)
\end{array} \right)  .
\end{equation}
Furthermore, the boundary condition \eqref{eq:bord_lin_fourier} on $\partial B_{R_0}$ is
\begin{equation*} 
	\partial _r \hat c(R_0) =  -a \tilde{c} \, \mathcal I \left[-\frac{\gamma}{R_0} (m^2-1)\hat \rho - \chi f'(\tilde{c})\hat c  (R_0)\right] .
\end{equation*}
\end{proof}

\begin{rema}[On symmetry properties]
Since the cosine (resp. sine) perturbations in the dynamical system \eqref{eq:lin_general} -- \eqref{eq:bord_lin_fourier}  are independent, by rotating the coordinate system we may consider, without loss of generality,  perturbations possessing the reflection symmetry with respect to the $x$-axis.\\
Specifically we assume symmetry of the initial data, namely domain $\Omega(0)$ and $c(0,x, y)$, with respect to $x$-axis which is preserved for $t > 0$ according to Lemma \ref{lem:ind:mode}.
\end{rema}

Let $\lambda \in \xC$. Then,  the eigenfunctions of $\mathcal A$ associated to the eigenvalue $\lambda$ are of the form 
\[
c(r,\theta) = \sum\limits_{m\ge 0} \hat c_{m\lambda}(r) \cos(m\theta )\q\t{and}\q \rho(\theta)  =\sum\limits_{m\ge 0} \hat \rho_{m\lambda}  \cos(m\theta )
\]
$ \left( \hat \rho_{m\lambda} ,\hat c_{m\lambda}\right)$ satisfying 
\begin{equation*} 
\mathcal A_m \left( \begin{array}{c}
\hat \rho_{m\lambda} \\
\hat c_{m\lambda}
\end{array} \right)=\lambda  \left( \begin{array}{c} 
\hat \rho_{m\lambda} \\
\hat c_{m\lambda}
\end{array} \right),
\end{equation*}
and
\begin{equation*} 
	\partial _r \hat c_{m\lambda}(R_0) =  -a \tilde{c} \, \mathcal I \left[-\frac{\gamma}{R_0} (m^2-1)\hat \rho_{m\lambda} - \chi f'(\tilde{c})\hat c_{m\lambda} (R_0)\right] ,
\end{equation*}
where $\mathcal A_m$ is  defined by \eqref{def:A_mfourier}.

Thanks to the previous remarks, we deduce the following result.\\
  For the sake of clarity, we omit the $\hat{\cdot}$ and the subscripts $m\lm$ when no ambiguity occurs.

\begin{lemm}
The problem
\[
\mathcal A_m \left( \begin{array}{c}
\rho \\
c
\end{array} \right)=\lambda  \left( \begin{array}{c} 
\rho \\
c
\end{array} \right),
\]
and 
\[
	\partial _r c =  -a \tilde{c} \, \mathcal I \left[-\frac{\gamma}{R_0} (m^2-1) \rho - \chi f'(\tilde{c}) c (R_0)\right] ,
\]
is equivalent to the following one
\begin{subequations}
\begin{align}
\lambda \rho &=-\partial _r P &&\mbox{ on }  \partial B_{R_0},\label{vp:dyn}\\
-\Delta P&=0 &&\mbox{ in }  B_{R_0},\label{vp:incomp}\\
P&=\frac{\gamma}{R_0} ( m ^2-1) \rho +\chi f'(\tilde{c}) c &&\mbox{ on }  \partial B_{R_0},\label{vp:bord}\\
\lambda c&= \Delta c &&\mbox{ in } B_{R_0},\label{vp:chaleur}\\
\partial _r c &= a\tilde{c}  \partial _r P &&\mbox{ on } \partial B_{R_0}.\label{vp:bord_chaleur}
\end{align}
\end{subequations}
\end{lemm}

\subsection{Sufficient conditions for the operator $\mathcal A$ to have all its eigenvalues with non-positive real parts}\label{ssec:A}

In this paragraph we give a sufficient condition on the non-negative parameters $a$, $\tilde{c}$, $f'(\tilde{c})$ and $\chi$ so that all the eigenvalues of $\mathcal A$  have non-positive real parts. \\
To do so, let 
\[
Q:=c-a\tilde{c} P
\]
with $c$ and $P$ satisfying \eqref{vp:dyn} -- \eqref{vp:bord_chaleur}.  We see that
\begin{subequations}
\begin{align}
	\Delta Q&= \lambda c &&\mbox{ in }  B_{R_0},\label{eq:Q}\\
	\pare{1-\frac{\chi}{\chi^*}} P&=\frac{\gamma }{R_0}( m ^2-1) \rho +\chi f'(\tilde{c}) Q &&\mbox{ on }  \partial B_{R_0},\label{eq:Q:bord}\\
	\nabla Q \cdot \bm{n} &=0 &&\mbox{ on }  \partial B_{R_0}.\label{eq:fluxQ:bord}
\end{align}
\end{subequations}
We recall that $\chi^*$ has been defined in \eqref{def:chi_star}.

Here we prove the following result.

\begin{prop}\label{prop:signe:vp}
	Assume 
\[
0\le \frac{\chi}{\chi^*} \le 1,
\]
 $\chi >0$ and  $m\ge 1$. Then, all the eigenvalues of $\mathcal A_m$  have non-positive real parts. More precisely, for $c$ and $P$ satisfying \eqref{vp:dyn} -- \eqref{vp:bord_chaleur}, then the following equality holds
	\begin{align}
		\lambda \int_{B_{R_0}}|c|^2 \dx\dy +  \frac{a \tilde{c} \bar \lambda \gamma (m^2-1)}{R_0\chi f'(\tilde{c})} \int_{\partial B_{R_0}}|\rho|^2 \dsigma  
	&=   -\frac{a\tilde{c} \left(1- \frac{\chi}{\chi^*} \right)}{\chi f'(\tilde{c}))} \int_{ B_{R_0}}|\nabla  P |^2  \dx\dy  \nonumber\\
&\q     -\int_{B_{R_0}}|\nabla Q|^2 \dx\dy .\label{eq:signe:vp}
	\end{align}
\end{prop}

\begin{proof}
	We compute
	\begin{align}
		\lambda \int_{B_{R_0}}|c|^2 \dx\dy&=   \int_{B_{R_0}}\bar c \Delta c  \dx\dy\nonumber\\
& = \int_{B_{R_0}}\left( \bar{Q}+a\tilde{c} \bar{P}\right) \Delta Q  \dx\dy\nonumber \\
		&=    -\int_{B_{R_0}}|\nabla Q|^2 \dx\dy +a \tilde{c} \int_{ B_{R_0}}\bar P \Delta Q  \dx\dy \nonumber \\
		&=    -\int_{B_{R_0}}|\nabla Q|^2 \dx\dy -a \tilde{c} \int_{ B_{R_0}}\nabla \bar P \nabla Q  \dx\dy  .\nonumber
	\end{align}
	Since $P$ is harmonic and using \eqref{eq:Q:bord} together with \eqref{vp:dyn}, we deduce that
	\begin{align*}
		\int_{B_{R_0}}\nabla \bar P \nabla Q \dx\dy&= \int_{B_{R_0}}\cn \left( Q\nabla \bar P \right) \dx\dy \\
		&=   \int_{\partial B_{R_0}}Q \nabla \bar P \cdot \bm{n}  \dsigma\nonumber \\
		&=    \frac{1}{\chi  f'(\tilde{c})} \int_{\partial B_{R_0}}\left( \left(1- \frac{\chi}{\chi^*} \right)  P -\frac{\gamma }{R_0}(m^2-1) \rho \right) \nabla \bar P \cdot \bm{n}  \dsigma  \nonumber \\
		&=   \frac{\bar \lambda \gamma (m^2-1)}{R_0\chi f'(\tilde{c})} \int_{\partial B_{R_0}}|\rho|^2 \dsigma +\frac{1- \frac{\chi}{\chi^*}}{ \chi f'(\tilde{c})} \int_{ B_{R_0}}|\nabla  P |^2  \dx\dy  , \nonumber 
	\end{align*}
	from which the result follows.
\end{proof}

\begin{rema}
	If we have the strict inequality $0<  \frac{\chi}{\chi^*} < 1$, and if $\lambda = iw$, $w \in \xR$, is an eigenvalue of $\mathcal A_m$, then we must have $w =0$.  Indeed in such a case both sides of \eqref{eq:signe:vp} are equal to zero, hence $\nabla P=0$ in $B_{R_0}$. Consequently $\Delta P=0$ and therefore $\lambda =0$ or $c=0$ in $B_{R_0}$. But, if $c=0$ in $B_{R_0}$, then $\rho =0$ on $\partial B_{R_0}$.
\end{rema}

\begin{rema}[On the case $f'(\tilde{c})\le 0$]
	If we consider the negative case for the function $f$, i.e. $f'(\tilde{c})\le 0$, then for any non negative values of $a$ and $\chi$, all the eigenvalues of $\mathcal A$  have non positive real parts.  
\end{rema}

Let $\lambda _{m,p}$ be the $p$-th real positive root of $J_m'$, the derivative of the Bessel function of the first kind of order $m$, $J_m$, see the Appendix \eqref{eq:root_bessel}. Then a straightforward computation gives
\begin{prop}
	Assume that $a\chi = 0$, then the spectrum of the operator $\mathcal A$ is 
	\[
	\left\{ 
	-m(m^2-1),-\lambda _{m,p}^2, m\in \xN, p\in \xZ \right\}.
	\]
\end{prop}

\subsection{Spectrum of $\mathcal A$}

In this part we describe in more details the eigenvalues of $\mathcal A$. The techniques developed in this section depend, in part, on expanding solutions of \eqref{vp:dyn} -- \eqref{vp:bord_chaleur} in terms of the modified Bessel functions $I_m(x)$.

\begin{lemm}
	For $m\in \xN$, the eigenfunctions of $\mathcal A_m$ associated with the eigenvalue $\lambda \in \xC$  are 
	\begin{equation*} 
		\left( \begin{array}{c}
			\rho (\theta)\\
			c(r,\theta)
		\end{array} \right)=\left( \begin{array}{c}
			\hat \rho_{m\lambda}   \\
			\hat c_{m\lambda}  I_m\left(-r\lambda  ^{1/2}\right)
		\end{array} \right) \cos(m\theta ),
	\end{equation*}
	with $0<r<R_0$, $\theta \in (-\pi,\pi]$ and $(\hat c_{m\lambda}, \hat \rho_{m\lambda}) \in \xC ^2$  solutions of 
	\begin{align}\label{eq:eigenmode_mod}
	\left(\lambda +\frac{\gamma}{R_0^2} m(m^2-1)\right)   \hat \rho_{m\lambda}&=-\chi \frac{f'(\tilde{c})}{R_0} m I_m\left(-R_0\lambda ^{1/2}\right)\hat c_{m\lambda} ,  \\
		\frac{\sqrt{\lambda}}{2}\left(I_{m-1}\left(-R_0\lambda ^{1/2}\right) +I_{m+1}\left(-R_0\lambda ^{1/2}\right)\right)  \hat c_{m\lambda} &=\lambda a \tilde{c} \, \hat \rho_{m\lambda} . \label{eq:eigenmode_mod_2}
	\end{align}
\end{lemm}

\begin{proof}
	In what follows,  we only consider solutions that are smooth in $r = 0$. \\
	 
	Since $P$ satisfies the Laplace equation, solutions have following expression:
	\[
	P(r,\theta)=A_{m\lambda} r^m \cos(m\theta)\qq \t{for }0<r<R_0,\, \theta \in (-\pi, \pi],
	\] 
 with $A_{m\lambda} \in \xC$.\\
	Consider  
	$\rho(\theta)  =\hat \rho_{m\lambda}  \cos(m\theta )$, with  $\hat \rho_{m\lambda}  \in \xC$. 
	Using \eqref{vp:dyn}, we get that
	\[\lambda \hat \rho _{m\lambda}=-mR_0^{m-1}A_{m\lambda}.
	\]
	Furthermore, for $c(r,\theta) =  \hat c_{m\lambda}(r) \cos(m\theta )$, we have that \eqref{vp:bord} reads
	\[
	A_{m\lambda} R_0^m  =\frac{\gamma}{R_0}  (m^2-1) \hat \rho_{m\lambda} +\chi f'(\tilde{c}) \hat c_{m\lambda}(R_0) ,
	\]
	so that,
	\begin{equation}\label{eq:bord_force}
		\lambda \hat \rho _{m\lambda}=-\frac{\gamma}{R_0^2} m(m^2-1)\hat \rho_{m\lambda}-\chi m \frac{f'(\tilde{c})}{R_0}\hat c_{m\lambda}(R_0).
	\end{equation}
	Finally equation \eqref{vp:chaleur} yields 
	\begin{equation}\label{eq:bulk}
		\left( \partial_{r}^2 + r^{-1}\partial_r  - r^{-2}m^2  \right)\hat c_{m\lambda}  (r) = \lambda  \hat c_{m\lambda}  (r) , \quad r<R_0,
	\end{equation}
	and the boundary condition \eqref{vp:bord_chaleur} gives 
	\begin{equation}\label{eq:bord_lin}
		\partial_r \hat c_{m\lambda}  (R_0) = -\lambda a \tilde{c} \, \hat \rho _{m\lambda}.
	\end{equation}
	The smooth solutions of \eqref{eq:bulk} are known and given by 
	\[
	\hat c_{m\lambda}  (r) =\hat c_{m\lambda} I_m\left(-r\lambda ^{1/2}\right), \quad r<R_0,
	\] 
	with $\hat c_{m\lambda} \in \xC$ and where $I_m$ is the modified Bessel function of the first kind of order $m$, whose definition is recalled in \eqref{eq:Bessel_m} and \eqref{eq:Bessel_modified}.  \\
The result then follows by substituting the expression of $\hat c_{m\lambda}  (r)$ in \eqref{eq:bord_force} and \eqref{eq:bord_lin} together with the property of the Bessel function 
\[
I_m'(x)=\left(I_{m-1}(x)+I_{m+1}(x)\right)/2.
\]
\end{proof}

Let $H_m$ be the function defined by
\begin{equation}
	H_m(z) =z^{1/2}I'_{m}\left(-R_0z ^{1/2}\right) \left(z +\frac{\gamma}{R_0^2}  m(m^2-1)\right)  
	+a\chi \frac{\tilde{c} f'(\tilde{c})}{R_0}  mz  I_m\left(-R_0z ^{1/2}\right).\label{eq:eigenmode_mod_5}
\end{equation} 
The eigenvalue equation is $H_m(z)=0$, hence we deduce the spectrum of $\mathcal A_m$.
\begin{lemm}
	The spectrum of $\mathcal A_m$ is 
	\[
	sp(\mathcal A_m)=\{\lambda \in \xC \textrm{ s.t. }H_m(\lambda)=0\}.
	\]
\end{lemm}

We can now deduce information on the eigenvalues of $\mathcal A$.

\begin{theo}\label{the:lin_stab}
	For  $\gamma >0$,  the following dichotomy holds.
	\begin{itemize}
		\item[(i)] If 
\[
0<\frac{\chi}{\chi^*} \le 1,
\]
 then $\mathcal A$ has zero eigenvalue $\lambda =0$ of multiplicity three, while the other eigenvalues have negative real parts.
		\item[(ii)] If 
\[
\frac{\chi}{\chi^*}  >1,
\]
 then the operator $\mathcal A$ has a positive eigenvalue $\lambda >0$. 
	\end{itemize}
\end{theo}

\begin{proof}
	Recall that the eigenmode $v_{m\lambda} (r,\theta)$ is defined by 
	\begin{equation}\label{eq:eigenmode_3}
		v_{m\lambda} (r,\theta)=
		\left( \begin{array}{c}
			\hat \rho _{m\lambda} \\
			\hat c_{m\lambda}   I_m(-r\sqrt{\lambda})
		\end{array} \right)
		\cos(m\theta).
	\end{equation}

\noindent
{\it Proof of (i).}\\
	Let us start with the $m=0$ mode. Substituting $m = 0$ in \eqref{eq:eigenmode_mod_5} gives
	\begin{equation}\label{eq:mode0}
		H_0(\lambda) = \frac{\lambda^{3/2}}{2}I_{1}(-R_0\sqrt{\lambda}).
	\end{equation}
	Using \eqref{eq:eigenmode_mod} and \eqref{eq:eigenmode_mod_2}, we see that $\lambda = 0$ is associated to the two following eigenmodes
	\begin{equation*}
		v_{00}^1 (r,\theta)=
		\left( \begin{array}{c}
			I_0(0)\\
			0
		\end{array} \right)=\left( \begin{array}{c}
			1\\
			0
		\end{array} \right)
		 \quad \textrm{and} \quad 
		v_{00}^2 (r,\theta)=
		\left( \begin{array}{c}
			0\\
			1 
		\end{array} \right)\, 
		.
	\end{equation*} 
	Moreover, rewriting \eqref{eq:mode0} as
	\begin{equation*} 
		H_0(\lambda) =-i \frac{\lambda^{3/2}}{2}J_{1}\left(-iR_0\lambda ^{1/2}\right),
	\end{equation*} 
	we deduce the other roots of $H_0$, that is $\lambda_{0k} = -x^2_{1k}<0$, where $x_{1k}>0$ is the $k$-th zero of  the Bessel function of order 1, $J_1(x)$. These real-negative roots, $\lambda_{0k}$, are associated with the $m = 0$ diffusion mode, given by
	\begin{equation*}
		v_{0k} (r,\theta)=
		\left( \begin{array}{c}
			J_0(x_{1k}r)\\
			0
		\end{array} \right)
		,
	\end{equation*} 
	where $J_0$ is the Bessel function of order 0.\\
	
	Consider now the $m=1$ mode. Using \eqref{eq:eigenmode_mod_5} it yields to
	\begin{equation*}
		H_1(\lambda) =\lambda^{3/2}I'_{1}\left(-R_0\lambda ^{1/2}\right)    +a\chi  \frac{\tilde{c} f'(\tilde{c})}{R_0}\lambda  I_1\left(-R_0\lambda ^{1/2}\right). 
	\end{equation*}
	We see that $\lambda =0$ is a root of $H_1$ and, using \eqref{eq:eigenmode_mod_2}, it is associated with the eigenmode 
	\begin{equation*}
		v_{10} (r,\theta)=
		\left( \begin{array}{c}
			0\\
			1
		\end{array} \right)\cos \theta
		.
	\end{equation*} 
	
	Let us now look for non zero roots of $H_1$.\\ Let $\lambda$ be a positive root of $H_1$, then $x=-\sqrt{\lambda}<0$ satisfies
	\begin{equation}
		x\frac{I'_{1}(R_0x)}{I_1(R_0x)}  -a\chi \frac{\tilde{c} f'(\tilde{c})}{R_0}=0\,  .\label{eq:eigenmode_mod1_bis}
	\end{equation}
	Since
	\begin{equation*}
		R_0x\frac{I'_{1}(R_0x)}{I_1(R_0x)} =R_0x\frac{I_{2}(R_0x)}{I_1(R_0x)} +1,
	\end{equation*}
	equation \eqref{eq:eigenmode_mod1_bis} rewrites as 
	\begin{equation}
		R_0x\frac{I_{2}(R_0x)}{I_1(R_0x)}  +1-a\chi \tilde{c} f'(\tilde{c}) =0\,  .\label{eq:eigenmode_mod1_ter}
	\end{equation}
	Recalling next the definition of the Bessel function $I_m$, 
	\[
	I_m(x)= \sum _{p=0}^\infty \frac{1}{p!(m+p)!}\left( \frac{x}{2}\right)^{m+2p},
	\]
	we see that 
	\[
	x\frac{I_{2}(x)}{I_1(x)} > 0 \qquad \textrm{ if } x< 0, 
	\]
	thus from \eqref{eq:eigenmode_mod1_ter}, it follows that if $0<a\chi f'(\tilde{c})\le 1$, there is no positive root of $H_1$. 
	
	Let us next prove that there is no root $\lambda$ of $H_1$ that belongs to $\xC\setminus \xR$. By contradiction, let us assume that $z\in \xC\setminus \left( \xR \cup i\xR\right) $ is such that $z^2 =\lambda$ and is solution of \eqref{eq:eigenmode_mod1_ter}. Since $\bar z ^2 \neq z ^2$ and
	\begin{equation*} 
		z \frac{I_2(R_0z)}{I_1(R_0z)}= \bar z \frac{I_2(R_0\bar z)}{I_1(R_0\bar z)},
	\end{equation*} 
	we get
	\begin{equation*}  
		z I_2(R_0z)I_1(R_0\bar z) - \bar z I_2(R_0\bar z) I_1( R_0z)=0.
	\end{equation*} 
	Using equality \eqref{eq:relation_Bessel}, this leads to $z^2=\bar z^2$ hence a contradiction.

	From the two previous arguments we deduce that the only non-zero roots of $H_1$ are real and non positive. \\
	
	Consider now the $m\ge 2$ mode. Using \eqref{eq:eigenmode_mod_5}, it rewrites as
	\begin{equation*}
		H_m(\lambda) =\lambda^{1/2}I'_{m}\left(-R_0\lambda ^{1/2}\right) \left(\lambda +\frac{\gamma}{R_0^2} m(m^2-1)\right)  +a\chi\frac{\tilde{c} f'(\tilde{c})}{R_0} m\lambda  I_m\left(-R_0\lambda ^{1/2}\right).\label{eq:eigenmode_mod_5bis}
	\end{equation*}
	
	We see that $\lambda =0$ is a root of $H_m$. In such a case, since $\gamma>0$, using \eqref{eq:eigenmode_mod} we deduce that $\hat \rho _{m0}=0$. Furthermore, from $I_m(0)=0$ and \eqref{eq:eigenmode_3}, we get  that $\lambda=0$ is  associated with the zero eigenmode
	\[v_{m0}
	(r,\theta)=
	\left( \begin{array}{c}
		0\\
		0
	\end{array} \right)\cos (m\theta)
	,
	\]
	thus $\lambda =0$ is not an eigenvalue for $m\ge 2$ and $\gamma >0$.\\
	
	Let us next look for other roots of $H_m$.\\ Let $\lambda$ be a positive root of $H_m$, then $x=-\sqrt{\lambda}<0$ satisfies
	\begin{equation*}
		x\frac{I'_{m+1}(R_0x)}{I_m(R_0x)}  -a\chi \frac{\tilde{c}f'(\tilde{c})}{R_0} \frac{m x^2}{ x^2 + \frac{\gamma}{R_0^2} m(m^2-1)}=0\,  . 
	\end{equation*}
	Since
	\begin{equation*}
		R_0x\frac{I'_{m}(R_0x)}{I_m(R_0x)} =R_0x\frac{I_{m+1}(R_0x)}{I_m(R_0x)} +m,
	\end{equation*}
	it rewrites as 
	\begin{equation*}
		R_0x\frac{I_{m+1}(R_0x)}{I_{m}(R_0x)}  +m-a\chi \tilde{c} f'(\tilde{c})  \frac{m x^2}{ x^2 + \frac{\gamma}{R_0^2} m(m^2-1)}=0\,  . 
	\end{equation*}
	Using  that 
	\[
	x\frac{I_{m+1}(x)}{I_{m+1}(x)} > 0 \qquad \textrm{ if } x< 0, 
	\]
	we see that there is no positive root of $H_m$ if $0<a\chi \tilde{c} f'(\tilde{c})\le 1$.  \\
	The end of the proof of $(i)$ (the case where $\lambda \in \xC$) is a consequence of  Proposition \ref{prop:signe:vp}.\\

\noindent
{\it Proof of (ii).}\\
	Let us now prove $(ii)$. To do so let us exhibit a positive eigenvalue of $\mathcal A$. More precisely, let us expand $H_1(\lambda)$ around $\lambda =0$: 
	\begin{equation*} 
		H_1(\lambda)= \frac{\lambda^{3/2}}{2}\left(1-a\chi \tilde{c} f'(\tilde{c}) +\frac18(3-a\chi \tilde{c} f'(\tilde{c}))\lambda  R_0^2\right) +O(|\lambda| ^{7/2}),
	\end{equation*}
	which can be written as
	\[
	H_1(\lambda)=  \frac{\lambda^{3/2}}{2} g(\lambda) +O(|\lambda| ^{7/2}).
	\]
	The function $g$ has  $\lambda_1= \frac{8(a \chi \tilde{c} f'(\tilde{c}) -1)}{R_0^2\left(3-a\chi \tilde{c}f'(\tilde{c})\right)}$ as root, which changes sign from negative to positive as $a\chi \tilde{c} f'(\tilde{c})$ exceeds 1. Note that $\lambda_1$ approximates a true eigenvalue of $\mathcal A$, for $a\chi \tilde{c}f'(\tilde{c})$ close to  1. In such a case, we can write $\lambda_1 = \frac{4}{R_0^2}(a\chi \tilde{c} f'(\tilde{c})  - 1)+o(|a\chi \tilde{c} f'(\tilde{c})  - 1|)$ and its associated eigen-mode is 
	\begin{align*}
		v_{1\lambda_1} (r,\theta)&=
		\left( \begin{array}{c}
			\lambda_1   I_1(r\sqrt{\lambda_1})\\
			-\chi \tilde{c} \tilde{c}f'(\tilde{c})  I_1(\sqrt{\lambda_1})
		\end{array} \right)
		\cos(\theta)\nonumber \\
		&= \frac{\sqrt{\lambda_1}}{2}\left( \begin{array}{c}
			-\lambda_1r\\
			\chi \tilde{c} \tilde{c} f'(\tilde{c}) 
		\end{array} \right)
		\cos(\theta) +O(|a\chi f'(\tilde{c})  - 1|^{3/2})\nonumber \\
		&=\frac{\sqrt{a\chi \tilde{c} f'(\tilde{c}) - 1}}{2}\left( \begin{array}{c}
			\frac{4}{R-0^2} (1-a \chi \tilde{c} f'(\tilde{c}))r\\
			\chi  
		\end{array} \right)
		\cos(\theta)    +O(|a\chi \tilde{c} f'(\tilde{c}) - 1|^{3/2}),
	\end{align*} 
	where we expanded the Bessel functions for small values of $|\lambda_1|$. \\
	
	Note that as $a\chi \tilde{c} f'(\tilde{c})$ crosses $1$,  the solute gradient component in $v_{1\lambda_1}(r, \theta)$, that is the first component, changes its sign and $v_{1\lambda_1}$ becomes unstable.

\end{proof}

\section{Proof of Propositions \ref{prop:cP} and \ref{prop:non-ex-TW}}\label{sec:low_bound}

\begin{proof}[Proof of Proposition \ref{prop:cP}]
\noindent
{\it Concerning $P$.}\\
Since $\Omega(t)=\tO+ (V,0)t$, the shape is stationary and using \eqref{eq:cin_TW_1} we have
\begin{equation*} 
	(V,0)\cdot\bm{n}-\bm{u}_{\bm{\mathcal C}}\cdot \bm{n}=-(\nabla P+\bm{u}_{\bm{\mathcal C}})\cdot \bm{n}=0 \qquad \mbox{in } \tGa,
\end{equation*}
where $\bm{u}_{\bm{\mathcal C}}$ is the velocity of the center of mass defined by \eqref{def:vit_center_mass}. Hence, 
\[
\bm{u}_{\bm{\mathcal C}}=(V,0)=V \bm{e}_x.
\]

Let us define $\tilde{\bm{u}}$ in $\tO$ by
\[
\tilde{\bm{u}} = -\nabla P  - \bm{u}_{\bm{\mathcal C}}:= \nabla \Phi  ,
\] 
with 
\[ 
\Phi =-P - Vx.
\] 
Then, we see that 
\begin{equation*} 
	- \int_{\tO}\Phi \Delta \Phi \dx \dy = \int_{\tO}|\nabla \Phi |^2 \dx \dy -\int_{\tGa}\Phi \nabla \Phi \cdot \bm{n}\dsigma.
\end{equation*}
Thanks to incompressibility, we deduce that $\Delta \Phi = 0$ in $\tO$. Moreover, using the stationary shape condition \eqref{eq:TW_shape} we get $\nabla \Phi \cdot \bm{n} = \tilde{\bm{u}} \cdot  \bm{n} = 0$ over $\tGa$.  Thus,
\begin{equation*} 
	0 = \int_{\tO}|\nabla \Phi |^2 \dx \dy =\int_{\tO}|\tilde {\bm{u}} |^2 \dx \dy.
\end{equation*} 
In other words, we find that the traveling state is characterized by a uniform flow of the entire fluid bulk, i.e. 
\[
\forall t \ge 0 \quad \bm{u}(\cdot, t) =-\nabla P(\cdot, t) =  \bm{u}_{\bm{\mathcal C}}= (V,0) \quad \textrm{ in }\tO ,
\]
hence 
\begin{equation*}\label{eq:P_TW}
	P(x,y)=p_1-V x  \qquad \mbox{ on } \tGa.
\end{equation*} 
Substituting $-\nabla P  = (V,0)$ in \eqref{eq:marqueur_TW_1} -- \eqref{eq:marqueur_bord_TW_1} it yields 
\begin{align*}
	 \cn \left(ac(V,0)+\nabla c\right)&= 0 &&  \mbox{ in } \tO,\\ 
	ac(V,0)\cdot \bm{n} +\nabla c \cdot \bm{n}&= 0&& \mbox{ on } \partial\tO . 
\end{align*}

\noindent
{\it Concerning $c$.}\\
The fact that
\begin{equation*}\label{eq:c_TW}
	c(x,y) = c_1e^{-aV x} \qquad \mbox{ in } \tO.
\end{equation*}
follows from Lemma \ref{lem:stat_sol}.\\

\noindent
{\it Concerning \eqref{eq:TW_shape}.}\\
Working in the reference frame of the moving domain, we substitute the expressions  of $P(x,y)$ and $c(x,y)$ in the normal force balance \eqref{eq:pression_bord} and we get that the curvature must satisfy \eqref{eq:TW_shape}.
\end{proof}

\begin{lemm}\label{lem:stat_sol}
	Any non negative solution of \eqref{eq:marqueur_TW_1} -- \eqref{eq:marqueur_bord_TW_1} is given by
	\begin{equation*} 
		c(x,y)=c_1 e^{-aV x},
	\end{equation*}
	where $c_1>0$.
\end{lemm}
\begin{proof}
	It is straightforward that $ e^{-aV x}$ is a solution of \eqref{eq:marqueur_TW_1} -- \eqref{eq:marqueur_bord_TW_1}. 
	Assume that $c(x,y)=C_1(x,y) e^{-aV x}$ is solution of \eqref{eq:marqueur_TW_1} -- \eqref{eq:marqueur_bord_TW_1}, multiplying \eqref{eq:marqueur_TW_1} by $C_1(x,y)$ and integrating by parts we obtain
	\[
	\int_{\tO} \nabla C_1(x,y) \cdot\left( \nabla \left(C_1(x,y) e^{-aV x} \right) +a(V,0) C_1(x,y) e^{-aV x} \right) \dx \dy=0,
	\]
	hence
	\[
	\int_{\tO} |\nabla C_1(x,y)|^2 e^{-aV x} \dx \dy =0,
	\]
	which implies $\nabla c_1=0$.
\end{proof}

\subsection{Graph formulation and proof of Proposition \ref{prop:non-ex-TW}} \label{sec:graph}

Given $\gamma, p_1, a,\chi>0$, we look for a set $\tO$, solution of \eqref{eq:TW_shape}, in the particular form:
\begin{equation}\label{eq:Omega0}
	\tO= \{(x,y)\in \xR^2\, ;\, x_L<x<x_R,\; -h(x)<y<h(x) \}
\end{equation}
for some positive function $h(x)$ defined on an interval  $(x_L,x_R)$ and satisfying:
\begin{subequations}\label{bc:h}
\begin{align} 
		h(x_L)&=h(x_R)=0, \label{bc:h-1}\\
		h'(x_L)&=+\infty,\label{bc:h-2} \\
		h'(x_R)&=-\infty.\label{bc:h-3}
\end{align}
\end{subequations}
Conditions \eqref{bc:h} have the following meaning. We want symmetric domains $\tO$, hence we impose \eqref{bc:h-1}. We also want smooth  boundaries, so we require \eqref{bc:h-2} -- \eqref{bc:h-3}. A possible domain is given in the following picture.
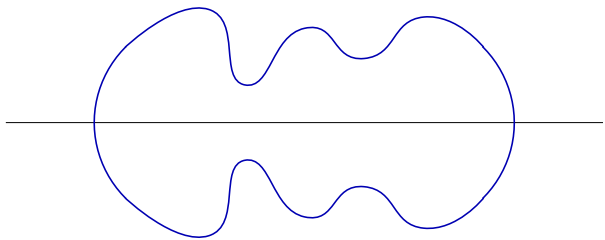
\begin{figure}[h]
\centering
\scalebox{.5}{
\begin{tikzpicture}[domain=-2*pi:pi]
\draw[-] (-3*pi,0)--(2*pi,0);

\draw [blue!70!black,very thick] plot [smooth, tension=1] coordinates { 
(-2*pi,2)  (-1.3*pi,3) (-pi,1) (-.5*pi,2.5) (0, 1.7)  (.5*pi,2.8) (pi, 2)};

\draw [blue!70!black,very thick] plot [smooth, tension=1] coordinates { 
(-2*pi,-2)  (-1.3*pi,-3) (-pi,-1) (-.5*pi,-2.5) (0, -1.7)  (.5*pi,-2.8) (pi, -2)};

\begin{scope}[transform canvas={xshift = -4.25cm},scale=2.85]
\draw [blue!70!black,very thick,domain=135:225] plot ({cos(\x)}, {sin(\x)});
\end{scope}
\begin{scope}[transform canvas={xshift = 1.1cm},scale=2.85]
\draw [blue!70!black,very thick,domain=-45:45] plot ({cos(\x)}, {sin(\x)});
\end{scope}
\end{tikzpicture}
}
\end{figure}

We will further fixe the invariance by translation in the $\bm{e}_x$ direction by requiring that 
\begin{equation}\label{eq:h'0} 
x_L<0<x_R, \quad	h'(0)=0.
\end{equation}

Our first task is to write equation \eqref{eq:TW_shape}, i.e.
\[
\gamma \kappa(x)= p_1 - V x - \chi f\left(  \frac{M}{\int_{\Omega_{0}}e^{-aVx'}\dx' \dy'  }e^{-aVx}\right) \qquad \textrm{  on }   \tGa,
\] 
with \eqref{eq:Omega0} -- \eqref{bc:h}  in  terms of the function $h(x)$. Since the boundary conditions \eqref{bc:h} concern the function $h$ and also its derivative $h'$, it is natural to identify a problem for which $h'$ is solution and then find the function $h$ by integration. For this reason, we change variables.
 
We first notice that the tangent vector $\bold{t}$, the normal vector $\bold{n}$ and the mean-curvature $\kappa$ of $\tO$ are defined by
\[
\bold{t}(x)=-\frac{(1,h'(x))}{\sqrt{1+(h'(x))^2}}, \quad \bold{n}(x)=\frac{(-h'(x),1)}{\sqrt{1+(h'(x))^2}} , \quad \kappa(x) = \displaystyle -\frac{h''(x)}{(1+(h'(x))^2)^{3/2}} .
\]
These quantities can be written easily using the function $Y(x)$ defined by 
\begin{equation} \label{eq:defY}
	Y(x)=\bold{e}_x \cdot \bold{n}(x) = \bold{n}_x (x) = \displaystyle -\frac{h'(x)}{\sqrt{1+(h'(x))^2}}.
\end{equation}
In particular, we have 
\[
\kappa (x)= Y'(x) ,
\]
which is consistent with Frenet's formula $\bold{n}'_x (x)= - \kappa \sqrt{1+(h'(x))^2} \, \bold{t}_x(x)$.\\
Then, equation \eqref{eq:TW_shape}  and condition \eqref{eq:h'0}   reduce to the following initial value problem:
\begin{equation} \label{eq:Y}
	\begin{cases}
		\gamma Y'(x) &= p_1 - Vx - \chi f \left( c_1 e^{-aVx}\right)  \quad \mbox{on } (x_L,x_R), \\
		Y(0)&=0,
	\end{cases}
\end{equation}
where 
\begin{equation}\label{def:normalisation_0}
	c_1:=\frac{M}{\int_{\tO}  e^{-aVx}   \dx\dy} .
\end{equation}  
Taking advantage of the fact that $e^{-aVx}$ only depends on the $x$-coordinate, we  rewrite $c_1$ in \eqref{def:normalisation_0} as
\begin{equation*} 
	c_1=\frac{M}{2\int_{x_L}^{x_R}  e^{-aVx}h(x)   \dx}.%
\end{equation*}  
Inverting \eqref{eq:defY}, it yields  
\begin{equation}\label{eq:inversionh}
h'(x)=-\displaystyle \frac{Y(x)}{\sqrt{1-Y^2(x)}},
\end{equation}
integrating by parts, we see that 
\[
c_1=\frac{aVM}{2\int_{x_L}^{x_R}  e^{-aVx} h'(x) \dx}=-\frac{aVM}{2\int_{x_L}^{x_R}  e^{-aVx} \frac{Y(x)}{\left(1-Y^2(x)\right)^{1/2}}\dx},
\]
 
Such a solution has to satisfy the following  properties.  
The boundary conditions \eqref{bc:h-2} -- \eqref{bc:h-3} imply that $Y$ has to verify 
\begin{align} \label{cond:Y}
	Y(x_L)=-1 \quad \mbox{ and } \qquad Y(x_R)=1.
\end{align}
We postpone the proof of the existence and uniqueness of a solution verifying \eqref{cond:Y} to Section \ref{sec:construction} (see Lemmas \ref{prop:x_L} and \ref{prop:x_R}).\\

Finally, using \eqref{eq:inversionh}, we recover the function $h$ 
and then we integrate this relation on $(x_L,x)$. \\
 
We recall that the 
boundary condition \eqref{bc:h-1} requires the function $Y$ to satisfy:
\begin{equation}\label{eq:condYh}
h(x_R)-h(x_L)=-	\int_{x_L}^{x_R} \frac{Y(x)}{\sqrt{1-Y^2(x)}} \dx = 0 .
\end{equation}

\begin{proof}[Proof of Proposition \ref{prop:non-ex-TW}]
We claim that $\kappa$ is a non-increasing function. Indeed, let $0<x_1-x_2\ll1$. Then, $\kappa(x_1)-\kappa(x_2)$ reads
\begin{align*}
\kappa(x_1)-\kappa(x_2)&=-V\pare{x_1-x_2}+\chi \pare{f\pare{c_1 e^{-aVx_2}}-f\pare{c_1 e^{-aVx_1}}}\\
&\simeq V\pare{\chi a  c_1 e^{-aV x_1}f'(c_1 e^{-aV x_1})-1}(x_1-x_2)\\
&<0
\end{align*}
by \eqref{def:no_TW}. 
It follows that the mean curvature is a decreasing function.  Suppose by contradiction that there exists a closed curve $\partial \tO$ whose mean curvature is $\kappa$. Since
\begin{align*}
\begin{pmatrix}
0\\0
\end{pmatrix}=\int_{\partial \tO} \kappa \bm{n} \dsigma
=\int_{\tO} \nabla \kappa \dx \dy
=2\begin{pmatrix}
\int_{x_L}^{x_R} h(x) \kappa'(x) \dx\\0
\end{pmatrix} , 
\end{align*}  
it yields a contradiction with $\kappa'(x) < 0$ and $h(x)\ge 0$, this proves Proposition~\ref{prop:non-ex-TW}.
\end{proof}

\section{Proof of Theorem \ref{thm:TW_intro}}\label{sec:construction}

This Section is devoted to the proof of Theorem \ref{thm:TW_intro} and it gives a constructive proof of the existence of 
traveling wave solutions of \eqref{eq:pb} when 
\[
\chi>\chi^*:= \frac{\pi  R_0^2}{aM f'\left(\frac{M}{\pi R_0^2}\right)}.
\]
Here, the constant $a\in (0,1]$ and $M$ is defined in \eqref{eq:M}. Furthermore, $R_0>0$ is the radius of $B_{R_0}$, where $\av{B_{R_0}}\equiv\av{\Omega_0}$. We recall that, assuming $\av{B_{R_0}}\equiv\av{\Omega_0}$, then Lemma \ref{lem:stat} assures that problem \eqref{eq:pb} admits a unique stationary solution which has the form
\[
\pare{\tilde{c},\wt{P}}=\pare{\frac{M}{\pi R_0^2},\frac{\gamma }{R_0}+\chi f\pare{\frac{M}{\pi R_0^2}}}=\pare{\tilde{c},\frac{\gamma }{R_0}+\chi f(\tilde{c})}.
\]
Hence, the above threshold on $\chi$ can be equivalently written as 
\begin{equation}\label{eq:chi*}
\chi>\chi^*=\frac{1}{a\tilde{c}f'(\tilde{c})}
\end{equation}
(see \eqref{def:chi_star} too).

More precisely, we will prove the following result.
\begin{theo}\label{thm:TW_intro_2}
	Assume that $f$ satisfies assumptions in \eqref{A}. Then, for all $a\in (0,1]$, $\gamma$, $p_1>\chi L$, there exists a one parameter family of traveling wave  solutions $(\tO_{\chi }^{p_1},V_{\chi }^{p_1})\in \xR^2\times\xR$  of \eqref{eq:pb} (with $\bold{u} =\bold{e}_x$), parametrized by $\chi$ such that
	\[
	\chi^*\le\chi<\infty
	\]
	and satisfying:
	\item[(i)] $V_{\chi}^{p_1}>0$ when $\chi>\chi^* $.  
	\item[(ii)] There exists
	$x_L,\,x_R\in\xR$, with $x_L<0<x_R$, and $0\le h$ such that the set $\tO_{\chi}^{p_1}$ is a convex set with $\mathcal{C}^{2,1}$ boundary of the form
	\[
	\tO_{\chi}^{p_1} = \{ (x,y)\, ;\, x_L<x<x_R,\; -h(x)<y<h(x)\},
	\] 
	with  
	$h$ satisfying $ h'(0)=0$.
	\item[(iii)] The normal vector $\bold{n}$ is the vertical vector $(0,1)$ at the point $m=(0,h(0))\in\partial \tO_{\chi}^{p_1}$, and the curvature is given by
	\begin{equation}\label{eq:p1}
		\kappa(m)=\frac{p_1}{\gamma}-\frac{\chi}{\gamma} f(c_1). 
	\end{equation}
	\item[(iv)] 	
	Assume that $f$ statisfies the additional assumption \begin{equation}\label{hyp:supplementaire}
		s f'(s) \le \tilde c f'(\tilde c) \q\forall s\in \xR^+.
	\end{equation}
	Then, $V_{\chi}^{p_1} \to V_{\chi^*}^{p_1}:=0$ when $\chi\to \chi^*$.
	Furthermore, $\Omega_{\chi^* }^{p_1} $ is a disk of radius $R_0$ and the following relation holds   
	\begin{equation}\label{eq:compati_TW}
		\gamma = R_0 \left( p_1-\frac{\pi R_0^2 f\left(\frac{M}{\pi R_0^2}\right)}{aM f'\left(\frac{M}{\pi R_0^2}\right)} \right).
	\end{equation}
\end{theo}

\begin{rema}[On Theorem \ref{thm:TW_intro_2}]
	Property (i) guarantees in particular that we are constructing non-trivial traveling wave solutions (i.e. not stationary solutions). Property (ii) fixes the natural invariance by translation of the model. Properties (iii) and (iv) relate the value of the parameter $p_1$ to some geometric property of $\tO_{ \chi }^{p_1}$ namely its area is $|\tO_{ \chi}^{p_1}|=2 \int_{x_L}^{x_R} h(x)\dx$. It proves that each value of $p_1$ yields of different traveling wave, see \eqref{eq:p1}. In particular, if we let $\chi\to\chi^*$, this relation becomes
	\begin{equation}\label{eq:compati_2_TW}
		p_1= \frac{\gamma}{R_0}+\frac{\pi R_0^2 f\left(\frac{M}{\pi R_0^2}\right)}{a M f'\left(\frac{M}{\pi R_0^2}\right)},
	\end{equation}
	and it	suggests that increasing values of $p_1$ correspond to sets with decreasing volume. It is the case for the prototype example \eqref{eq:prototype_example} of function $f$. 
\end{rema}

\begin{rema}
	A prototype example of a function satisfying \eqref{A} and \eqref{hyp:supplementaire} is 
	\begin{equation}\label{eq:prototype_example}
		f(x)= \frac{ Lx}{\tilde c + x}.
	\end{equation}
\end{rema}

In this Section we prove the following Proposition, which implies Theorem~\ref{thm:TW_intro_2}.
\begin{prop}\label{prop:Y}
	Assume that $f$ satisfies the assumptions in \eqref{A}. Then, given $a$, $M$ and $R_0$ as before, and letting $\chi$ as in \eqref{eq:chi*},	and
	\begin{equation}\label{eq:pchiL}
		p_1  > L\chi,
	\end{equation}
	there exists $V>0$, $x_L,\,x_R\in \xR$ such that the solution $Y(x)$ of \eqref{eq:Y} satisfies the conditions \eqref{cond:Y} and \eqref{eq:condYh}. Furthermore, for a fixed $p_1$, the speed $V$ converges to $0$
	when $\chi$ approaches the critical value $\chi^*$.
\end{prop}

\begin{proof}[Proof of Theorem \ref{thm:TW_intro_2}]
{\sl Proof of (i).}\\
\noindent
It follows from Proposition \ref{prop:Y}.\\

\noindent
{\sl Proof of (ii).}\\
\noindent
We recall the notation in Section \ref{sec:graph} so that, thanks also to  Proposition \ref{prop:Y}, we set 
	\[
	h(x) = \int_{x_L}^x -  \frac{Y(x')}{\sqrt{1-Y^2(x')}}\dx' \qquad x\in [x_L,x_R] .
	\]
We also recall the definition of the  set $\Omega_{\chi}$ by \eqref{eq:Omega0}, and that  \eqref{eq:condYh} implies that $h(x_L)=h(x_R)=0$.

	The boundary conditions \eqref{cond:Y} imply that the normal vector is continuous since it achieves the values $(-1,0)$ and $(1,0)$ continuously at the extremal points $(x_L,0)$ and $(x_R,0)$. This means that $\partial \Omega_{\chi}$ is at least $C^1$.\\ 
	Furthermore, we have
	$ \kappa(x) = Y'(x) $
	hence Proposition \ref{prop:convex} implies 
	\[
	0\leq \gamma \kappa \leq p_1 +4\chi  +\sqrt{\left(  p_1 -4\chi \right)^2+2V}   \mbox{ on } \partial \Omega_{\chi} ,
	\] 
	thus $\Omega_{ \chi}$ is convex and $\partial \Omega_{ \chi}$ is $C^{1,1}$. 
	In turns, \eqref{eq:Y} can be used to show that $Y'$, and therefore $\kappa$ is Lipschitz continuous so that the boundary $ \partial \Omega_{ \chi}$ is $\mathcal{C}^{2,1}$ and satisfies \eqref{eq:TW_shape}.\\
	
	\noindent
{\sl Proof of (iii).}\\
\noindent
	Finally, since $Y(0)=0$, we get that $h'(0)=0$, and from \eqref{eq:Y}  it follows that the mean-curvature of $\pa \Omega_{ \chi}$ at the point $(0,h(0))$ is given by $\gamma Y'(0)=p_1-\chi f(c_1)$.
	When $ \chi \to \chi^*$, we have $V\to 0$ and so $Y$ converges to the solution of 
	\[
	\gamma Y'=p_1-\frac{f(\tilde{c})}{a \tilde{c} f'(\tilde{c})}. 
	\]
	In particular $\Omega_{\chi}$ converges to the set with constant mean curvature $\frac1\gamma \left(p_1-\frac{f(\tilde{c})}{a \tilde{c} f'(\tilde{c})}\right)$, that is the ball $B\left( O,\frac{\gamma}{ p_1-\frac{f(\tilde{c})}{a \tilde{c} f'(\tilde{c})}}\right)$.\\
	Recalling the definition \eqref{eq:stat} of	$\tilde{c} =\frac{M}{\pi R_0^2}$, we obtain \eqref{eq:compati_TW} and \eqref{eq:compati_2_TW}.
\end{proof}

The remainder of this section is devoted to the proof of Proposition \ref{prop:Y}. We first find the set of parameters for which the  points $x_L$ and $x_R$ satisfying the condition \eqref{cond:Y} exist.  This is done in Section \ref{sec:x_L}.  The existence of the solution $Y$ is then given by regularity properties of the problem \eqref{eq:Y}. We then fix the parameters $p_1$, $c_1$ and $\chi$ such that $x_L$ exists, by considering the solutions of  \eqref{eq:Y} for those values of $V$ such that also $x_R$ exists, we will prove, in Section \ref{sec:x_L}, that there exists a value $V^*>0$ (depending on  $p_1$, $c_1$ and $\chi $) such that the condition \eqref{eq:condYh} is verified.

\subsection{Proof of Proposition \ref{prop:Y}}\label{sec:x_L}

The difficulty in proving an existence result to \eqref{eq:Y} relies on in the fact that the domain $(x_L,x_R)$ is one of the unknown of the problem. Indeed, we need $(x_L,x_R)$ such that condition \eqref{cond:Y} is satisfied. \\
A preliminary step consists in dealing with the intermediate problem
\begin{equation} \label{eq:Y2}
	\begin{cases}
		\gamma Y'(x) &= p_1 - Vx - \chi f \left( c_1 e^{-aVx}\right)  \quad \mbox{on } (z_L,z_R), \\
		Y(0)&=0,
	\end{cases}
\end{equation}
for
\[
c_1=-\frac{aVM}{2\int_{z_L}^{z_R}  e^{-aVx} \frac{Y(x)}{\left(1-Y^2(x)\right)^{1/2}}\dx},
\]
and given
\[
z_L<0<z_R.
\]
We can thus apply the Cauchy-Lipschitz Theorem to say that there exists a unique solution $Y\in C^2([z_L,z_R])$. 

For the sake of clarity, we decompose the proof into two steps. First, we prove that we can extend the interval $(z_L,z_R)$ to  $(x_L,x_R)$ such that \eqref{cond:Y} holds (see Lemmas \ref{prop:x_L} and \ref{prop:x_R}), and then the existence of $V>0$ (see Subsection \ref{sec:V}). For both steps, we will assume \eqref{eq:chi*} and \eqref{eq:pchiL}, i.e.
\begin{equation}\label{eq:chip}
\chi \ge  
\chi^*= \frac{\pi  R_0^2}{aM f'\left(\frac{M}{\pi R_0^2}\right)} \qquad  \textrm{ and } \qquad  p_1>\chi L.
\end{equation}

\subsubsection{The existence of $x_L,\,x_R$}

We break the proof in two steps. We first give a sufficient condition for the existence of $x_L$.

\begin{lemm}\label{prop:x_L}
	Assume that $f$ satisfies the assumptions in \eqref{A}, and that \eqref{eq:pchiL} holds. For all $V\ge 0$, there exists $x_L<0$ such that the solution of \eqref{eq:Y} satisfies  $-1<Y(x)<0$ for all $x \in (x_L,0)$ and $Y(x_L)=-1$. Moreover,  $x_L$ is such that 
	\[
	\frac{\left(  p_1 -L\chi \right) -\sqrt{\left(  p_1 -L\chi \right)^2+2\gamma V}}{V} < x_L < 0,
	\]
	and we have 
	\begin{equation}\label{eq:Y'x_L}
		Y'(x)\ge \frac{p_1-L\chi}{\gamma}, \quad \forall x \in [x_L,0] .
	\end{equation}
\end{lemm}

\begin{proof}
	From \eqref{eq:pchiL}, we see that
	\[
	Y'(0)=\kappa(0)=\frac{p_1-\chi f(c_1)}{\gamma} \ge  p_1-L\chi  >0.
	\]
	Moreover, since $Y(0)=0$ we let 
	\begin{equation*} 
		x_L = \inf\{ z_L<0\, ;\, Y(x)\in (-1,0) \mbox{ for all } x\in (z_L,0) \}. 
	\end{equation*}
Note that  $\set{z_L<0\, ;\, Y(x)\in (-1,0) \mbox{ for all } x\in (z_L,0)}\ne\emptyset$ thanks to the regularity of $Y$. We possibly have $x_L=-\infty$ if $Y(x)>-1$ for all $x<0$. Hence, we need to show that $x_L>-\infty$ and that $Y(x_L)=-1$.

	Recalling assumptions \eqref{A}, we see that
	\begin{equation*} 
		0\le f(c_1e^{-aV_1x})\le L, \quad \forall x \in \xR,
	\end{equation*} 
	thus from \eqref{eq:Y2} we deduce 
	\begin{equation} \label{diseq:Yp_below2}
		p_1 -L\chi - V \, x  \le 	\gamma Y'(x) \le   p_1 - V \, x  ,  \quad \forall x \in \xR .
	\end{equation}
	Since $Y(0)=0$, we have	
	\begin{align} \label{eq:above_Y}
	-V \frac{x^2}{2} +  p_1  x	\le \gamma Y(x) \le  -V \frac{x^2}{2} +\left(  p_1 -L\chi  \right)  x, \quad \forall x \in \xR^-.
	\end{align}
	The right-hand side of the previous inequality is negative, monotone increasing on $\xR^- $ and it converges
	towards $-\infty$ as $x \to -\infty$. Since $Y (x_L) \ge  -1$, by definition of $x_L$, it follows that
	$x_L >-\infty$ and \eqref{eq:above_Y} implies that $Y(x_L)<0$ so that we must have $Y(x_L)=-1$.
\end{proof}

\begin{rema}\label{rem:hyp:xL}
	Note that the assumption $p_1\ge \chi L$ is a sufficient condition for the existence of $x_L$. Indeed, we may assume that $p_1\ge \chi f(c_1)$. However, the disadvantage of such an assumption is that the quantity $ f(c_1)$ depends on $V$. More generally, without any assumption on $p_1$ we can prove that there exists $x_L<0$ such that $Y(x_L)=-1$ by using \eqref{eq:above_Y}. However, in this latter case we can not assure that \eqref{eq:Y'x_L} holds.     
\end{rema}

Next, we give a sufficient condition for the existence of $x_R$.

\begin{lemm} \label{prop:x_R}
	Assume that $f$ satisfies the assumptions in \eqref{A}, and that \eqref{eq:pchiL} holds. 
Then, 	there exists  
\begin{equation}\label{eq:vmax}
V_{\t{max}}\in \bra{\frac{p_1-L\chi}{2\ga},\frac{p_1^2}{2\ga}}
\end{equation}
such that, for all $V\in (0,V_{\t{max}})$, there exists $x_R>0$ such that the solution of \eqref{eq:Y} satisfies 
\[
0<Y(x)<1\q\t{for all}\q x \in (0,x_R),
\]
and 
\[
Y(x_R)=1.
\]
Moreover, $x_R$ is such that
	\begin{equation*} 
		0<x_R\le \frac{p_1 }{V},
	\end{equation*}
	and 
	\begin{equation}\label{eq:Y'R} 
		\begin{cases}
			Y'(x_R) >0, \quad \mbox{ if } V <V_{\max},\\ 
			%\qquad  
			Y'(x_R)  = 0, \quad  \mbox{ if } V = V_{\max},
		\end{cases} 
	\end{equation}
	and
	\begin{equation}\label{eq:Y'R2} 
		Y'(x) \geq  Y'(x_R) \quad \forall x\in[0,x_R].
	\end{equation}
\end{lemm}

\begin{proof}
	Since 
\[
\gamma Y'(0)=p_1-\chi f(c_1) >0,
\]
we can 	define $\bar x>0$ by 
	\begin{equation*} 
		\bar x :=\sup \{z_R> 0 \, ; \, Y'(x)\ge 0 \textrm{ for all } x \in (0,z_R)   \}.
	\end{equation*}
In particular, since by \eqref{A}
\[
\ga Y'(x)\ge p_1-Vx-\chi L,
\]
we have $Y'(x)\ge0$ if $x\le \frac{p_1-L\chi}{V}$. Hence, by definition,
\[
\bar{x}\ge \frac{p_1-L\chi}{V}
\]
Furthermore, from \eqref{diseq:Yp_below2} evaluated in $\bar{x}$ and \eqref{A}, it follows that 
\[
p_1-L\chi\le V\bar{x}+\ga Y'(\bar{x})\le p_1, 
\]
hence, recalling \eqref{eq:pchiL} and thanks to the definition of $\bar{x}$, we also have
	\[
 \bar x \le  \frac{p_1}{V}. 
	\] 
Integrating \eqref{eq:Y} in $x$  and using $Y(0)=0$, we have
	\begin{align*} 
		- V \, \frac{x^2}{2} +\left(p_1 - L\chi \right) x \le 	\gamma Y(x) \le   - V \, \frac{x^2}{2} +  p_1  x,  \qquad \forall x \ge 0.
	\end{align*}
Recalling that $Y$ is increasing in $(0,\bar{x})$, we get
	\begin{equation*} 
		 \frac{\pare{p_1-L\chi}^2}{2\ga V}\leq	Y(\bar x)  =\sup_{x\in (0,\bar{x})} Y(x)\leq \frac{p_1^2}{2\gamma V}.
	\end{equation*} 
In view of the previous computations, a sufficient condition for the existence of $x_R$ is $Y(x_R)=1$ 
and, in such a case, it holds that
	\begin{equation*} 
		x_R<\bar x.
	\end{equation*}
Define
	\begin{equation*} 
		V_{\max} := \sup\{ V_0 \, ;\, Y(\bar x) >1, \,  \forall  V \in [0,V_0)\}.
	\end{equation*}
	We first see that 
	\[
	\frac{\left( p_1 - L\chi  \right) ^2}{2\gamma }\le V_{\max} \le \frac{p_1^2}{2\gamma },
	\]
	and by continuity with respect to $V$, when $V=V_{\max}$, we have $Y(\bar x) \geq 1$. Furthermore, if $Y(\bar x)>1$, then there exists $\delta >0$ such that $\sup Y  >1$ for $V\in [V_{\max} ,V_{\max} +\delta)$ which contradicts the definition of $V_{\max}$. Consequently, $Y(\bar x)=1$ when $V=V_{\max}$ and so $x_R=\bar x $ and $Y'(x_R)=0$. 
\end{proof}

\begin{rema}
	As in Remark \ref{rem:hyp:xL}, we may weaken the assumption made on $p_1$ for the existence of $x_R$ such that $Y(x_R)=1$. However in such a case we may not assure that \eqref{eq:Y'R} and \eqref{eq:Y'R2} hold true.  
\end{rema}

From the two previous Lemma we deduce the following result. 
\begin{prop}\label{prop:convex}
	Assume that $f$ satisfies the assumptions in \eqref{A}, that \eqref{eq:pchiL} holds, and that $V\in (0,V_{\rm max})$ for $V_{\t{max}}$ as in \eqref{eq:vmax}. 
	Then $Y(x)$ the solution of \eqref{eq:Y} satisfies 
	\[
	0\leq \gamma Y'(x) \leq L\chi  +\sqrt{\left(  p_1 -L\chi \right)^2+2V}  \quad \forall x\in (x_L,x_R).
	\]
\end{prop}

\begin{proof} 
Lemma \ref{prop:x_L} implies that $Y'(x)\geq 0$ on $[x_L,0]$ and Lemma \ref{prop:x_R}  implies that $Y'(x)\geq 0$ on $[0,x_R]$. 
	Next, Equation \eqref{eq:Y} implies
	\[
	\gamma Y'(x) \leq p_1 - V x_L,
	\]
	and the upper bound follows from Lemma \ref{prop:x_L}.
\end{proof}

\subsubsection{The existence of $V$}\label{sec:V}

We now turn to the proof of the existence of $V$. From now on, we denote by $Y(x,V)$, $x_L(V)$, $x_R(V)$ the solutions of \eqref{eq:Y}, \eqref{cond:Y} for all $V\in (0,V_{\max})$. Moreover, since we can read $c_1$ as a function of $V$, we rewrite it as
\begin{equation*} 
	c_1(V)= 
	-\frac{aVM}{2\int_{x_L(V)}^{x_R(V)}  e^{-aVx} \frac{Y(x,V)}{\left(1-Y^2(x,V)\right)^{1/2}}\dx}.
\end{equation*}

To end the proof of Proposition \ref{prop:Y}, we must show that if $\chi >  \chi^*$, then there exists  $V\in (0,V_{\max})$ such that \eqref{eq:condYh} is satisfied.
We introduce the function
\begin{equation}\label{eq:Gdef}
	G(V) := \int_{x_L(V)}^{x_R(V)} \frac{Y(x,V)}{\sqrt{1-Y^2(x,V)}} \dx,
\end{equation}
and first prove the following result.
\begin{prop}\label{prop:c}
	The function $G:[0,V_{\mathrm{max}}) \to \xR$ defined by \eqref{eq:Gdef} is continuous and satisfies
	\[
	G(0)=0\q\t{and}\q G(V) \to +\infty \quad \mbox{ as } \quad V\to V_{max}. 
	\]
	Furthermore, when  $\chi >  \chi^*$, then it holds
	\[
	G(V)<0 \quad \mbox{ for } \quad 0<V \ll 1.
	\] 
\end{prop}

\begin{proof} 
	\noindent
{\it Continuity of $V\mapsto G(V)$.} \\
	Since $x_L(V)$ and $x_R(V)$ are determined by the conditions 
	\[
	Y(x_L,V)=-1\quad \mbox{ and } \quad  Y(x_R,V)=1,
	\]
	recalling equations \eqref{eq:Y'x_L} and  \eqref{eq:Y'R}, we can apply the Implicit Function Theorem to get that $V\mapsto x_L(V) $ and $V\mapsto x_R(V) $ are continuous functions.
	
	To prove the continuity of $G$, we now consider a sequence $V_n$ of positive numbers such that $V_n \to V>0$.   
	We fix $\delta>0$.
	The continuity of $x_L$ and $x_R$ implies that for large enough $n$:
	\[
	x_L(V_n)\leq x_L(V)+\delta\leq x_L(V_n)+2\delta,
	\] 
	and
	\[ 
	x_R(V_n) \geq x_R(V)-\delta\geq x_R(V_n)-2\delta.
	\]
	We claim that  $Y(x,V_n) \to Y(x,V)$ uniformly in $[x_L(V),x_R(V)]$. Indeed, differentiating equation \eqref{eq:Y} with respect to $V$, we find that the function $Z:x\mapsto \partial_V Y(x,V)$ solves
	\begin{align*}
			\gamma Z'(x) &=-x+\chi \left( a c_1x-\partial_V c_1 \right) e^{-aVx} f'\left(c_1e^{-aVx}\right) \qquad \mbox{ on } (x_L,x_R),\\
			Z(0)&=0,
	\end{align*} 
	with
	\begin{equation*} 
		\partial_Vc_1(0)=0.
	\end{equation*}

Then, we have
	\begin{equation}\label{bound:Y}
	|Y(x,V_n)|\leq 1-\eta \qquad \mbox{ in } (x_L(V)+\delta,x_R(V)-\delta),
	\end{equation}
	for some $\eta>0$ and $n$ large enough.
	
	We now write
	\begin{align*}
		G(V_n)
		& =\int_{x_L(V_n)}^{x_R(V_n)} \frac{Y(x,V_n)}{\sqrt{1-Y^2(x,V_n)}} \dx \\
		& =\int_{x_L(V_n)+\delta}^{x_R(V_n)-\delta} \frac{Y(x,V_n)}{\sqrt{1-Y^2(x,V_n)}} \dx
		+ \int_{x_L(V_n)}^{x_L(V_n)+\delta} \frac{Y(x,V_n)}{\sqrt{1-Y^2(x,V_n)}}\dx \\
		& \quad 
		+\int_{x_R(V_n)-\delta}^{x_R(V_n)} \frac{Y(x,V_n)}{\sqrt{1-Y^2(x,V_n)}}\dx .
	\end{align*}
	The bound \eqref{bound:Y} implies that 
\[
\int_{x_L(V_n)+\delta}^{x_R(V_n)-\delta} \frac{Y(x,V_n)}{\sqrt{1-Y^2(x,V_n)}} \dx\to \int_{x_L(V)+\delta}^{x_R(V)-\delta} \frac{Y(x,V)}{\sqrt{1-Y^2(x,V)}} \dx \qq\t{for }n\to \infty.
\]
	Next, using that $Y$ verifies \eqref{eq:Y'x_L}, we have  that
\[
\av{Y(x)}\le 1\q \t{and}\q 1+Y(x) \geq C(x-x_L(V_n))
\]
 for $x\in [x_L(V_n),x_L(V_n)+\delta]$.	It follows that the second term satisfies
	\begin{align*}
		\left| \int_{x_L(V_n)}^{x_L(V_n)+\delta} \frac{Y(x,V_n)}{\sqrt{1-Y^2(x,V_n)}}\dx \right|
		&  \leq  \int_{x_L(V_n)}^{x_L(V_n)+\delta} \frac{1}{\sqrt{C(x-x_L(V_n))}} \dx  \leq C' \delta^{1/2},
	\end{align*}
	where $C,C'$ are positive constants.
	
	Using \eqref{eq:Y'R2}, we obtain a similar bound for the third term and we get that
	\[
	\lim_{n\to \infty} G(V_n) = \int_{x_L(V)+\delta}^{x_R(V)-\delta} \frac{Y(x,V)}{\sqrt{1-Y^2(x,V)}}\dx  +\mathcal O(\delta^{1/2})
	=G(V)  +\mathcal O(\delta^{1/2})
	\]
	from which we deduce the continuity of $G$.\\

	\noindent
{\it Proving that $G(0)=0$.}\\
	Recalling that $c_1(0)=\tilde{c}= \frac{M}{|\Omega_0|}$, it is easy to check that
\[
\gamma Y(x,0)=\left(p_1-\chi f(\tilde{c})\right) x,\qq x_L(0)=-\frac{\gamma}{ p_1-\chi f(\tilde{c}) },\qq x_R(0)=\frac{\gamma}{  p_1-\chi f(\tilde{c}) },
\]
	hence by parity
	\begin{equation*} 
		G(0)=\int_{-\frac{\gamma}{p_1-\chi f(c_1)}}^{\frac{\gamma}{p_1-\chi f(\tilde{c}) }} \frac{\left(p_1-\chi f(\tilde{c})\right) x}{\left(\gamma ^2-\left(p_1-\chi f(\tilde{c})\right) ^2 x^2\right)^{1/2}} \dx = 0.
	\end{equation*}

\noindent
{\it On the sign of $G$.}\\
	Now, we define the function $H: (-1,1) \to \xR$ by 
	\[
	H(y) = \frac{y}{\sqrt{1-y^2}}. 
	\]
	For $0<V<V_{\mathrm{max}}$, 
	we have
	\[
p_1 -\chi f(\tilde c)=	\gamma \partial_x Y(x,V) +Vx +\chi \left( f\left(c_1 e^{-aVx} \right) -f(\tilde c)\right) ,
	\]
	hence
	\begin{align}
		G(V)&=\int_{x_L(V)}^{x_R(V)} H(Y(x,V)) \dx\nonumber  \\
		&= \frac{1}{p_1-\chi f(\tilde c)} \int_{x_L(V)}^{x_R(V)} \left(p_1-\chi f(\tilde c) \right) H(Y(x,V))  \dx \nonumber\\
		&= \frac{1}{p_1-\chi f(\tilde c)}  \int_{x_L(V)}^{x_R(V)} \left[\gamma \partial_x Y(x,V) + Vx +  \chi  \left( f\left(c_1 e^{-aVx} \right) -f(\tilde c)\right) \right]  H(Y(x,V))  \dx  \nonumber   \\
		& = \frac{1}{p_1-\chi f(\tilde c)}  \int_{x_L(V)}^{x_R(V)}\left[ Vx +  \chi  \left( f\left(c_1 e^{-aVx} \right) -f(\tilde c)\right) \right]  H(Y(x,V)) \dx,  \label{eq:signG}
	\end{align}
 	since
\[
\int_{x_L(V)}^{x_R(V)} \partial_x Y(x,V)  H(Y(x,V))  \dx=0.
\]
 Note that, for all $V \in [0,V_{\mathrm{max}})$, we already know that
	\begin{equation} \label{sign:H}
		\begin{cases}
			H(Y(x,V))>0  \qquad \mbox{for all } 0 < x < x_R(V), \\
			H(Y(x,V))<0  \qquad \mbox{for all } x_L(V) < x < 0,
		\end{cases}
	\end{equation}
	so we need to determine  the sign of the function
\begin{equation}\label{eq:signe_Y}
 W_{\chi}(x):=Vx +  \chi  \left( f\left(c_1 e^{-aVx} \right) -f(\tilde c)\right).
\end{equation}
We split this study w.r.t. the cases $V\ll1$ and $V\to\infty$.\\

		\noindent{\it Behavior of $G$ for  $V\ll1$.}\\
		For $0<V \ll 1$, using that $\partial_V c_1(0)=0$ and that $c_1(0)=\tilde{c}$, we have
	\[
f(c_1(V)e^{-aVx})=f(\tilde{c})-ax\tilde{c}f'(\tilde{c})V+ o(V),
	\]
	hence, in such a case, \eqref{eq:signG} becomes
	\begin{align*}
		G(V)&
		= \frac{1}{p_1-\chi f(\tilde{c}}\int_{x_L(V)}^{x_R(V)} Vx\left( 1-  a\chi \tilde{c} f'(\tilde{c}) \right) H(Y(x,V))  \dx     +  o(V)  \\
		& = V\frac{1-  a\chi \tilde{c} f'(\tilde{c})}{p_1-\chi f(\tilde{c})}  \int_{x_L(V)}^{x_R(V)}  x H(Y(x,V)) \dx +  o(V) .
	\end{align*}
	Using \eqref{sign:H} we deduce that for $0<V \ll 1$ the sign of $G(V)$ is that of 
	\[
	\left(p_1-\chi f(\tilde{c})\right)\left( 1-  a\chi \tilde{c} f'(\tilde{c}) \right).
	\] 
	The condition \eqref{eq:chip} implies that it is negative.\\
	
	\noindent{\it  Behavior of $G$ when $V \to V_{\mathrm{max}}$.} \\
Equation \eqref{eq:Y'R} gives:
	\[
	\pa_x Y(x_R(V_{\mathrm{max}}),V_{\mathrm{max}}) =0.
	\]
	Recalling that  $\alpha := 2	\pa_{xx} Y(x_R(V_{\mathrm{max}}),V_{\mathrm{max}}) \le 0$, we deduce 
	\begin{equation} \label{N_cmax}
		Y(x,V_{\mathrm{max}}) = 1 + \alpha (x-x_R(V_{\mathrm{max}}))^2 + \mathcal O \left( |x-x_R(V_{\mathrm{max}})|^3\right) \quad\mbox{ as } x \to x_R.
	\end{equation}
	Thus,  we get 
	\begin{align*}
		1-Y^2(x,V)  
		&= -2\alpha (x-x_R(V))^2 +\mathcal O \left( |x-x_R(V)|^3\right) ,
	\end{align*}
	leading to
	\begin{equation} \label{D_cmax}
		\sqrt{1-Y^2(x,V)} = \sqrt{2|\alpha| } \, |x-x_R(V)|+\mathcal O \left( |x-x_R(V)|^{3/2}\right) .
	\end{equation}
	We notice that for all $\varepsilon>0$ the function $G$ can be written by
	\begin{align*}
		G(V) 
		&= \int_{x_L(V)}^{x_R(V)-\varepsilon} \frac{Y(x,V)}{\sqrt{1-Y^2(x,V)}} \dx 
		+  \int_{x_R(V)-\varepsilon}^{x_R(V)} \frac{Y(x,V)}{\sqrt{1-Y^2(x,V)}} \dx.
	\end{align*}
	The first right hand side is always finite, while the second right hand side by \eqref{N_cmax} and \eqref{D_cmax} for $V \to V_{\mathrm{max}}$ we get that 
	\[
	\int_{x_R(V)-\varepsilon}^{x_R(V)} \frac{Y(x,V)}{\sqrt{1-Y^2(x,V)}} \dx \simeq \int_{x_R(V)-\varepsilon}^{x_R(V)} \frac{1}{\sqrt{2} \, |x-x_R(V)|} \dx = +\infty.
	\]
	Therefore, for $V \to V_{\mathrm{max}}$ we get that 
	\begin{equation*} 
		G(V) \to +\infty,
	\end{equation*}
	which completes the proof of Proposition \ref{prop:c}.
\end{proof}

Proposition \ref{prop:Y} is a consequence of the following result.
\begin{coro}\label{cor:d}
	Let \eqref{eq:chip} be in force. Then,  there exists $V_{\chi}\in (0,V_{\mathrm{max}})$ such that $G(V_\chi)=0$.
	Furthermore, assume in addition that $f$ satisfies \eqref{hyp:supplementaire}.
	Then, $V_\chi\to 0 $ as $\chi \to \chi^*$.
\end{coro}

\begin{proof}
	The Corollary follows from Proposition  \ref{prop:c}.
	Indeed, for any $\chi>\chi^*$ the Intermediate Value Theorem gives the existence of $V_\chi>0$ such that $G(V_\chi)=0$.\\
	In order to show that 
	$V_\chi\to 0 $ as $\chi \to \chi^*$, we note that if we consider $G$ as a function of $\chi$ and $V$ (instead of $V$ only), then the continuity with respect to $\chi$ can be proved similarly to that with respect to $V$. Consider the situation where $\chi\to \chi^*$ along any subsequence such that $V_\chi\to V^*$. Such a subsequence exists since $0\leq V\leq V_{\mathrm{max}}\leq  \frac{ p_1  ^2}{2\gamma }$.
	In such a case we have $0=G(\chi,V_\chi) \to G(\chi^*,V^*)$.
	It remains to prove that $ G(\chi^*,V)>0$ for all $V>0$ from which it would follow that $V^*=0$ and hence the whole sequence $V_\chi$ converges to zero.
	
	Recalling the definition \eqref{eq:signe_Y} of $W$, we see that 
	\[
	W_{\chi*}'(x) = V \left(1- \frac{c_1 e^{-aVx} f'\left(c_1 e^{-aVx} \right) }{\tilde c f'(\tilde c)} \right). 
	\] 
	 Consequently, if $f$ satisfies the additional assumption  \eqref{hyp:supplementaire}, then $W'(x)\ge 0$ for all $x \in \xR$ and hence $H(Y(x,V)) W(x) \ge 0$ for all $x \in \xR$.

\end{proof}

\section{Proof of Theorem {\ref{thm:TW:bif}}}\label{sec:thm:TW:bif}

In this section, we show that a bifurcation of traveling wave solutions occurs from the family of radially symmetric steady states $B_{R_0}$. This bifurcation is determined by the following parameters: the size of the cell $R_0$, the active parameter $\chi$, the constant $M$ defined in \eqref{eq:M}, $a$ and the function $f$. It is convenient to  choose $\chi$ as the bifurcation parameter in the bifurcation conditions. More precisely, we will prove the following result which implies  Theorem {\ref{thm:TW:bif}. 
	\begin{theo}\label{thm:TW:bif_2}
 	Assume that $f$ satisfies the assumptions in \eqref{A}. Then, problem \eqref{eq:pb} has a branch of traveling wave solutions $(\tO_{\chi }^{R_0},V_{\chi }^{R_0})$ with volume $|B_{R_0} |$  
 	bifurcating from the radial solution $B_{R_0}$ at $\chi =\chi^*$.
 \end{theo}

 The next result characterizes the structure of the bifurcating branch of the traveling wave.

 \begin{theo}\label{thm:TW:bif-CR}
 	Assume that $f$ satisfies the assumptions in \eqref{A}. Then, the  bifurcation solution has the following form: 
 	\begin{equation*}
 		\begin{split}
 			\chi(s)&=\frac{\pi  R_0^2}{aM f'\left(\frac{M}{\pi R_0^2}\right)} + \eta s^2 +  o(s^2), \\ 
 			V(s) &= s + o(s),
 		\end{split}
 	\end{equation*}
 	where $\eta \in \xR$ and the parameter $s$ takes values in an interval $(-\delta,\delta)$.
 \end{theo}
 
 \begin{rema}
 	By lack of uniqueness, we cannot claim that the $V$ function in Theorem \ref{thm:TW:bif-CR} is the same as the previous existence Theorems.
 \end{rema}

Since the disk is a solution of the system  \eqref{eq:pression_TW_1} -- \eqref{eq:marqueur_bord_TW_1} with zero bulk velocity $V=0$, our aim is to seek for other solutions in the form of a perturbation of the disk of radius $R_0$. We seek for those domain $\Omega_0$ of the form
\begin{equation*} 
	\Omega_0 = \{(r,\theta)\, :\, 0\leq r < R_0 + \rho (\theta) \mbox{ and 
	} \theta \in [-\pi,\pi] \},
\end{equation*}
where the function $\rho:\xR \to (-R_0,\infty)$ is $2\pi$-periodic and such that 
\begin{equation}\label{eq:preservation-area}
\int_{-\pi}^\pi \left((R_0+\rho(\theta))^2-R_0^2 \right)\dtheta = 0
\end{equation}
(this condition guarantees that $|\Omega_0|=|B_{R_0}|$).

Furthermore, since we look for traveling wave propagating in the $x$-direction, we restrict ourselves (as in the previous section) to domain $\Omega_0$ that are symmetric with respect to the $y$-axis. We thus introduce the functional space: 
\begin{equation*} 
	X=\set{ \rho\in \mathcal C^{2,\alpha}_{\rm per}(-\pi , \pi):\q \rho(\theta)=\rho(-\theta),\; \forall \theta \in (-\pi , \pi)}.
\end{equation*}
Note that the boundary $\pa\Omega_0$ is parametrized by
\[
\Big( (R_0+\rho(\theta) ) \cos\theta , (R_0+\rho(\theta) ) \sin \theta \Big) \quad \mbox{for } \theta \in [-\pi,\pi],
\]
the normal vector is given by
\[
n(\theta)= \frac{1}{((R_0+\rho(\theta))^2 + \rho'(\theta)^2 )^{1/2}}
\left(
\begin{array}{c}
	(R_0+\rho(\theta) ) \cos\theta +\rho'(\theta) \sin \theta\\
	(R_0+\rho(\theta) ) \sin\theta -\rho'(\theta) \cos \theta
\end{array}
\right),
\]
and the mean-curvature by
\[
\kappa(\theta) = \frac{
	(R_0+\rho(\theta))^2 + 2 \rho'(\theta)^2 - (R_0+\rho(\theta)) \rho''(\theta)
}{\left((R_0+\rho(\theta))^2 + \rho'(\theta)^2\right)^{3/2}}.
\]
In such a case, equation \eqref{eq:TW_shape} can  be rewritten as 
\begin{equation}\label{eq:TWrho}
	\gamma \kappa(\theta) + \chi f\left( c_1(V,\rho)e^{-aV(R_0+\rho(\theta))\cos\theta} \right)  + V  (R_0+\rho(\theta)) \cos\theta = 	p_1,
\end{equation}
for all $\theta \in [-\pi,\pi)$, where we recall that $c_1(V,\rho)$ is defined by  
\begin{equation}\label{def:normalisation}
	c_1(V,\rho)=\frac{M}{\int_{-\pi}^\pi \int_{r=0}^{R_0+\rho(\theta)} e^{-aVr\cos\theta} r \dd r\dtheta}.
\end{equation}

  Therefore, the existence of a boundary $\pa \Omega_0$ solving \eqref{eq:TW_shape} follows from the existence of a function $\rho$ solution of equation \eqref{eq:TWrho}. 

Before getting into the real proof, we also introduce the functional space
	\[
 Y=\mathcal C^{0,\alpha}_{\rm per}(-\pi , \pi).
	\]

The existence of solutions of \eqref{eq:TWrho} is proved by the following result.
\begin{theo}\label{thm:TW:bif_2}
There exists an interval $I=(-\eps,+\eps)$ and 
	four $C^1$ functions 
\[
\rho: I \times [-\pi,+\pi] \to \xR,\q  \chi: I \to \xR,\q V: I \to\xR,\q p_1: I \to\xR,
\]
such that, for  all $s \in I$, the equation \eqref{eq:TWrho} has a solution $\rho(s,\theta)$ for all $\theta \in [-\pi,\pi]$ representing a parametrization of the boundary $\pa \Omega_0$ with $\chi=\chi(s)$, $V=V(s)$ and $p_1=p_1(s)$.
\end{theo}

\begin{rema}[On Theorems \ref{thm:TW_intro_2} and \ref{thm:TW:bif_2}]
	We note that our two approaches, Theorem \ref{thm:TW_intro_2} on the one hand and Theorems \ref{thm:TW:bif_2} on the other hand are different. Indeed,  the first result fixes the Lagrange multiplier $p_1$, while the second one fixes the volume. 
	However, if the radius $R_0$ of Theorems \ref{thm:TW:bif_2}  verifies
	$|\tilde \Omega _\chi ^{R_0}|=\pi R_0^2$, then both approaches prove the existence of non trivial traveling wave solutions for $\chi> \chi^*$,  which converge to $B_{R_0}$ when $\chi \to \chi^*$.  Otherwise, 
	\[
	(\tO_{\chi^* }^{p_1},V_{\chi^* }^{p_1})=(\tO_{\chi^* }^{R_0},V_{\chi^* }^{R_0})=(\tO_{\chi^* },0).
	\]
	Theorem \ref{thm:TW_intro} can be interpreted as a bifurcation in $p_1$, since the immobile, circular solution always exists, and this result shows that for sufficiently large $p_1$ there are non-trivial traveling waves.
\end{rema}

We take advantage of the following Lemma.

\begin{lemm}\label{lemma:CR_hysteresis}
	Assume that $f$ satisfies the assumptions in \eqref{A}.  Then, the functional $\FF$ defined by \eqref{def:TW_F} has the following properties
		\begin{enumerate}[parsep=0cm,itemsep=0cm,topsep=0cm]
			\item $\mathcal F (\chi,0,0,0)=0$  for all  $\chi \in \xR.$
			\item ${\rm Ker}\, \mathcal  \partial_{(\rho,V,p_1)} \FF(\chi^*,0,0,0)$ is a one dimensional subspace of $\xR\times X\times \xR \times \xR$ spanned by  $(0,1,0)$; 
			\item ${\rm Range }\, \partial_{(\rho,V,p_1)} \FF(\chi^*,0,0,0)$ is a closed subspace of $Y\times \xR \times \xR \times \xR$ of codimension 1. 
		\end{enumerate}
\end{lemm}

\begin{rema}
Note that points 2. and 3. imply that $\mathcal{F}$ is a Fredholm mapping of order zero.
\end{rema}

	\begin{proof}[Proof of Lemma \ref{lemma:CR_hysteresis}]
	We define the function 
	\[
	\FF:\xR\times X \times\xR\times\xR\to Y\times \xR\times \xR\times \xR 
	\]
	by
	\begin{align}
		\FF(\chi,\rho,V,p_1) & =\Big( 
		\gamma\kappa(\theta)+ \chi f\left( c_1(V,\rho)e^{-aV(R_0+\rho(\theta))\cos\theta} \right) \label{def:TW_F}  \\
		& \qquad\qquad  + V  (R_0+\rho(\theta)) \cos\theta - \frac{  \gamma}{R_0} - p_1, \nonumber \\
		& \qquad\qquad \int_{-\pi}^\pi \left( (R_0+\rho(\theta))^2-R_0^2 \right)\dtheta ,\int_{-\pi}^\pi \rho(\theta) \cos \theta \dtheta,\int_{-\pi}^\pi \rho(\theta) 
		\sin \theta \dtheta  \Big),  \nonumber
	\end{align}
	with $c_1(V,\rho)$ defined by \eqref{def:normalisation}.\\

\noindent
{\it Proof of 1.}\\
The first point is obvious.\\

\noindent
{\it Proof of 2.}\\
Next, recalling the definition of $\mathcal{F}$ in \eqref{def:TW_F}, we compute $ 	\mathcal L_\chi:= \partial_{(\rho,V,p_1)}\mathcal F (\chi,0,0,0)$ which is the linear operator 
		\[
		\mathcal L_\chi: X \times \xR \times \xR \rightarrow Y \times \xR \times \xR \times \xR
		\]
		defined by
		\begin{align} \label{def:Lbeta}
			\mathcal L_\chi(\rho,V,p_1) = \FF_\rho(\chi,0,0,0) [\, \rho \,] 
			+ \FF_V(\chi,0,0,0) \, V + \FF_{p_1} (\chi,0,0,0) \, p_1
		\end{align} 
evaluated in $(\chi,0,0,0)$.	\\
		We recall that the linear operator $\FF_{\rho}(\chi,\rho,V,p_1)$ is defined by
		\[
		\FF_{\rho}(\chi,\rho,V,p_1)[\eta] =  \frac{\dd}{\dd \veps} \FF(\chi,\rho+\veps\eta,V,p_1)_{\displaystyle \mid_{\eps=0}} \quad \mbox{for } \eta \in X.
		\]
We thus compute:
		\begin{align} 
			\FF(\chi,\rho+\veps\eta,V,p_1) &= \Big(  \gamma \, \kappa_{\rho + \veps \eta}(\theta)+ \chi f\left( c_1(V,\rho + \veps \eta)e^{-aV(R_0+\rho(\theta) + \veps \eta(\theta))\cos\theta} \right)  \nonumber \\ 
			&\q + V \, [R_0 + \rho(\theta) + \veps \eta(\theta)] \cos(\theta) - \frac{  \gamma}{R_0} - p_1, \nonumber\\  
			&\int_{-\pi}^\pi (R_0+\rho(\theta) + \veps \eta(\theta))^2-R_0^2 \dtheta , \nonumber\\  
			&\q \int_{-\pi}^\pi [\rho(\theta) + \veps \eta(\theta)] \cos \theta \dtheta,\int_{-\pi}^\pi [\rho(\theta) + \veps \eta(\theta)] \sin \theta \dtheta  \Big),\label{eq:F1_pert}
		\end{align}
		where $\kappa_{\rho + \veps \eta}(\theta)$ is the mean-curvature of the perturbed boundary, that is
		\begin{align*}
			 &\kappa_{\rho + \veps \eta}(\theta) \\
			&= \frac{
				[R_0+\rho(\theta)+ \veps \eta(\theta)]^2 + 2 [\rho'(\theta)+ \veps \eta'(\theta)]^2 - [R_0+\rho(\theta) + \veps \eta(\theta)] \, [\rho''(\theta) 
				+ \veps \eta''(\theta)]
			}{\left[ \, (R_0+\rho(\theta)+ \veps \eta(\theta))^2 + (\rho'(\theta) + \veps \eta'(\theta))^2 \, \right]^{3/2}}.
		\end{align*}
We now derive  \eqref{eq:F1_pert} with respect to $\veps$,  we set $\veps=0$, and we finally consider $\rho=0$, $V=0$ and $p_1=0$. Hence, for $\eta = \rho$, we find the following expression 
		\begin{align} \label{eq:F_rho}
			\FF_\rho(\chi,0,0,0) [\, \rho \,] &=\bigg( - \gamma \, \frac{\rho(\theta)+\rho''(\theta) }{R_0}-\chi\tilde{c}R_0 f'(\tilde{c})\int_{-\pi}^\pi  \rho(\theta) \dtheta,\\ & \quad 2R_0\int_{-\pi}^\pi  \rho(\theta) \dtheta,\int_{-\pi}^\pi \rho(\theta) \cos \theta \dtheta,\int_{-\pi}^\pi \rho(\theta) \sin \theta \dtheta\bigg).\nonumber
		\end{align}
		The second and the third terms in \eqref{def:Lbeta} are simpler to compute since $V$ and $p_1$ are real quantities. We obtain that 
		\begin{align*} 
			\FF_V(\chi,0,0,0) \, V = \left( -V R_0\, a\chi \tilde{c} f'(\tilde{c})\cos \theta 
			+ V \, R_0 \cos \theta,0,0,0 \right)
		\end{align*}
	since $\partial_V c_1(0,0)=0$, and 
		\begin{align} \label{eq:F_lamb}
			\FF_{p_1}(\chi,0,0,0) \, p_1 = \left( -p_1,0,0,0 \right).
		\end{align}
		Finally, the linear operator $\mathcal L_\chi$ is given by the sum of the expressions \eqref{eq:F_rho} -- \eqref{eq:F_lamb}, that is
		\begin{align*}
			\mathcal L_\chi(\rho,V,p_1)  
&=   
			\left( - \gamma \, \frac{\rho(\theta)+\rho''(\theta) }{R_0^2}-R_0\chi\tilde{c} f'(\tilde{c})\int_{-\pi}^\pi  \rho(\theta) \dtheta \right. \\ \nonumber 
			& \quad \left. - VR_0 a\chi\tilde{c} f'(\tilde{c}) \cos \theta + V R_0 \cos \theta-p_1,\right. \\ \nonumber 
			& \quad  \left. 2R_0\int_{-\pi}^\pi \rho(\theta) \, d\theta,\int_{-\pi}^\pi \rho(\theta) \cos \theta  \dtheta,\int_{-\pi}^\pi \rho(\theta) \sin \theta \dtheta\right).
		\end{align*}
		When $\chi^*  = \frac{1}{a\tilde{c}f'(\tilde{c})}$, we get 
		\begin{align}
			\mathcal L_{\chi^*}(\rho,V,p_1) 
			&= \bigg( - \gamma \, \frac{\rho(\theta)+\rho''(\theta) }{R_0^2} -\frac{R_0}{a}  \int_{-\pi}^\pi  \rho(\theta) \dtheta -p_1, 2R_0\int_{-\pi}^\pi \rho(\theta) \dtheta,\nonumber\\
			&\q\int_{-\pi}^\pi \rho(\theta) \cos \theta \dtheta,\int_{-\pi}^\pi \rho(\theta) \sin \theta \dtheta \bigg).\label{def:Lchi0}
		\end{align}
		Thus, the elements $\{ (\rho,V,p_1) \}$ belonging to  $\mathrm{Ker}\, \mathcal L_{\chi^*}$ are such that
		\begin{equation*} 
			\rho''(\theta)+\rho(\theta)+\frac{R_0^3}{a\ga}  \int_{-\pi}^\pi  \rho(\theta) \dtheta=-\frac{R_0^2}{\gamma}p_1.
		\end{equation*}
		Using the condition $\int_{-\pi}^{\pi} \rho(\theta)  \dtheta=0$, we obtain that
		\begin{equation} \label{eq:ker}
			\rho''(\theta)+\rho(\theta)=-\frac{R_0^2}{\gamma}p_1.
		\end{equation}
		The parameter $V$ does not appear in equation \eqref{eq:ker},  so $(0,V,0)\in \mathrm{Ker}\, \mathcal L_{\chi^*}$ for all $V\in \xR$ and thus $\dim \mathrm{Ker}\, \mathcal L_{\chi^*} \geq 1$.
		Furthermore, if $(\rho,p_1)$ solves \eqref{eq:ker}, then $\rho$ has the form
		\[
		\rho(\theta) = a\cos \theta + b \sin \theta - \frac{R_0^2}{\gamma}p_1, \quad \mbox{ for some } a,b \in\xR.
		\]
We use the conditions 
\[
\int_{-\pi}^{\pi} \rho(\theta)  \dtheta=0,\q \int_{-\pi}^{\pi} \rho(\theta) \cos \theta \dtheta=0,\q \int_{-\pi}^{\pi} \rho(\theta) \sin \theta \dtheta=0
\]
to determine the coefficients $a,\,b$ and to identify $p_1$. These integral conditions  respectively imply that $p_1=0$, $a=0$ and $b=0$, hence
		$
		\mathrm{Ker}\, \mathcal L_{\chi^*}  = \mathrm {span} \{ (0,1,0)\}$
		and   $	\dim \mathrm{Ker}\, \mathcal{L}_{\chi^*}=1$.\\

\noindent
{\it Proof of 3.}\\		
Next, we show that the range of $\mathcal{L}_{\chi^*}$ consists of all the quadruplets
\[
(h,C_1,C_2,C_3)\in Y\times \xR\times\xR \times\xR\q\t{such that}\q \int_{-\pi}^{\pi} h(\theta)\cos \theta \dtheta = 0.
\]
The fact that this condition is necessary is obtained by multiplying the equation
		\begin{equation}\label{eq:cod}
			- \gamma \, \frac{\rho(\theta)+\rho''(\theta) }{R_0^2}  -p_1 = h
		\end{equation}
		by $\cos(\theta)$ and integrating over $(-\pi,\pi)$.
		To check that this condition is sufficient, we note that, for given $h\in \mathcal C^{0,\alpha}_{\rm per}(-\pi , \pi)$ and $p_1\in \xR$, equation \eqref{eq:cod} has a solution in $\mathcal C^{2,\alpha}_{\rm per}(-\pi , \pi)$ if and only if $\int_{-\pi}^{\pi} h(\theta)\cos \theta \dtheta=\int_{-\pi}^{\pi} h(\theta)\sin \theta \dtheta = 0$, and the general solution is of the form 
		$$ \rho(\theta) =\bar\rho(\theta) + a \cos\theta +b\sin\theta,$$
		for some particular solution $\bar\rho$, which we can assume to be even (otherwise we replace
		it with $\frac12[\bar\rho(\theta) + \bar\rho(-\theta)]$).
		
		When $h\in Y$, the condition $\int_{-\pi}^{\pi} h(\theta)\sin \theta \dtheta = 0$ is always satisfied since $h$ is even,
		and we find a solution in $X$ by taking the even part of $\rho$:
		$$ \rho(\theta) =\frac 1 2[ \bar\rho(\theta) +\bar\rho(-\theta) ]+ a \cos\theta .$$
		We then choose $a$ so that $\int_{-\pi}^\pi \rho(\theta) \cos \theta \dtheta=C_2$
		and integrating \eqref{eq:cod} with respect to $\theta$ yields
		$$ - \frac{\gamma}{R_0^2} \int_{-\pi}^{\pi} \rho(\theta) \dtheta - 2\pi p_1= \int_{-\pi}^{\pi}f (\theta)\dtheta$$
		so we can now choose $p_1 $ so that $2R_0^2\int_{-\pi}^\pi \rho(\theta) \dtheta=C_2$.

	\end{proof}

\begin{proof}[Proof of Theorem \ref{thm:TW:bif_2}]
 Lemma \ref{lemma:CR_hysteresis} imply that we can apply the Implicit Function Theorem to $\FF$. 
\end{proof}

\begin{proof}[Proof of Theorem \ref{thm:TW:bif}]

We follow the arguments contained in \cite[Theorem 6.4]{Berlyand2} to prove the existence of a brunch bifurcating traveling wave solutions. In particular, we want to exploit the Leray-Schauder degree theory and its connection with bifurcation points. We refer to the Appendix \ref{app:LS} for a brief synthesis on the theory hidden in the proof below. \\

We omit the dependence on the parameter $s$ when possible.\\

We take advantage of  the expression of the mean-curvature in polar coordinates to rewrite \eqref{eq:TWrho} as
\begin{align}
&\frac{(R_0+\rho)\rho''-\pare{\rho'}^2}{(R_0+\rho)^2+\pare{\rho'}^2}\label{eq:pre-arctan}\\
&=1+\frac{\pare{(R_0+\rho)^2+\pare{\rho'}^2}^\frac{1}{2}}{\ga}\biggl[\chi f\left( c_1e^{-aV(R_0+\rho)\cos\theta} \right) + V  (R_0+\rho) \cos\theta -p_1\biggr]\nonumber
\end{align}
with $\rho=\rho(\theta)$ for all $\theta \in [-\pi,\pi)$, where we recall that $c_1(V,\rho)$ is defined by  \eqref{def:normalisation}.

Since
\[
\arctan\pare{\frac{\rho'}{R_0+\rho}}'=\frac{(R_0+\rho)\rho''-\pare{\rho'}^2}{(R_0+\rho)^2+\pare{\rho'}^2},
\]
we integrate  \eqref{eq:pre-arctan} twice in $\theta$, and we get that
\begin{equation*}
\rho=K(\rho,V;s)
\end{equation*}
for
\begin{align*}
&K(\rho,V;s) =\int_0^\theta(R_o+\rho)\tan\biggl[\psi\\
&+\int_0^\psi \frac{\pare{(R_0+\rho)^2+\pare{\rho'}^2}^\frac{1}{2}}{\ga}\biggl[\chi f\left( c_1e^{-aV(R_0+\rho)\cos\psi'} \right) + V  (R_0+\rho) \cos\psi' -p_1\biggr]\biggr]\,d\psi'\,d\psi.
\end{align*}
 Since we have supposed that the domain is symmetric w.r.t. the $x$-axis, we have that 
\[
\rho\t{ is an even function, so }\rho'(0)=0.
\]
We also recall that the area preservation condition
\[
\int_{-\pi}^\pi \rho(\theta)\,d\theta=0
\]
is verified (see \eqref{eq:preservation-area}).\\
We set  the center of mass of the domain at the origin: 
\[
\frac{1}{3}\int_{-\pi}^\pi (R_0+\rho(\theta))^3\cos \theta\,d\theta=0.
\]
This implies that we do not consider copies of solutions translated in the $x$-direction. \\
The previous conditions lead us to
\[
\rho=K_1(\rho,V;s)=K(\rho,V;s)-\frac{1}{2\pi}\int_{-\pi}^{\pi}K(\rho,V;s)\,d\theta.
\]
Now, setting
\[
V=K_2(\rho,V;s)=V+\frac{1}{3}\int_{-\pi}^{\pi}(R_0+\rho)^3\cos\theta\,d\theta,
\]
the problem reduces to the following fixed point one
\begin{equation}\label{def:wK}
(\rho,V)=\pare{K_1(\rho,V;s),K_2(\rho,V;s)}=\w{K}(\rho,V;s)\q\t{in}\q X\times\xR.
\end{equation}

\begin{rema}[On the operator $\w{K}$]
We here collect some remarks on the component $K_1$ of the operator $\w{K}$.
\begin{itemize}
\item Note that, if $\rho\in X$, then 
\[
K_1(\rho,V;s)\in X.
\]
Moreover, by the definition of $K$,  $K_1(\rho,V;s)$ maps even functions in even ones.
\item  The  $K_1$ just defined is equivalent to the first component of the  operator $\mathcal{F}$ defined in \eqref{def:TW_F}. In particular, this means that the Lemma \ref{lemma:CR_hysteresis} holds for $\w{K}$, i.e.
\[
\w{K}(0,0;s)=0\q\forall s\in \xR,
\]
and that $\w{K}$ is Fréchet differentiable in $0$ (see the proof of Lemma \ref{lemma:CR_hysteresis}).\\
 We can thus linearize $\w{K}$ as explained in Remark \ref{rmk:important}.
\end{itemize}
\end{rema}

Reasoning as in the proof of Lemma \ref{lemma:CR_hysteresis}, the linearizations of $K_1,\, K_2$ around the point $(0,0;s)$, $p_1=0$, are
\begin{align}
L_\rho(\rho,V;s)&=\frac{R_0^3}{\ga}\int_0^\theta\int_0^\psi V\pare{-\chi a	\tilde{c}f'(\tilde{c})+1}\cos \psi'-\frac{\ga}{R_0^3}\rho\,d\psi'\,d\psi-\w{C}\label{eq:Lrho}
\\
L_V(\rho,V;s)&=V+R_0^2\int_{-\pi}^\pi \rho\cos\theta\,d\theta\label{eq:LV}
\end{align}
where $\w{C}$ is the mean value of the first term in $L_\rho(\rho,V;s)$.

Thanks again to Remark \ref{rmk:important},  we  proceed computing the eigenvalues of
\[
\pare{L_\rho(\rho,V;s),L_V(\rho,V;s)}=\w{L}(\rho,V;s).
\]
 In particular, it is without loss of generality that we reduce the case $V\ne 0$ to $V=\pm 1$ by rescaling arguments.\\

We begin considering $(\rho,V)=(\rho,\pm 1)$. Let $E_1$ be the desired eigenvalue. Then, from \eqref{eq:LV}, we have
\[
\int_{-\pi}^\pi \rho\cos\theta\,d\theta=\frac{E_1-1}{R_0^2}.
\]
We now derive \eqref{eq:Lrho}
\[
L_\rho(\rho,V;s)=E_1\rho
\]
 twice by $\theta$, obtaining
\[
\cos\theta \pare{-a\tilde{c}\chi f'(\tilde{c})+1}-\frac{\ga}{R_0^3}\rho=\frac{\ga}{R_0^3}E_1\rho''.
\]
Then,  we multiply this equation by $\cos \theta$, and we integrate in $[-\pi,\pi]$, getting that
\begin{equation*}
 \pare{-a\chi\tilde{c} f'(\tilde{c})+1}\pi =\ga\frac{(E_1-1)^2}{R_0^5},
\end{equation*}
otherwise, by definition of $\chi^*$,
\begin{equation}\label{eq:jump1}
 \pare{1-\frac{\chi}{\chi^*}}\pi =\ga\frac{(E_1-1)^2}{R_0^5}\q\forall\chi(s)\le \chi^*.
\end{equation}
We set 
\begin{align*}
E_{1,1}&= R_0^2\sqrt{\frac{R_0}{\ga}\pare{1-\frac{\chi}{\chi^*}}\pi}+1\ge1,
\\
E_{1,2}&=-R_0^2\sqrt{\frac{R_0}{\ga}\pare{1-\frac{\chi}{\chi^*}}\pi}+1\le1.
\end{align*}
We claim that the cases $E_{1,1}=E_{1,2}=1$ are not admitted. Indeed, this would mean that
\[
\chi(s)=\chi^*=\chi(0),
\]
i.e., $\chi=\chi(s)$ constant, so no bifurcation occurs.\\
The case $(\rho,V)=(\rho,-1)$ can be dealt with in the same way, and it leads to
\begin{equation}\label{eq:jump2}
- \pare{1-\frac{\chi}{\chi^*}}\pi =\ga\frac{(E_{-1}-1)^2}{R_0^5}\q\forall\chi(s)\ge \chi^*.
\end{equation}
This means that,  there is a jump of the local Leray-Scauder index through $s=0$.

As far as the case $(\rho,V)=(\rho,0)$ is concerned, we need
\[
\int_{-\pi}^\pi \rho\cos\theta\,d\theta=0.
\]
Hence
\[
\rho=\cos m\theta,\q m\ge 2.
\]
Reasoning as before, we get that the condition
\[
-\rho=E_{0,m}\rho''
\]
has to be satisfied. Then 
\begin{equation*} 
E_{0,m}=m^{-2}.
\end{equation*}
Note that $E_{0,m}<1$ for all $m\ge 2$.\\

We now want to prove that $(0,0;s_*)$ is a bifurcation point for some $s_*\in \xR$. In order to do so, we are going to apply Theorems \ref{teoA4} and \ref{prop2}.

Since we have that the l.h.s. of \eqref{eq:jump1} (resp. \eqref{eq:jump2}) is positive if
\[
\chi^*>\chi(s)\q\pare{\t{resp. } \chi^*<\chi(s)},
\]
we can identify a value $s_-$ (resp. $s_+$) such that 
\[
\t{deg}_{LS}(I-\w{K}(\rho,V;s_\pm),B_\eps(0),0)
\]
is well defined for
\[
B_\eps(0)=\set{(\rho,V):\,\, |V|<\eps,\,\,\norm{\rho}_{C^{2,\al}_{per}(-\pi,\pi)}<\eps},\q \eps\ll1.
\] 
We now claim that $(0,0;s)$ is an isolated solution for every $s\in\xR$.\\
 Indeed,  $\w{K}$ is Fréchet differentiable at $0$, and $1$ belongs to the resolvent of $\overline{L}$ because none of the eigenvalues we found is equal to $1$.\\
Then, being the assumptions of Theorem \ref{teoA4} satisfied,  $(0,0;s)$ is an isolated solution and, thanks also to \eqref{eq:resume},  we get that
\[
\t{deg}_{LS}(I-\w{K}(\rho,V;s_\pm),B_\eps(0),0)=(-1)^{m_\pm},
\]
where $m_\pm$ have been defined in Theorem \ref{teoA4}, and $m_+$ (resp. $m_-$) refers to $s_+$ (resp. $s_-$). 
In order to verify that we have a bifurcation point, it suffices to prove that $m_-\ne m_+$ (see Theorem \ref{prop2}).\\
Since the number of eigenvalues contained in $(1,+\infty)$ coincides at $s_\pm$ for $E_1$, but it differs by one for $E_{0,m}$, we conclude that
\[
\t{deg}_{LS}(I-\w{K}(\rho,V;s_-),B_\eps(0),0)\ne\t{deg}_{LS}(I-\w{K}(\rho,V;s_+),B_\eps(0),0),
\]
and then, from Theorem \ref{prop2}, there exists a $s_*\in \bra{\min s_\pm,\max s_\pm}$ such that $(0,0;s_*)$ is the desired bifurcation point.   \\
This means that there exists a sequence of solutions 
\[
(V_\eps,\rho_\eps)\in\partial B_\eps(0)
\]
such that
\[
(V_\eps,\rho_\eps;s_\eps)\to (0,0;s_*).
\]
(see Definition \ref{def:bp-2}).\\
We point out that we must have  $(V_\eps,\rho_\eps)\in\partial B_\eps(0)$ because $(V,\rho)=(0,0)$ is an isolated solution in $B_\eps(0)$ for $\eps$ small enough.\\
To conclude this proof, we have just to show that the sequence $(V_\eps,\rho_\eps)$ is made of traveling waves. 
This is trivial since, being $(V_\eps,\rho_\eps)$ solutions to $I-\w{K}(\rho,V;s)$, then they verify Definition \ref{def:TW_1} (see also \eqref{eq:TW_shape} in Remark \ref{rmk:equiv-def}) by \eqref{def:wK}.
\end{proof}

\section{Proof of Theorem \ref{thm:TW:bif-CR}}\label{sec:thm:TW:bif-CR}

We now give some qualitative results on the nature of the bifurcating branch of traveling waves, using Crandall-Rabinowitz bifurcation results. \\
In what follow, we assume that $f$ verifies \eqref{A} and also
\[
f\in C^3(\xR^+).
\]

We begin recalling Lemma \ref{lemma:CR_hysteresis}, and we prove the following.

	\begin{lemm} \label{lemma:CR_hysteresis-CR}
		Let the assumptions of Lemma \ref{lemma:CR_hysteresis} be in force. Then the functional $\FF$ defined by \eqref{def:TW_F} also verifies
\[
\partial _{\chi } \, \partial_{(\rho,V,p_1)} \FF  (\chi^*,0,0,0)[(0,1,0)]\notin {\rm Range }\,  \partial_{(\rho,V,p_1)} \FF(\chi^*,0,0,0).
\]
		\end{lemm}

\begin{proof}
	We note that the original problem \eqref{eq:pression_TW_1} -- \eqref{eq:marqueur_bord_TW_1} is invariant by 
	translation. Thus, it is natural to eliminate these invariances by looking for solution of \eqref{eq:TWrho} satisfying the orthogonality conditions  
\[
\int_{-\pi}^\pi \rho(\theta) \cos \theta \dtheta =0\q\t{and}\q \int_{-\pi}^\pi \rho(\theta) \sin \theta \dtheta =0.
\]

We  prove the transversality condition  with respect to the value $\chi^*$, that is we have to prove that $(\pa_{\chi} \mathcal{L}_{\chi^*})(0,1,0) \notin \mathrm{Range}\, \mathcal{L}_{\chi^*}$ (see \eqref{def:Lchi0} for the definition of $\mathcal{L}_{\chi^*}$). We have
		\[
		(\pa_{\chi} \mathcal L_{\chi^*})(0,1,0) = (-a R_0 \tilde{c} f'(\tilde{c}) \cos \theta,0,0,0).
		\]
		Assume by contradiction that $(\pa_{\chi} \mathcal{L}_{\chi^*})(0,1,0) 
		\in \mathrm{Range}\, \mathcal{L}_{\chi^*}$. Then, we would have that
		\begin{equation*}
			\gamma \frac{\rho(\theta)+\rho''(\theta) }{R_0^2} +\frac{R_0}{a} \int_{-\pi}^\pi  \rho(\theta) \dtheta +p_1 = aR_0\tilde{c}f'(\tilde{c}) \cos \theta. 
		\end{equation*}
		Multiplying by $\cos \theta$ and integrating on $[-\pi,\pi]$, it yields 
		that 
		\begin{eqnarray} \label{eq:cond_trs}
			& &\gamma \int_{-\pi}^\pi \rho(\theta) \cos \theta \dtheta + \gamma  \int_{-\pi}^\pi \rho''(\theta) \cos \theta \dtheta  \\
			& & \qquad+\left( \frac{R_0}{a} \int_{-\pi}^\pi  \rho(\theta) \dtheta +p_1\right)  R^2_0 \int_{-\pi}^\pi \cos \theta \dtheta  = R_0^3 f'(\tilde{c}) a \tilde{c}\int_{-\pi}^\pi \cos^2  \theta \dtheta.\nonumber
		\end{eqnarray}
		Integrating by parts twice the second term of the left-hand side and using that $\rho'$ is a $2\pi$-period function together with $\int_{-\pi}^\pi \rho(\theta) \cos \theta \dtheta = 0$, we deduce that
		\[
		\int_{-\pi}^\pi \rho''(\theta) \cos \theta \dtheta  = \int_{-\pi}^\pi \rho'(\theta) \sin \theta \dtheta = - \int_{-\pi}^\pi \rho(\theta) \cos 
		\theta \dtheta = 0.
		\]
		Then,  equation \eqref{eq:cond_trs} leads to
		\[ 
		0 = R_0^3a\tilde{c} f'(\tilde{c}) \pi, 
		\]
		which is a contradiction since $f'(\tilde{c})>0$. 
\end{proof}

\begin{theo}\label{thm:TW:bif_2-CR}
Let the assumptions of Theorem \ref{thm:TW:bif_2} be in force. Then, the functions $\chi: I \to \xR$ and $V: I \to\xR$ verify
	\begin{itemize}[parsep=0cm,itemsep=0cm,topsep=0cm]
		\item[(i)] $V=V(s) = s + o(s)$ for all $s \in I$.
		\item[(ii)]   $\chi(0)=\frac{1}{a\tilde{c}f'(\tilde{c})}$, $\chi'(0)=0$  and 	\begin{align*}
			\chi''(0)  = & -\frac{aMR_0^2}{2\pare{f'(\tilde{c}) }^2}\bra{\frac{M}{2\pi R_0^2}f'''(\tilde{c})+f''(\tilde{c})}. 
		\end{align*} 
	\end{itemize}
\end{theo}
	\begin{proof} 
	We can now apply the Crandall-Rabinowitz Bifurcation Theorem   \ref{thm:CR}.\\
 Let us denote by $Z$ any complement space of $\mathrm{Ker}\, \mathcal{L}_{\chi^*}$, there 
	exists an interval $I=(-\eps,\eps)$ and four $\mathcal{C}^1$ functions $\vp: I \rightarrow \xR$, $\psi_1: I \times [-\pi,\pi] \rightarrow Z$, $\psi_2: I \rightarrow Z$ and $\psi_3 : I \rightarrow Z$ such that
	\begin{equation}\label{eq:Fs0} 
		\FF(\vp(s),\psi_1(s,\theta),\psi_2(s),\psi_3(s))=(0,0,0,0) \quad \mbox{for all } s \in I, \, \theta \in [-\pi,\pi],
	\end{equation}
	and 
	\begin{equation*}
		\vp(0)=\frac{1}{a\tilde{c}f'(\tilde{c})}, \quad \psi_1(0,\theta)=0 \mbox{ for all } \theta \in [-\pi,\pi], \quad \psi_2(0)=0, \quad \psi_3(0)=0.
	\end{equation*}
	In particular, the solutions $(\chi,\rho,V,p_1)=(\chi(s),\rho(s,\theta),V(s),p_1(s))$ of the equation $\FF(\chi,\rho,V,p_1)=(0,0,0,0)$ are of the form 
	\begin{equation} \label{eq:TW_sol}
		\begin{cases}
			\chi(s)=\vp(s), \quad & \rho(s,\theta)=0+s \, \psi_1(s,\theta), \\ V(s)=s+s\,\psi_2(s), \quad & p_1(s)=0+s \,\psi_3(s)
		\end{cases}
	\end{equation}
	and they verify
\begin{equation}\label{eq:TW_sol_0}
\begin{split}
\chi(0)&=\chi^*=\frac{1}{a\tilde{c}f'(\tilde{c})}, \qq
\rho(0,\theta)= \partial_s \rho(0,\theta)=0 \quad\mbox{ for all } \theta \in [-\pi,\pi], \\ 
V(0)&=0 \quad \textrm{ and } 	\quad V'(0)=1, \qq
p_1(0)= p_1'(0)=0 .
\end{split}
\end{equation}
	The point (ii) will follow from Lemma 
	\ref{lem:beta'} and \ref{lem:beta''} below.
	\end{proof}
	
	In the proofs that follow, we use extensively the fact that the functions 
\[
\theta\mapsto \rho(0,\theta)\q\t{and}\q \theta\mapsto \pa_s \rho(0,\theta)
\]
(and all their derivatives with respect to $\theta$) vanish. 
	Using that
\[
\int_{-\pi}^\pi \rho(s,\theta) \cos \theta \dtheta=0\q\t{and}\q \int_{-\pi}^\pi \big((R_0+\rho(s,\theta))^2-R_0^2 \big)\dtheta=0
\]
for all $s$, we also have
	\begin{equation}\label{eq:rhointzero}
		\int_{-\pi}^\pi \pa_{s}^n \rho(0,\theta) \cos \theta \dtheta=0\q \forall n\in \mathbb{N}\q\t{and}
	\q	\int_{-\pi}^\pi\pa_{ss} \rho(0,\theta) \dtheta=0.
	\end{equation}	
	Some  technical computations are contained in the Appendix (see Sections \ref{app:kappa}--\ref{app:z}).

\begin{lemm}\label{lem:beta'}
Assume that \eqref{eq:rho0} holds. Then, we have $\chi'(0)=0$.
\end{lemm}
\begin{proof} 
We recall that
\begin{equation}\label{def:k}
	\kappa(s,\theta)=\frac{((R_0+\rho)^2 + 2 \pat \rho^2 - (R_0+\rho) \patt \rho)(s,\theta)}{\big(((R_0+\rho)^2 + \pat \rho^2)^{3/2}\big)(s,\theta)},
\end{equation}
and we set 
\begin{equation}\label{def:z}
	z(s,\theta)= c_1(V,\rho)e^{-aV(s)(R_0+\rho(s,\theta))\cos\theta},
\end{equation}
with $c_1(V,\rho)$ given by \eqref{def:normalisation}.

	We differentiate with respect to $s$ the first component of $\FF$ given by \eqref{def:TW_F}  and we use \eqref{eq:Fs0} to obtain
	\begin{align} \label{eqF1}
		0= & \gamma \pas \kappa (s,\theta) + \chi'(s) f(z(s,\theta)) + \chi(s) f'(z(s,\theta)) \pas z(s,\theta) \\ \nonumber
		& + V'(s) (R_0 + \rho(s,\theta) )\cos \theta + V(s) \pas \rho(s,\theta) 
		\cos \theta - p_1'(s).
	\end{align}
	
Let $s=0$. Then, \eqref{eq:TW_sol_0} implies that $c_1(V,\rho)=\tilde{c}$,  hence
\[
z(0,\theta)=\tilde{c} \quad \mbox{for all } \theta \in [-\pi,\pi].
\] 
Consequently,  \eqref{eqF1} simply writes as 
\begin{equation} \label{eqF1_bis}
 \gamma \pas \kappa (0,\theta) + \chi'(0) f(\tilde{c}) + \chi(0) f'(\tilde{c}) \pas z(0,\theta) +R_0\cos \theta =0.
\end{equation}

Recalling that $\chi(0)=\frac{1}{a\tilde{c}f'(\tilde{c})}$ and \eqref{eq:pasz}, it follows that \eqref{eqF1_bis} simplifies as 
\begin{equation*} 
	\gamma \pas \kappa (0,\theta) + \chi'(0) f(\tilde{c}) =0.
\end{equation*} 

	The conclusion follows since we know that $\pas \kappa (0,\theta)=0$ (see Lemma \ref{lem:app}) and $f(\tilde{c})\neq 0$ (see assumptions \eqref{A}).
\end{proof}
	
	\begin{lemm}\label{lem:beta''}
Assume \eqref{eq:rho00}.		It holds that
		\begin{align*}
			 \chi''(0)  = -\frac{aMR_0^2}{2\pare{f'(\tilde{c}) }^2}\bra{\frac{M}{2\pi R_0^2}f'''(\tilde{c})+f''(\tilde{c})}.
		\end{align*} 
	\end{lemm}
\begin{proof}
	We now differentiate twice equation \eqref{eqF1} with respect to $s$, obtaining
\begin{align} 
	0 = & \gamma \passs \kappa (s,\theta) + \chi'''(s) f(z(s,\theta)) + 3 \chi''(s) f'(z(s,\theta)) \pas z(s,\theta)\nonumber \\ 
	& +3\chi '(s) [ f''(z(s,\theta)) (\pas z(s,\theta))^2 +  f'(z(s,\theta)) \pass z(s,\theta)] \nonumber\\ 
	& + \chi(s) [ f'''(z(s,\theta)) (\pas z(s,\theta))^3 + 3 f''(z(s,\theta)) \pas z(s,\theta) \, \pass z(s,\theta)+ f'(z(s,\theta)) \passs z(s,\theta)]   \nonumber\\
	&  + V'''(s) (R_0+\rho(s,\theta)) 
	\cos \theta   + 3 V''(s) \pas \rho(s,\theta) \cos \theta  + 3 V'(s) \pass \rho(s,\theta) \cos \theta  \nonumber\\
	& + V(s) \passs \rho(s,\theta) \cos \theta - p_1'''(s). \label{eq:F3}
\end{align}
Since we have \eqref{eq:TW_sol}, \eqref{eq:rho00}, \eqref{eq:pasz}, and Lemma \ref{lem:beta'}, 
the expression in \eqref{eq:F3} for $s=0$ becomes
\begin{align} 
	0 = & \gamma \passs \kappa (0,\theta) + \chi'''(0) f(\tilde{c}) -3a\tilde{c} R_0  \chi''(0) f'(\tilde{c}) \cos\theta\nonumber \\ 
	& + \frac{1}{a\tilde{c}f'(\tilde{c})} [- f'''(\tilde{c}) (a\tilde{c}R_0\cos \theta)^3 -3a\tilde{c}R_0\cos \theta  f''(\tilde{c})  \pass z(0,\theta)+ f'(\tilde{c}) \passs z(0,\theta)]   \nonumber\\
	&  + V'''(0) R_0
	\cos \theta       - p_1'''(0). \label{eq:F3_bis}
\end{align}
We multiply  \eqref{eq:F3_bis} by $\cos \theta $, and we integrate in $\theta$:
\begin{align} 
	0 = & \gamma \int_{-\pi}^\pi\passs \kappa (0,\theta)\cos\theta \,d\theta-3\pi a\tilde{c} R_0  \chi''(0) f'(\tilde{c}) \nonumber \\ 
	& + \frac{1}{a\tilde{c}f'(\tilde{c})} \Bigg(- \frac{3\pi}{4}f'''(\tilde{c}) (a\tilde{c}R_0)^3 -3a\tilde{c}R_0f''(\tilde{c})\int_{-\pi}^\pi\cos^2 \theta    \pass z(0,\theta)\,d\theta
\nonumber \\ 
& + f'(\tilde{c})\int_{-\pi}^\pi\cos\theta \passs z(0,\theta)\,d\theta\Bigg)    + \pi V'''(0) R_0
 .\nonumber
\end{align}
Using  Lemma \ref{lem:app_2}, we simplify as follows:
\begin{align} 
	0 = & -3\pi a\tilde{c} R_0  \chi''(0) f'(\tilde{c})  + \frac{1}{a\tilde{c}f'(\tilde{c})} \bra{- \frac{3\pi}{4}f'''(\tilde{c}) (a\tilde{c}R_0)^3 -\frac{3}{2}a^3M^2\tilde{c}R_0f''(\tilde{c}) }   \nonumber\\
&+\frac{1}{a\tilde{c}}\pare{-\frac{aM}{ R_0}V'''(0)-\frac{2M}{\pi R_0^3}\int_{-\pi}^\pi\cos\theta\passs\rho(0,\theta)\,d\theta}  + \pi V'''(0) R_0\label{eq:chi''}.
\end{align}
Since $\tilde{c}=M/\pi R_0^2$,   \eqref{eq:chi''} becomes as
\begin{align*}
	0 = & -3\frac{ aM}{R_0}   \chi''(0) f'(\tilde{c})  -\frac{3a^2M^2  R_0}{2f'(\tilde{c})}\bra{\frac{M}{2\pi R_0^2}f'''(\tilde{c})+f''(\tilde{c})} -\frac{2}{aR_0}\int_{-\pi}^\pi\cos\theta\passs\rho(0,\theta)\,d\theta.
\end{align*}
Using \eqref{eq:rhointzero}, we finally get
\begin{equation*}
 \chi''(0)=-\frac{aMR_0^2}{2\pare{f'(\tilde{c}) }^2}\bra{\frac{M}{2\pi R_0^2}f'''(\tilde{c})+f''(\tilde{c})}.
\end{equation*}
\end{proof}

\appendix
\section{Spectrum of the heat equation with homogeneous Neumann boundary condition}\label{app-a}

We recall here some facts concerning the spectrum of the heat equation with homogeneous Neumann boundary condition.\\

Let $m\in \xZ$, $z\in \xC$, and $w:\xC\to\xR$. Then,  the Bessel and the modified Bessel second-order differential equations read
\begin{equation}\label{eq:diff_Bessel_m}
z^2w''(z)+zw'(z)+(z^2-m^2)w(z)=0,
\end{equation}
and
\begin{equation}\label{eq:diff_Bessel_modified_m}
z^2w''(z)+zw'(z)-(z^2+m^2)w(z)=0.
\end{equation}

The Bessel $J_m$  and the modified Bessel $I_m$ functions of the first kind of order $m\in \xZ$ are particular solutions to \eqref{eq:diff_Bessel_m} and \eqref{eq:diff_Bessel_modified_m}, and their expressions are
\begin{equation}\label{eq:Bessel_m}
J_m(x)=\sum_{p=0}^\infty \frac{(-1)^p}{p!(p+m)!} \left( \frac{x}{2}\right)^{2p+m},
\end{equation}
and
\begin{equation}\label{eq:Bessel_modified}
i^m I_m(x)=J_m(ix),
\end{equation}
that is
\begin{equation}\label{eq:Bessel_modified}
I_m(x)= \sum _{p=0}^\infty \frac{1}{p!(m+p)!}\left( \frac{x}{2}\right)^{m+2p}.
\end{equation}

Note that $I_m$ and $J_m$ verify the following equation:
\begin{equation}\label{eq:der_Bessel_m}
-2J'_m(x) = J_{m+1}(x)-J_{m-1}(x),
\end{equation}
where $J'_m(x) $ is the derivative of $J_m(x)$. 

\begin{lemm}
For all $z \in \xC$, there holds 
\begin{equation}\label{eq:relation_Bessel}
z I_{m+1}(z)I_m(\bar z)-\bar z I_{m+1}(\bar z ) I_m(z)= \left( z^2-\bar z^2\right) \int_0^1 u I_m(uz)I_m(u\bar z ) \du.
\end{equation}
\end{lemm}

Let $\lambda _{m,p}$ be the $p$-th real positive root of $J_m'$, the derivative of $J_m$, such that 
\begin{equation}\label{eq:root_bessel}
\lambda _{m,p+1}> \lambda_{m,p}>\cdots >\lambda _{m,0}>0.
\end{equation}

We consider the heat equation with Neumann boundary condition 
\begin{subequations}\label{eq:CD_uncoupled_lin}
\begin{align}
\partial_t c(t,r,\theta) &=\Delta c (t,r,\theta) \qq &&(r,\theta) \in [0,1)\times [-\pi , \pi),\\
\partial_r c(t,1,\theta) &=0 \qq  &&\theta \in  [-\pi , \pi).
 \end{align}
\end{subequations}

\begin{prop}
The spectrum of the operator \eqref{eq:CD_uncoupled_lin} is  
\[
\{-\lambda _{m,p}^2, \, (m,p) \in \xN\times \xN \}.
\]
\end{prop}

\section{Bifurcation through the Leray-Schauder degree theory}\label{app:LS}

We here give a very brief background on the (huge) theory regarding the Leray-Schauder degree and the local Leray-Schauder  index
, LS degree and LLS index from now on. A synthesis of the incoming presentation is given in \cite{Mawhin}.\\
We also recall the link among the above notions and bifurcation points. For this part, we mainly refer to \cite{K}.\\

The LS degree is a powerful tool that, under suitable hypothesis, algebraically counts the number of zeros of functions. In our case, we are interested in functions of the type $I-F$, so that their zeros are, actually, fixed points.\\
The LS degree is defined in normed linear space which may be infinite-dimensional, and its  expression can be derived from  the Brouwer degree theory (see, for instance, \cite[Chapters 8 and 9]{Brown}, \cite[Chapter 2]{Droniou}, and \cite[Theorem 10.6]{Brown} too). However, we do not deepen more in this direction because it would require technical tools from homology theory. 
The interested reader can find the construction of the LS degree from the Brouwer one, for instance, in  \cite[Section 3.3]{Droniou}, \cite[Chapter 10]{Brown}. \\

The well-posedness setting for the LS degree is the following. \\
Let $X$ be a real Banach space, $U$ a bounded open subset of $X$, and $\Psi$ a mapping such that
\[
\Psi=I-F:\overline{U}\to X,
\]
where $F:\overline{U}\to X$ is a continuous and compact map. Then, the LS degree
\[
\t{deg}_{LS}\bra{\Psi,U,z},
\]
with
\[
z\in X\q\t{and}\q z\notin \Psi(\partial U),
\]
is a function which associates to the triple $(\Psi,U,z)$ a certain integer.\\
For the sake of completeness, we here list some of the most important properties that $\t{deg}_{LS}$ verifies. We also refer to
\cite[Section 3.2.2]{Droniou}, \cite[Chapter 11]{Brown}, and  \cite[Appendix]{CR2}.
\begin{theo}[{\cite[Theorem 3.2.2]{Droniou}}]
Let 
\[
Y=\set{(\Psi,U,z) \t{ defined as above}}.
\]
Then, there exists a function 
\[
\t{deg}_{LS}:Y\to \mathbb{Z}
\]
such that the following properties are satisfied:
\begin{itemize}
\item normalization: $\t{deg}_{LS}\bra{I,U,z}=1$ for all $z\in U$;
\item additivity: if $U_1,\,U_2$ are open subsets of $U$ such that $U_1\cap U_2=\emptyset$, and $z\notin \Psi(\overline{U}\setminus(U_1\cup U_2))$, then
\[
\t{deg}_{LS}\bra{\Psi,U,z}=\t{deg}_{LS}\bra{\Psi,U_1,z}+\t{deg}_{LS}\bra{\Psi,U_2,z};
\]
\item homotopy invariance: if $h:[0,1]\times \overline{U}\to X$ is compact,  $y:[0,1]\to X$ is continuous, and $y(t)\notin \Psi(\partial U)$ for all $t\in[0,1] $, then
\[
\t{deg}_{LS}\bra{I-h(0,\cdot),U,y(0)}=\t{deg}_{LS}\bra{I-h(1,\cdot),U,y(1)}.
\]
\end{itemize}
\end{theo}
We also recall the following Theorem (see also \cite[Proposition 3.2.6]{Droniou}).
\begin{theo}[{\cite[Theorem 11.1]{Brown}}]\label{thm:Brown1}
If
\[
\t{deg}_{LS}\bra{\Psi,U,z}\ne 0,
\]
then $\Psi(x)=0$, i.e.,  $F(x)=x$ for some $x\in U$.
\end{theo}

However, we are not interested in applying the LS degree theory to find fixed points, but bifurcation ones. To this aim, let us introduce the bifurcation parameter $s\in\xR$. We represent the dependence on $s$ of the function $\Psi$ with the subscript $s$:
\[
s\to \Psi(x;s)=\Psi_s(x).
\]
Roughly speaking, once we know that
\[
\Psi_s(x_0)=0
\] 
for a certain $x_0\in U$ and for all $s\in\xR$, then we can exploit the link among LS degree  and bifurcation points to show that there exists a value $s_0\in\xR$  such that $(x_0;s_0)$ is a bifurcation point. \\
With this aim, we restrict our attention to the particular case of isolated solutions, whose definition is given below.
\begin{defi}[Isolated solution {\cite[Section 8]{LS}}]
Let
\[
B_\eps(x_0)=\set{x:\,\norm{x-x_0}_X< \eps}.
\]
Then, we say that $x_0$ is an isolated solution to
\[
\Psi(x_0)=x_0-F(x_0)=z
\]
if the only solution contained in $B_\eps(x_0)$ is $x=x_0$.
\end{defi}
Let us show the connection among LS degree, LLS index, and solutions to $\Psi(x)=0$. We refer to \cite[Section 8]{LS} for the incoming results.\\
Assume that there exists a finite number of  $x_i\notin\partial U$ such that they are the only points satisfying 
\[
\Psi(x_i)=x_i-F(x_i)=z\q\t{for }1\le i\le r<\infty.
\]
Then, the LS degree can be expressed through the following formula:
\begin{equation}\label{eq:sum}
\t{deg}_{LS}\bra{\Psi,U,z}=\sum_{i=1}^{r}\t{ind}(\Psi,x_i).
\end{equation}
The formula in \eqref{eq:sum} implies that, if $x_0$ is an isolated solution, then
\[
\t{deg}_{LS}\bra{\Psi,B_\eps(x_0),y}=\t{ind}(\Psi,x_0).
\]
Hence, in order to compute the LS degree of isolated solutions, it suffices to have a representation formula for the LLS index in this particular case. It is with this reason that we recall the following result.
\begin{theo}[{\cite[Lemma A.4]{CR2}}]\label{teoA4}
Let $0\in U$ and $F(0)=0$. Assume that $F$ is Fréchet differentiable at $0\in X$, and $1$ belongs to the resolvent set of $F'(0)$. \\
Then, $0$ is an isolated solution to $\Psi(x)=0$ and 
\[
\t{ind}(\Psi,0)=(-1)^m,
\]
where $m$ is the sum of the multiplicities of the characteristic values of $F'(0)$ in the interval $(0,1)$.
\end{theo}

We recall that 
\begin{itemize}
\item a characteristic values of $F'(0)$  is a real number $\mu$ such that the equation
\[
\mu F'(0)x=x
\]
has a non-trivial solution;
\item the algebraic multiplicity of an eigenvalue $\mu$ is, in this case, the dimension of 
\[
\cup_{k\in\mathbb{N}}\pare{I-\mu F'(0)}^k;
\]
\item the resolvent set is given by  $\xR\setminus\set{\frac{1}{\mu}:\q \mu F'(0)x=x}$.
\end{itemize}
Then, the $m$-power is the sum of the algebraic multiplicities of the real eigenvalues to $F'(0)$ that are greater than one. 
\begin{rema}[{\cite[Chapter IV §2]{K}}]\label{rmk:important}
Let $F$ as in Theorem \ref{teoA4}. Then, the operator $F$ can be expressed in the form
\[
F=F'(0)+R,
\]
being $F'(0)$ its Fréchet derivative evaluated in $0$, and where the term $R$ verifies
\[
\lim_{\norm{x}_X}\frac{\norm{R(x)}_X}{\norm{x}_X}=0.
\]
This is an important remark because, instead of dealing with a possibly nonlinear $F$, we can consider its linearization $F'(0)$. 
\end{rema}
Resuming, the above Theorem \ref{teoA4} and \eqref{eq:sum} imply that
\begin{equation}\label{eq:resume}
\t{deg}_{LS}\bra{\Psi,B_\eps(0),0}=\t{ind}(\Psi,0)=(-1)^m,
\end{equation}
for $m$ defined as in Theorem \ref{teoA4}.\\

We now complete the puzzle showing the link among LLS index and bifurcation points.\\ 
Roughly speaking, as explained in  \cite[Section 6]{Berkovits}, we need to have a change of degree at a certain point  in order to have a bifurcation at this point.
 
\begin{defi}[Bifurcation point - {\cite[Section 5]{LalouxMawhin}}]\label{def:bp-1}
Assume that $F$ is a linear compact Fredholm mapping of order zero, i.e. $Im F$ is closed in $X$ and $dim \,ker F=codim\, Im F<\infty$.  
Assume also that $\Psi_s(0)=0$. Then, the couple 
\[
(0;s_0)\in D= \set{(0;s)\in X\times \xR}
\]
is a bifurcation point for the solutions of $\Psi_s(x)=0$ w.r.t. $D$ if every neighborhood of $(0;s_0)\in \overline{U}\times \xR$ contains at least one solution $(x;s)\ne  (0;s)$
to $\Psi_s(x)=0$.
\end{defi}

\begin{defi}[Bifurcation point - {\cite[Section 8]{Mawhin}}]\label{def:bp-2}
Let $F:U\to X$ completely continuous, and $\Psi_s(0)=0$ for every $s\in \xR$. Then, the couple $(0;s_0)$ is a bifurcation point to $\Psi_s(x)=0$ if  there exists a sequence $\set{(x_k;s_k)}$ of solutions to $\Psi_{s_k}(x_k)=0$ with $(x_k;s_k)\in \pare{U\setminus\set{0}}\times \xR$ such that
\[
s_k\to s_0\q\t{and}\q \norm{x_k}_X\to 0.
\]
\end{defi}
We refer also to \cite{Berkovits}, \cite[Section 1]{Toland} and \cite[Chapter IV p. 181]{K} for the notion of bifurcation point in Definition \ref{def:bp-2}.

\begin{theo}[{\cite[Proposition 2]{LalouxMawhin}, \cite[Chapter IV §5]{K}}]\label{prop2}
Let $F$ as in Theorem \ref{teoA4}. Assume that, for some $s_1<s_2$, we know that $ \t{ind}(\Psi_{s_i},0)$ is well-defined, and
\[
 \t{ind}(\Psi_{s_1},0)\ne \t{ind}(\Psi_{s_2},0).
\]
Then, there exists an $s_0\in [s_1,s_2]$ such that $(0;s_0)$ is a bifurcation point for $\Psi$.
\end{theo}

\section{Bifurcation through Crandall-Rabinowitz theory}\label{app:CR}
\subsection{A theorem of Crandall-Rabinovitz} \label{C3_Bifurcation}

We recall here the classical  Bifurcation Theorem of Crandall-Rabinovitz \cite{CR} that we used to prove Theorem \ref{thm:TW:bif_2-CR}.\\

 Given $U, V$ two real Banach spaces and a continuous map $\mathcal F :\, \xR \times U \to V$, the goal 
is to analyze the structure of the solution set 
\[
\mathcal F [\lambda ,u]=0, \quad \left( \lambda ,u\right) \in \xR \times U.
\]

\begin{theo}[Local bifurcation \cite{CR}] \label{thm:CR}
	Let $U,V$ be Banach spaces, $W$ a neighborhood of $(\lambda_0,0)$ in $\xR\times U$, and 
	$\mathcal F: W \longrightarrow V$. 
	Suppose that the following properties are satisfied
	\begin{enumerate}[parsep=0cm,itemsep=0cm,topsep=0cm]
		\item $\mathcal F (\lambda,0)=0$ for all $\lambda $ in a neighborhood of $\lambda_0$;
		\item The Fr\'echet partial derivatives $\partial _u \, \mathcal F,\partial _\lambda  \, \mathcal F,\partial _{\lambda u} \, \mathcal F$ exist and 
		are continuous;
		\item ${\rm Ker}\, \partial _u \, \mathcal F(\lambda_0,0)$ is a one dimensional subspace of $U$ spanned by a nonzero vector $u_0 \in U$; 
		\item ${\rm Range}\, \partial_u \, \mathcal F(\lambda_0,0)$ is a closed subspace of V of codimension 1;
		\item $\partial _{\lambda u} \, \mathcal F(\lambda_0,0)[u_0] \notin {\rm Range}\, \partial _{ u} \, \mathcal F(\lambda_0,0)$.
	\end{enumerate}
Then, for any complement $Z$  of ${\rm Ker} \, \partial _u \, \mathcal F(\lambda_0,0)$ in $U$,  there exist a neighborhood $N$ of  $(\lambda_0,0)$ in 
	$\xR \times U$, an interval $I=(-\eps,\eps)$ for some $\eps>0$ and two continuous functions 
	\[
	\vp: (-\eps,\eps) \longrightarrow \xR, \quad \psi: (-\eps,\eps) \longrightarrow Z
	\] 
	such that $\vp(0)=\lambda_0$, $\psi(0)=0$ and 
	\[
	\mathcal F^{-1}[0] \cap U = \lbrace (\vp(s), s u_0 + s \psi(s)): \, 
	|s| < \eps \rbrace \cup \lbrace (\lambda,0) : (\lambda,0) \in N\rbrace.
	\]
Furthermore, if $\partial _{ uu} \, \mathcal F$ is continuous then the functions $\rho$ and $\psi$ are once continuously differentiable.
\end{theo}
\begin{rema}[Bifurcation point through Jordan curves]
	In Theorem  \ref{thm:CR},  $(\lambda_0,0)$ is a bifurcation point of the equation $\mathcal F(\lambda,u) = 0$ in the following sense: in a neighborhood of $(\lambda_0,0)$, the set of solutions of $\mathcal F(\lambda,u) = 0$ consists of two curves $\Gamma_1$ and $\Gamma_2$ which intersect 
	only at the point $(\lambda_0,0)$; $\Gamma_1$ is the curve $(\lambda_0,0)$ and $\Gamma_2$ can be parameterized as follows:
	\[
	\Gamma_2 : ( \lambda(s),u(s)),\,  |s| \textrm{ small }; \, ( \lambda(0),u(0)) = (\lambda_0,0);\,  u'(0) = u_0, \, \lambda'(0)\neq 0.
	\]
\end{rema}

\subsection{Computation of $\pas \kappa$, $\pass \kappa$ and $\passs \kappa$}\label{app:kappa}

Recall that $\rho(s,\theta)$ and $\kappa(s,\theta)$ are defined by \eqref{eq:TW_sol}, \eqref{eq:TW_sol_0} and \eqref{def:k} for $s \in I$ and $\theta \in [-\pi,\pi]$.

\begin{lemm}\label{lem:app}
		Assume 
\begin{equation}\label{eq:rho0}
\pat^n\rho(0,\theta) =\pat^n \pas \rho(0,\theta)=0
\qq\forall\theta \in [-\pi,\pi],\,\q \forall n\in \mathbb{N}.
\end{equation}
	Then
	$$\pas \kappa (0,\theta) = 0 $$
	and
	\begin{equation}
		R_0^2\pass \kappa (0,\theta)=- \pass \rho(0,\theta)  -\passtt \rho(0,\theta). \label{eq:der_seconde_k}
	\end{equation}
\end{lemm}

\begin{proof}
	We define the functions
	\begin{align*}
		& N(s,\theta)=\bra{(R_0+\rho)^2 + 2 \pat \rho^2 - (R_0+\rho) \patt \rho}(s,\theta) ,\\
		& D(s,\theta)=\bra{((R_0+\rho)^2 + \pat \rho^2)^{3/2}}(s,\theta).
	\end{align*}
	Then, we can rewrite $\kappa(s,\theta)$ and $\pas\kappa(s,\theta)$ as
	\begin{equation*} 
		\kappa (s,\theta)=\frac{N(s,\theta)}{D(s,\theta)} \quad \textrm{ and } \quad 
		\pas \kappa (s,\theta) = \bra{ \frac{(\pas N) D -  N (\pas D)}{D^2}}
		(s,\theta)
	\end{equation*}
	with 
	\begin{align*}
		(\pas N)(s,\theta) &= \bra{2(R_0+\rho) \pas \rho + 4 \pat \rho \, \past \rho - R_0 \pastt \rho - \pas \rho \, \patt \rho - \rho \, \pastt \rho}(s,\theta), \\
		 (\pas D)(s,\theta) &= \bra{3 (\pat \rho \, \past \rho + (R_0+\rho) \pas \rho) \sqrt{(R_0 + \rho)^2 + (\pat \rho)^2}}(s,\theta).
	\end{align*}
	Since \eqref{eq:rho0} holds for all $\theta$, we see that 
\[
(\pas N)(0,\theta)= (\pas D)(0,\theta)=0,
\]
so
\[
 \pas \kappa (0,\theta)= 0.
\]

Next, we define
	\begin{equation*}
		n(s,\theta)=((\pas N) D -  N (\pas D))(s,\theta) \quad \mbox{and} \quad d(s,\theta)= D^2(s,\theta),
	\end{equation*}
	so that  
	\begin{align*} 
\pas \kappa (s,\theta) =\frac{n(s,\theta)}{d(s,\theta)}\q\t{and}\q
		\pass \kappa (s,\theta) =\bra{\frac{(\pas n) d -  n (\pas d)}{d^2}} (s,\theta).
	\end{align*}
	We compute
	\begin{equation*}
		(\pas n)(s,\theta)=\bra{(\pass N) D -  N (\pass D)}(s,\theta) \quad \mbox{and} \quad (\pas d)(s,\theta)= \bra{2 D \pas D}(s,\theta),
	\end{equation*}
	with
	\begin{align*}
		(\pass N)(s,\theta) = & \left[ 2\pas \rho^2 + (2R_0 + 2\rho - \patt \rho) \pass \rho +4 (\past \rho)^2\right. \\ \nonumber
		& \left.\q+ 4 \pat \rho \, \passt \rho - R_0 \passtt \rho - 2 \pas \rho \pastt \rho - \rho \passtt \rho\right] (s,\theta),\\
(\pass D)(s,\theta) & = \left[ 3 ((R_0 + \rho)^2 + (\pat\rho)^2)^{-1/2}((R_0+\rho)\pas \rho + \pat\rho \past \rho)^2 \right.\\ \nonumber
		& \left.\q+ 3 \sqrt{(R_0 + \rho)^2 + (\pat\rho)^2} ((\pas \rho)^2 + (R_0 + \rho)\pass \rho \right.\\ \nonumber
		&  \left.\q+ (\past \rho)^2 + \pat \rho \passt\rho) \right] (s,\theta).
	\end{align*}
	Using \eqref{eq:rho0}, we see that
\begin{align*}
(\pass N)(0,\theta) &=  2R_0 \pass \rho(0,\theta) -R_0 \passtt \rho(0,\theta)\\
(\pass D)(0,\theta) & = 3R_0^2 \pass \rho(0,\theta) .
\end{align*}
Then, since  
\[
N(0,\theta)=R_0^2,\q D(0,\theta) =R_0^3,\q (\pas D)(0,\theta)=0,
\]
we deduce
\begin{align*}
	\pas n(0,\theta)&=-R_0^4 \pass \rho(0,\theta)  -R_0^4\passtt \rho(0,\theta),\\
\pas d(0,\theta )&=0.
\end{align*}

We gather the above computations and we finally have that
\begin{eqnarray*}
\pass \kappa (0,\theta) &=&\bra{\frac{(\pas n) d -  n (\pas d)}{d^2}} (s,\theta)=  \frac{\pas n(0,\theta)}{D^2(0,\theta)}\\
&=&\frac{-R_0^4 \pass \rho(0,\theta)  -R_0^4\passtt \rho(0,\theta)}{R_0^6},
\end{eqnarray*}
 so  \eqref{eq:der_seconde_k} follows.
\end{proof}

\begin{rema}
Note that, if we also suppose that
\[
\pat^n\pass \rho(0,\theta)=0
\]
 for all $\theta \in [-\pi,\pi]$ and $n\ge 0$, then
\[
\pass \kappa (0,\theta)=0.
\]
\end{rema}

\begin{lemm}\label{lem:app11}
	Assume 
\begin{equation}\label{eq:rho00}
\pat^n\rho(0,\theta) =\pat^n \pas \rho(0,\theta)=\pat^n\pass \rho(0,\theta)=0 \qq\forall\theta \in [-\pi,\pi],\,\q \forall n\in \mathbb{N}.
\end{equation}
	Then
	$$ 		R_0^{2} \passs \kappa (0,\theta)=  - \passs \rho(0,\theta)  -\passstt \rho(0,\theta),
	$$
	hence
\[
\int_{-\pi}^\pi \passs \kappa (0,\theta) \cos \theta \dtheta=0.
\]
\end{lemm}

\begin{proof}
	Since we now assume that $\pass \rho(0,\theta)$ (and all its derivatives w.r.t. $\theta$) is zero, the computations in Lemma \ref{lem:app} give in particular
	$$\pass N (0,\theta)=0, \quad \pass D (0,\theta)=0.$$
	Next, we define 
	\begin{equation*}
		a(s,\theta)=((\pas n) d -  n (\pas d))(s,\theta) \quad \mbox{and} \quad b(s,\theta)=d^2(s,\theta).
	\end{equation*}
We thus write 	$\passs \kappa (s,\theta) $ as
	\begin{align} \label{eq:k3}
		\passs \kappa (s,\theta) =\bra{ \frac{(\pas a) b -  a (\pas b)}{b^2}} (s,\theta).
	\end{align}
	We compute 
	\begin{equation*}
		(\pas a)(s,\theta)=\bra{(\pass n) d -  n (\pass d)}(s,\theta) \quad \mbox{and} \quad (\pas b)(s,\theta)= 2[d \, \pas d](s,\theta),
	\end{equation*}
	where 
	\begin{align*}
		& (\pass n)(s,\theta) = ((\passs N)D - N(\passs D) + (\pass N)\pas D - 
		(\pas N)(\pass D))(s,\theta), \\
		& (\pass d)(s,\theta) = 2((\pas D)^2 + D \, \pass D)(s,\theta) ,
	\end{align*}
with
	\begin{align*}
		(\passs N)(s,\theta) & = (6 \pas \rho \, \pass \rho + (2R_0+2\rho-\patt \rho)\passs \rho + 12 \past \rho \, \passt \rho \\ \nonumber
		& \quad + 4 \pat \rho \passst \rho - (R_0 + \rho) \passstt \rho - 3 \pass \rho 
		\pastt \rho - 3 \pas \rho \passtt \rho)(s,\theta)
\\
		(\passs D)(s,\theta) = & -3 ((R_0+\rho)^2 + (\pat \rho)^2)^{-3/2}((R_0+\rho)\pas \rho + \pat \rho \past \rho)^3 \\ \nonumber
		& + 9 ((R_0+\rho)^2 + (\pat \rho)^2)^{-1/2} ((R_0+\rho) \pas \rho + \pat 
		\rho \past \rho)(\pat \rho \passt \rho + (\pas \rho)^2) \\ \nonumber
		& + 9 ((R_0+\rho)^2 + (\pat \rho)^2)^{-1/2} ((R_0+\rho) \pas \rho + \pat 
		\rho \past \rho)\left((R_0+\rho) \pass \rho + (\past \rho)^2 \right)\\ \nonumber
		& + 3 \sqrt{(R_0+\rho)^2 +(\pat \rho)^2)}(3 \pas \rho \pass \rho + (R_0 + \rho) \passs \rho) \\ \nonumber
		& + 3 \sqrt{(R_0+\rho)^2 +(\pat \rho)^2)}(3\past \rho \passt \rho + \pat 
		\rho \passst \rho)(s,\theta).
	\end{align*}
	Using \eqref{eq:rho00}, we get:
	\begin{align*}
		\passs N(0,\theta) &= 2R_0 \passs  \rho(0,\theta) -R_0 \passstt  \rho(0,\theta),\\
	\passs D(0,\theta)& = 3 R_0^2\passs \rho(0,\theta) .
	\end{align*}
Furthermore, since 
\[
N(0,\theta)=R_0^2,\q D(0,\theta) =R_0^3,\q\pas D(0,\theta)=\pass D(0,\theta )=0, 
\]
we deduce that
\begin{align*}
 (\pass n)(0,\theta) &= \bra{(\passs N)D - N(\passs D)}(0,\theta)= -R_0^4 \passs  \rho(0,\theta) -R_0^4 \passstt  \rho(0,\theta), \\
 (\pass d)(0,\theta) &= 0.
\end{align*}
Hence, recalling also that $\pas d(0,\theta)=0$, 
\begin{align*}
(\pas a)(0,\theta)&=\bra{(\pass n) d -  n (\pass d)}(0,\theta)=-R_0^{10} \passs  \rho(0,\theta) -R_0^{10} \passstt  \rho(0,\theta),
\\
 (\pas b)(0,\theta)&= 2(d \, \pas d)(0,\theta) =0,
\end{align*}
and \eqref{eq:k3}   becomes
\begin{align*}
\passs \kappa (0,\theta) =\frac{-R_0^{10} \passs  \rho(0,\theta) -R_0^{10} \passstt  \rho(0,\theta)}{R_0^{12}},
\end{align*}
	so the result follows.\\

In particular
\begin{align*}
R_0^2\int_{-\pi}^\pi \passs \kappa (0,\theta) \cos \theta \dtheta&=-R_0^2\int_{-\pi}^\pi\pare{ \passs  \rho(0,\theta) + \passstt  \rho(0,\theta)} \cos \theta \dtheta\\
&=-R_0^2\int_{-\pi}^\pi\passs  \rho(0,\theta) \cos \theta \dtheta+R_0^2\int_{-\pi}^\pi\passs  \rho(0,\theta)\cos \theta \dtheta\\
&=0.
\end{align*}
\end{proof}

\subsection{Computation of $\pas z$, $\pass z$ and $\passs z$}\label{app:z}

Recall that $\rho(s,\theta)$ and $\kappa(s,\theta)$ are defined by \eqref{eq:TW_sol}, \eqref{eq:TW_sol_0} and \eqref{def:k} for $s \in I$ and $\theta \in [-\pi,\pi]$.

\begin{lemm}\label{lem:app_2}
	Assume \eqref{eq:rho0}. 
	Then for all $\theta \in [-\pi, \pi]$ we have 
	$$
\int_{-\pi}^\pi \pas z(0,\theta)\dtheta=0. 
$$
	Furthermore, if we also have  \eqref{eq:rho00}, 
then
\begin{equation*} 
\int_{-\pi}^\pi  \pass z(0,\theta)\dtheta=\int_{-\pi}^\pi \cos^2\theta \pass z(0,\theta)\dtheta=\frac{a^2M^2}{2},
\end{equation*}
and
\[
\int_{-\pi}^\pi \cos\theta\passs z(0,\theta)\dtheta=-\frac{aM}{ R_0}V'''(0)-\frac{2M}{\pi R_0^3}\int_{-\pi}^\pi\cos\theta\passs\rho(0,\theta)\,d\theta. 
\]
\end{lemm}

\begin{proof} 

Recalling the definition of $z$ \eqref{def:z}, we see that
	\begin{align*} 
		\pas z(s,\theta) &= A(s,\theta) e^{-aV(s)(R_0+\rho(s,\theta))\cos\theta} , 	
	\end{align*}
	with 
	\begin{align*} 
		A(s,\theta)&=\pas c_1(V,\rho)-a\cos\theta c_1(V,\rho)\pas\pare{V(s)(R_0+\rho(s,\theta))}\\
&=V'(s)  \pa _V c_1(V,\rho) +\pas \rho(s,\theta)  \pa _\rho c_1(V,\rho)\\
&\q-c_1(V,\rho) a\left(V'(s) (R_0+\rho (s,\theta))+V(s) \pas \rho (s,\theta) \right) \cos \theta ,
	\end{align*}
	and $c_1$  given by \eqref{def:normalisation}. 
Since
\begin{equation}\label{eq:TW_sol_000}
 \pa _V c_1(0,0) =0,\q c_1(0,0)=\tilde{c},
\end{equation}
we deduce that
\[
A(0,\theta)=-a\tilde{c}R_0\cos\theta,
\]
hence
\begin{equation}\label{eq:pasz}
\pas z(0,\theta)= 
-a\tilde{c}R_0\cos \theta ,
\end{equation}
	and the first result follows.\\

We now assume \eqref{eq:rho00}.\\
	Differentiating again, we obtain
	\begin{equation*} 
		\pass z(s,\theta)=B(V,\rho) e^{-aV(s)(R_0+\rho(s,\theta))\cos\theta},  
	\end{equation*}
	with 
	\[
	B(s,\theta)= \pas A(s,\theta) - A(s,\theta) a\left(V'(s) (R_0+\rho (s,\theta))+V(s) \pas \rho (s,\theta) \right) \cos \theta ,
	\]
	and
	\begin{align*} 
		\pas A(s,\theta) &=  V''(s)  \pa _V c_1(V,\rho) + \left(V'(s)\right)^2  \pa _{VV} c_1(V,\rho) +\pass \rho(s,\theta)  \pa _\rho c_1(V,\rho)\\
		&+\left(\pas \rho(s,\theta)\right)^2  \pa _{\rho\rho} c_1(V,\rho)+2 V'(s) \pas \rho(s,\theta)   \pa _{V\rho} c_1(V,\rho) \nonumber\\
		& -\left(V'(s)  \pa _V c_1(V,\rho) +\pas \rho(s,\theta)  \pa _\rho c_1(V,\rho)\right)  a\left( V'(s) (R_0+\rho (s,\theta)) +V(s) \pas \rho (s,\theta)  \right) \cos \theta \nonumber\\
		&-c_1(V,\rho) a\left( V''(s) (R_0+\rho (s,\theta))+ 2V'(s) \pas \rho (s,\theta) +V(s) \pass \rho (s,\theta)\right) \cos \theta. \nonumber
	\end{align*}
Since we have that \eqref{eq:rho00} and \eqref{eq:TW_sol_000} hold, the above expression reduces to
	\begin{align}\label{der:z_2_bis}
		\pas A(0,\theta) &=  \pa _{VV} c_1(0,0) -\tilde{c}a R_0V''(0)  \cos \theta. \nonumber
	\end{align}
	Using also that 
\begin{equation}\label{eq:TW_sol_0000}
 \pa _{VV} c_1(0,0) = -\frac{a^2M}{4\pi} 
\end{equation}
	we obtain
\begin{align*} 
		\pas A(0,\theta) &=  - \frac{a^2M}{4\pi}   - a\tilde{c}V''(0) R_0 \cos \theta, 
		\end{align*}
	and then
\begin{align*}
B(0,\theta)
& =  - \frac{a^2M}{4\pi}   - a\tilde{c}V''(0) R_0 \cos \theta +a^2\tilde{c}R_0^2\cos^2\theta\nonumber\\
&=  - \frac{a^2M}{4\pi}   - \frac{aM}{\pi R_0}V''(0)  \cos \theta +a^2\frac{M}{\pi}\cos^2\theta\nonumber\\
&=- \frac{a^2M}{4\pi}\pare{1-4\cos^2\theta}   - \frac{aM}{\pi R_0}V''(0)  \cos \theta 
\end{align*}
thanks to the definition of $\tilde{c}$. \\
	Hence,
\begin{equation*} 
\pass z(0,\theta) =- \frac{a^2M}{4\pi}\pare{1-4\cos^2\theta}   - \frac{aM}{\pi R_0}V''(0)  \cos \theta.
\end{equation*}
Then
\begin{align*}
\int_{-\pi}^\pi \pass z(0,\theta)\,d\theta&=-\frac{a^2M}{4\pi}\int_{-\pi}^\pi \pare{1-4\cos^2\theta}  \,d\theta=\frac{a^2M}{2},\\
\int_{-\pi}^\pi \pass z(0,\theta)\cos^2 \theta\,d\theta&=-\frac{a^2M}{4\pi}\int_{-\pi}^\pi \pare{1-4\cos^2\theta}\cos^2\theta  \,d\theta=\frac{a^2M}{2}.
  \end{align*}

Finally, differentiating a third time, we see that 
	\begin{equation*} 
	\passs z(s,\theta)=C(s,\theta) e^{-aV(s)(R_0+\rho(s,\theta))\cos\theta},  
	\end{equation*}
with 
\[
C(s,\theta)= \pas B(s,\theta) - B(s,\theta) a\left(V'(s) (R_0+\rho (s,\theta))+V(s) \pas \rho (s,\theta) \right) \cos \theta.
\]
We compute
\begin{align*}
	\pas B(s,\theta) &= \pass A(s,\theta) -a\pas A(s,\theta) \left(V'(s) (R_0+\rho (s,\theta))+V(s) \pas \rho (s,\theta) \right) \cos \theta \\
	&\q-aA(s,\theta) \left( V''(s) (R_0+\rho (s,\theta)) + 2V'(s) \pas \rho(s,\theta) +V(s) \pass \rho (s,\theta)  \right) \cos \theta ,
\end{align*}
and
\begin{align*}
\pass A(s,\theta)&=\passs c_1(V,\rho)-\passs\pare{aV(s)(R_0+\rho(s,\theta))\cos\theta}c_1(V,\rho)\\
&\q-2\pass\pare{aV(s)(R_0+\rho(s,\theta))\cos\theta}\pas c_1(V,\rho)\\
&\q-\pas\pare{aV(s)(R_0+\rho(s,\theta))\cos\theta}\pass c_1(V,\rho).
\end{align*}
Equations   \eqref{eq:TW_sol_000}, \eqref{eq:TW_sol_0000}, and
\[
\pa_{VVV}c_1(0,0)=0 
\]
imply that
\begin{align*}
\pas c_1(0,0)&=0,
\\
\pass c_1(0,0)&=-\frac{a^2M}{4\pi},
\\
\passs c_1(0,0)&=-3\frac{a^2M}{4\pi}V''(0)-\frac{2M}{R_0^3\pi}\passs\rho(0,\theta).
\end{align*}
Then
\begin{align*}
\pass A(0,\theta)&=\passs c_1(0,0)-\tilde{c}aR_0\cos\theta V'''(0)-aR_0\cos\theta \pass c_1(0,0)\\
&=-3\frac{a^2M}{4\pi}V''(0)-\frac{2M}{R_0^3\pi}\passs\rho(0,\theta)-\frac{aM}{\pi R_0}\cos\theta V'''(0) +\frac{a^3R_0M}{4\pi}\cos\theta%\\
\end{align*}
and hence 
\begin{align*}
\pas B(0,\theta) &= 
-3\frac{a^2M}{4\pi}V''(0)-\frac{2M}{R_0^3\pi}\passs\rho(0,\theta)-\frac{aM}{\pi R_0}\cos\theta V'''(0) +\frac{a^3R_0M}{4\pi}\cos\theta\\
&\q +aR_0\cos\theta\pare{\frac{a^2M}{4\pi}   + a\tilde{c}V''(0) R_0 \cos \theta}\\
&\q +a^2\tilde{c}V''(0)R_0^2\cos^2\theta\\
&=-3\frac{a^2M}{4\pi}V''(0) -\frac{aM}{\pi R_0}\cos\theta V'''(0)+\frac{a^3R_0M}{2\pi}\cos\theta\\
& \q +2a^2\tilde{c}V''(0)R_0^2\cos^2\theta-\frac{2M}{\pi R_0^3}\passs\rho(0,\theta)
.
\end{align*}
Finally
\begin{align*}
C(0,\theta)&= \pas B(0,\theta) - B(0,\theta) aR_0\cos \theta\\
&=-3\frac{a^2M}{4\pi}V''(0) -\frac{aM}{\pi R_0}\cos\theta V'''(0)+\frac{a^3R_0M}{2\pi}\cos\theta\\
&\q +2a^2\tilde{c}V''(0)R_0^2\cos^2\theta-\frac{2M}{\pi R_0^3}\passs\rho(0,\theta)\\
&\q+aR_0\cos\theta\pare{\frac{a^2M}{4\pi}\pare{1-4\cos^2\theta}   + a\tilde{c}V''(0) R_0 \cos \theta}\\
&=3\frac{a^2M}{4\pi}V''(0)\pare{4\cos^2\theta-1}  -\frac{aM}{\pi R_0}\cos\theta V'''(0)+R_0\cos\theta\frac{a^3M}{4\pi}\pare{3-4\cos^2\theta}\\
&\q-\frac{2M}{\pi R_0^3}\passs\rho(0,\theta),
  \end{align*}
and
\begin{align*}
\passs z(0,\theta)&=3\frac{a^2M}{4\pi}V''(0)\pare{4\cos^2\theta-1}  -\frac{aM}{\pi R_0}\cos\theta V'''(0)+R_0\cos\theta\frac{a^3M}{4\pi}\pare{3-4\cos^2\theta}\nonumber\\
&\q-\frac{2M}{\pi R_0^3}\passs\rho(0,\theta), 
\end{align*}
Then
\begin{align*}
\int_{-\pi}^\pi \passs z(0,\theta)\cos \theta\,d\theta&=\int_{-\pi}^\pi C(0,\theta)\cos \theta\,d\theta =-\frac{aM}{ R_0}V'''(0) -\frac{2M}{\pi R_0^3}\int_{-\pi}^\pi\cos\theta\passs\rho(0,\theta)\,d\theta.
  \end{align*}

\end{proof}

\end{document}